\documentclass[reqno,english]{amsart}
\usepackage{amsmath,amssymb,amsthm,enumerate,esint}
\usepackage[svgnames]{xcolor}
  \usepackage{hyperref}

 \usepackage{mathrsfs}
\usepackage{mathabx}
\usepackage{empheq}
\usepackage{bbm}

\numberwithin{equation}{section}
\newtheorem{theorem}{Theorem}[section]
\newtheorem{corollary}[theorem]{Corollary}
\newtheorem{lemma}[theorem]{Lemma}
\newtheorem{proposition}[theorem]{Proposition}

\newtheorem{remark}[theorem]{Remark}

\newcommand{\bke}[1]{\left ( #1 \right )}
\newcommand{\bkt}[1]{\left [ #1 \right ]}
\newcommand{\bket}[1]{\left \{ #1 \right \}}
\newcommand{\norm}[1]{\left \| #1 \right \|}
\newcommand{\bka}[1]{{\langle #1 \rangle}}
\newcommand{\abs}[1]{\left | #1 \right |}

\newcommand\be{\beta}
\newcommand\ga{\gamma}
\newcommand\de{\delta}

\newcommand\ve{\varepsilon}
\newcommand\e {\varepsilon}

\renewcommand\th{\theta}
\newcommand\ka{\kappa}
\newcommand\la{\lambda}
\newcommand\si{\sigma}
\newcommand\om{\omega}

\newcommand\De{\Delta}
\newcommand\Th{\Theta}

\newcommand\Si{\Sigma}
\newcommand\Om{\Omega}

\newcommand{\R}{\mathbb{R}}
\newcommand{\CC}{\mathbb{C}}

\newcommand{\N}{\mathbb{N}}

\renewcommand{\Re} {\mathop{\mathrm{Re}}\nolimits}
\renewcommand{\Im} {\mathop{\mathrm{Im}}\nolimits}

\renewcommand{\div}{\mathop{\rm div}\nolimits}
\newcommand{\curl} {\mathop{\rm curl}\nolimits}

\newcommand{\pd}{\partial}
\newcommand{\nb}{\nabla}
\newcommand{\td}{\tilde}

\renewcommand{\bar}[1]{\overline{#1}}
\newcommand{\lec}{{\ \lesssim \ }}
\newcommand{\gec}{{\ \gtrsim \ }}

\newcommand{\EQ}[1]{\begin{equation}\begin{split} #1 \end{split}\end{equation}}
\newcommand{\EQN}[1]{\begin{equation*}\begin{split} #1 \end{split}\end{equation*}}

\usepackage{scalerel}

\begin{document}

\title[Qualitative Analysis of KS-NS systems]{\textbf{Global Existence and Aggregation of Chemotaxis-fluid Systems in dimension two}}

\author[F. Kong]{Fanze Kong}
\address{\noindent
Department of Mathematics,
University of British Columbia, Vancouver, B.C., V6T 1Z2, Canada}
\email{fzkong@math.ubc.ca}

\author[C. Lai]{Chen-Chih Lai}
\address{\noindent
Department of Mathematics,
Columbia University, 2990 Broadway, New York, NY 10027, USA}
\email{cl4205@columbia.edu}

\author[J. Wei]{Juncheng Wei}
\address{\noindent
Department of Mathematics,
University of British Columbia, Vancouver, B.C., V6T 1Z2, Canada}
\email{jcwei@math.ubc.ca}

% \author{Fanze Kong, ChenChih Lai And Juncheng Wei
     %   }

\date{}
\maketitle
 \vspace{-0.25in}
 \begin{abstract}
To describe the cellular self-aggregation phenomenon, some strongly coupled PDEs named as Keller-Segel (KS) and Patlak-Keller-Segel (PKS) systems were proposed in 1970s. 
Since KS and PKS systems possess relatively simple structures but admit rich dynamics, plenty of scholars have studied them and obtained many significant results. 
However, the cells or bacteria in general direct their movement in liquid. 
As a consequence, it seems more realistic to consider the influence of ambient fluid flow on the chemotactic mechanism. 

Motivated by this, He et al. (SIAM J. Math. Anal., Vol. 53, No. 3, 2021) proposed a coupled Patlak-Keller-Segel-Navier-Stokes system that features the effect of the friction induced by the cells on the ambient fluid flow.
In their pioneer work, the global existence of solutions of such system in dimension two was established when the initial mass is strictly less than a threshold, which is referred to as the subcritical case.
The last two authors and Zhou (Indiana Univ. Math. J., Vol. 72, No. 1, 2023) extended their result to the critical case.
To our best knowledge, this system has only been studied in either the whole space or periodic domains.

In this paper, we take into account the boundary effect and consider the initial-boundary value problem of the coupled Patlak-Keller-Segel-Navier-Stokes system in two-dimensional bounded domains.
The boundary conditions are Neumann conditions for the cell density and the chemical concentration, and the Navier slip boundary condition with zero friction for the fluid velocity.
We prove that the solution of the system exists globally in time in the case of subcritical mass. 
Concerning the critical mass case, we construct the boundary spot steady states rigorously via the inner-outer gluing method. 

While studying the global well-posedness and the concentration phenomenon of the chemotaxis-fluid model, we develop the global $W^{2,p}$ theory for the 2D stationary Stokes system subject to Navier boundary conditions and further establish semigroup estimates of the nonstationary counterpart by analyzing the Stokes eigenvalue problem. 
%some coupling of PKS models and the incompressible Navier-Stokes (NS) equations proposed by
%It is well-known that in 2D case, of concern the global well-posedness, PKS systems admit the critical mass phenomenon.
%, which are referred to as the chemotaxi-fluid systems 
%f concern the concentration phenomenon within two-dimensional chemotaxis-fluid models
%To achieve this, some authors couple the PKS models with the Navier--Stokes (NS) equations and start to investigate the qualitative behaviors of the solution.  In this paper, we consider the effect of fluid velocity on the cellular movement and focus on chemotaxis-fluid systems.  Of concern the concentration phenomenon within two-dimensional chemotaxis-fluid models, we rigorously construct the boundary spot steady states via the inner-outer gluing method.
 \end{abstract}
\section{Introduction}\label{sect0-intro}

In this paper, we consider the following chemotaxis--fluid model in 2D:
\begin{align}\label{PKSNS-time-dependent}
\left\{\,
\begin{array}{ll}
\frac{\partial n}{\partial t}= \De n - \nb\cdot(n\nb c) - {\bf u}\cdot\nb n,&x\in\Omega,t>0,\\
\iota\frac{\partial c}{\partial t}= \De c - c + n, &x\in\Omega,t>0,\\
\frac{{\partial {\bf u}}}{\partial t}+{\bf u}\cdot\nb {\bf u} + \nb P = \De {\bf u} + \e_0 n\nb c,\quad \nb\cdot {\bf u} = 0,&x\in\Omega,t>0,\\
 \frac{\pd n}{\pd\boldsymbol{\nu}} = \frac{\pd c}{\pd\boldsymbol{\nu}} = 0,&x\in\partial\Omega,t>0,\\
{\bf u}\cdot\boldsymbol{\nu}= 0,\quad (\mathbb S{\bf u}\cdot\boldsymbol{\nu})_{\boldsymbol{\tau}} = 0,&x\in\partial\Omega,t>0,\\
(n,c,{\bf u})(\cdot,0)=(n_0,c_0,{\bf u}_0),&x\in\Omega,
\end{array}
\right.
\end{align}
where $\Omega\subset\mathbb R^2$ is a bounded domain with the smooth boundary $\partial\Omega$, $\boldsymbol{\nu}$ denotes the unit outer normal vector, ${\mathbb S}\bf u = \frac12 (\nb \bf u + \nb \bf u^\top)$ represents the strain tensor and the subscript $\boldsymbol{\tau}$ is the tangential component.  Here $n$ and $c$ are the cellular density and the chemical concentration; $\bf u$ and $P$ denote the fluid velocity field and the pressure, respectively.  In addition, $\iota\geq 0$ is the time-reaction constant. The system (\ref{PKSNS-time-dependent}) is referred to as the PKS-NS system when $\iota=0$, in which case the dynamics in the interior of the domain coincides with the chemotaxis-fluid system proposed by Siming He et al. \cite{SimingHe}.
The novelty of the system (\ref{PKSNS-time-dependent}) lies in its consideration of the boundary effect.
The boundary conditions \eqref{PKSNS-time-dependent}$_5$ are the Navier boundary conditions with zero friction. 
Moreover, $(n_0,c_0,{\bf u}_0)$ is a given initial data with $\nabla\cdot {\bf u}_0=0$ for the compatibility consideration.   
%When $\iota=0,$ 
%in the 2D bounded domain of the chemotaxis-fluid model proposed by Siming He et al. \cite{SimingHe} to describe the cellular self-aggregation phenomenon in a moving fluid.  
The physical explanation of the forcing term $\e_0 n\nb c$ is that the cells are driven by the fluid to move without acceleration, in which $\e_0$ measures the strength of coupling between fluid and the evolution of cells.

The system (\ref{PKSNS-time-dependent}) can be naturally treated as the coupling of Keller-Segel (KS) models and incompressible Navier-Stokes (NS) equations.  Indeed, with the absence of the fluid advection term ${\bf u}\cdot \nabla n$, the $n$-equation and $c$-equation in (\ref{PKSNS-time-dependent}) consist of the minimal Keller-Segel model, which serves as a paradigm to describe the travelling band of {\it{E. coli}} \cite{Keller1970,Keller1971}.  Of concern the variants and applications of minimal Keller-Segel model, we refer the readers to well-written surveys \cite{hillen2009user,painter2019mathematical}.  It is because the Keller-Segel models have relatively simple structures but admit rich dynamics that plenty of researchers have extensively studied them in 1D and higher dimensions over the past few decades \cite{del2006collapsing,kang2007stability,lin1988large,wang2013spiky}.  Focusing on the global well-posedness of systems, Osaki et al. \cite{osaki2001finite,nagai1995} proved that the solution in 1D is uniformly bounded in time.  It is worthy mentioning that for the 2D case, one of the most famous phenomena is so-called ``chemotactic collapse".  To be more precise, there exists some critical threshold $M_0$ defined by
\begin{align}\label{sect0-M0-critical}
M_0=\left\{\begin{array}{ll}
8\pi,&\Omega=\mathbb R^2\text{~or~}B_R\text{~with the initial data being radial},\\
4\pi,&\text{otherwise}
\end{array}
\right.
\end{align}
such that when the initial cellular mass satisfies
$$M:=\int_{\Omega}u(x,0)dx<M_0,$$ 
the solution to (\ref{PKSNS-time-dependent}) is bounded uniformly in time $t$; otherwise if $M>M_0,$ the time-dependent system (\ref{PKSNS-time-dependent}) admits the blow-up solutions \cite{nanjundiah1973,childress1981,herrero1996,nagai1997application1,biler1998local,gajewski1998global,senba2000some,wang2002steady}.  In particular, the authors detected the finite blow-up phenomenon by studying the evolution of the cellular second moment \cite{blanchet2006two,collot2022refined,dolbeault2009two,herrero1996,herrero1996singularity,herrero1997blow}.  Focusing on the critical case $M=M_0$, on one hand, the solution to (\ref{PKSNS-time-dependent}) is shown to exist globally \cite{biler20068pi,velazquez2004point}; on the other hand, the infinite time blow-up solution with the finite second moment was constructed \cite{blanchet2008infinite,davila2020existence}.  For the incompressible Navier-Stokes equation, it is well-known that Leray \cite{leray1934mouvement} and Hopf \cite{hopf1950anfangswertaufgabe} established the existence of global weak solutions to the time-dependent incompressible Naiver-Stokes equation.  There are also many results on the study of global regularity in Navier-Stokes equations \cite{ben1994global,gallay2002invariant,gallagher2005uniqueness,gallay2005global,kato1984strong,kato1994navier,ogawa1997energy}.  For the global regularity of time-dependent and stationary Naiver-Stokes equations with Navier boundary conditions in 3D, we refer the readers to \cite{AS-DEA2011,AS-MMMAS2013,AR-JDE2014,AAR-semigroup2016,AAE-CM2016}.  Focusing on the PKS-NS system in the whole space $\mathbb R^2$, He et al. \cite{SimingHe} have shown the local and global existence of solutions in the Sobolev space $H^s$, $s\geq 2$ when the initial cellular mass is strictly less than $8\pi.$  Moreover, Lai et al. \cite{lai2021global} have proved the global existence of the free-energy solution when the initial mass is equal to $8\pi.$  We are motivated by the results to consider the global well-posedness of the chemotaxis-fluid system (\ref{PKSNS-time-dependent}) in 2D bounded domains.        % We would like to mention that with the absence of the ambient fluid flow, there exists the well-known phenomenon called ``chemotactic collapse" in (\ref{eq-1.1}) with $\Omega\subseteq R^2$.   

To further understand the dynamics of (\ref{PKSNS-time-dependent}), it is natural to consider the corresponding stationary problem of (\ref{PKSNS-time-dependent}), which is 
\begin{align}\label{PKSNS-ss}
\left\{\begin{array}{ll}
0= \De n - \nb\cdot(n\nb c) - {\bf u}\cdot\nb n, &x\in\Omega,\\
0= \De c - c + n, &x\in\Omega,\\
{\bf u}\cdot\nb {\bf u} + \nb P = \De {\bf u} + \e_0 n\nb c,\quad \nb\cdot {\bf u} = 0, &x\in\Omega,\\
\frac{\pd n}{\pd\boldsymbol{\nu}}= \frac{\pd c}{\pd\boldsymbol{\nu}} = 0,&x\in\partial\Omega,\\
{\bf u}\cdot\boldsymbol{\nu} = 0,\quad (\mathbb S{\bf u}\cdot\boldsymbol{\nu})_{\boldsymbol{\tau}} = 0, &x\in\partial\Omega.
\end{array}
\right.
\end{align}
With the absence of ambient fluid flow in (\ref{PKSNS-ss}), it is well-known that (\ref{PKSNS-ss}) admits the concentration phenomenon \cite{wang2019steadyreview}.  Indeed, of concern steady state problem (\ref{PKSNS-ss}) with the velocity fluid field $u$ being identically zero, Lin, Ni and Takagi \cite{lin1988large,ni1991shape,ni1993locating} firstly initiated the analytical approach to construct the large amplitude solution.  Motivated by this seminal work, many researchers studied the non-constant steady states possessing striking structures to Keller-Segel models \cite{gui1996multipeak,Gui1999,del2006collapsing,lin2007number,carrillo2021boundary}.  For example, del Pino and Wei \cite{del2006collapsing} constructed the multi-spike equilibrium to minimal Keller-Segel models in 2D via ``localized energy method".  Kang et al. \cite{kang2007stability} formally showed the existence of spikes in the asymptotically limit of domain size $L\gg 1.$  Moreover, its local stability was studied by Chen et al. \cite{Chen2014}.  It is worthy mentioning that Wang and Xu \cite{wang2013spiky} adopted an innovative method arising from bifurcation techniques to directly tackle the steady state problem without heavily using the structure of equations.  Whereas, concerning system (\ref{PKSNS-ss}), there is few result involving the construction of non-constant solution with excited structures.  Motivated by this, we shall construct the non-constant solutions, especially boundary and interior spikes, to the stationary chemotaxis-fluid system (\ref{PKSNS-ss}).

In summary, our two main aims of this paper are to show the existence of global-in-time solution to (\ref{PKSNS-time-dependent}) when the initial mass $M<M_0$ and study the concentration phenomenon with the critical mass.  For achieving the former one, the main vehicle is the following decreasing free energy functional possessed by (\ref{PKSNS-time-dependent}):
\begin{align}\label{sect0-free-energy-calculate}
\mathcal J(n,c,{\bf u}):=\int_{\Omega} n\log n\, dx+\frac{1}{2}\int_{\Omega} |{\bf u}|^2\, dx-\int_{\Omega}nc\, dx+\frac{1}{2}\int_{\Omega} c^2\, dx+\frac{1}{2}\int_{\Omega}\vert \nabla c\vert^2\, dx,
\end{align}
where the first term is the entropy of the cellular density $n$ and the second term represents the kinetic energy of the velocity field ${\bf u}.$  Here and below, for the consideration of local and global well-posedness, without loss of generality, we assume $\e_0=1$.  In fact, we prove the free energy (\ref{sect0-free-energy-calculate}) is dissipative along the dynamics, which is
\begin{lemma}\label{lemma1-free}
Let $(n,c,{\bf u})$ be the solution of (\ref{PKSNS-time-dependent}), then we have the energy functional $\mathcal J$ given by (\ref{sect0-free-energy-calculate}) satisfies the following energy dissipation:
\begin{align}\label{dissipation}
\frac{d}{dt}\mathcal J(t)=-\int_{\Omega} n|\nabla(\log n-c)|^2\, dx-\iota\int_{\Omega} \Big(\frac{\partial c}{\pd t}\Big)^2\, dx-\int_{\Omega} \vert \nabla {\bf u}\vert^2\, dx.
\end{align}
\end{lemma}
\begin{proof}
   Test $(\ref{PKSNS-time-dependent})_{1}$ against $\log n-c$, then we integrate the $n$-equation by parts to get
   \begin{align}\label{free-eq-1}
\int_{\Omega} \frac{\partial n}{\partial t}(\log n- c)\, dx=&\int_{\Omega} \nabla\cdot [n\nabla(\log n- c)](\log n- c)-\int_{\Omega} {\bf u}\cdot \nabla n (\log n-c)\, dx\nonumber\\
=&-\int_{\Omega} n|\nabla(\log n-c)|^2\, dx+\int_{\Omega} {\bf u}\cdot \nabla n\log n\, dx+\int_{\Omega} {\bf u}\cdot \nabla n c\, dx\nonumber\\
=&-\int_{\Omega} n|\nabla(\log n-c)|^2\, dx+\int_{\Omega} {\bf u}\cdot \nabla n c\,dx,
\end{align}
where we use the integration by parts to obtain from $\nabla \cdot {\bf u}=0$ that
\begin{align*}
\int_{\Omega} {\bf u}\cdot \nabla n\log n\, dx=\int_{\Omega}{\bf u}\cdot\nabla [n(\log n-1)]\, dx=0.
\end{align*}
In addition, one has from $\int_{\Omega}\frac{\partial n}{\partial t} \, dx =0$ and the $c$-equation that
\begin{align}\label{free-eq-2}
\int_{\Omega} \frac{\partial n}{\partial t}(\log n- c)\, dx=&\frac{d}{dt}\int_{\Omega} n\log n\, dx-\frac{d}{dt}\int_{\Omega} nc\, dx+\int_{\Omega} n\frac{\partial c}{\partial t}\,dx\nonumber\\
=&\frac{d}{dt}\int_{\Omega} n(\log n- c)\, dx+\int_{\Omega} (\iota c_t-\Delta c+c) c_t\,dx\nonumber\\
=&\frac{d}{dt}\int_{\Omega} n(\log n- c)\, dx+\iota\int_{\Omega} \Big(\frac{\pd c}{\partial t}\Big)^2\, dx+\frac{1}{2}\frac{d}{dt}\Big(\int_{\Omega} |\nabla c|^2+c^2\Big)\, dx.
\end{align}
Combining (\ref{free-eq-1}) and (\ref{free-eq-2}), we have
\begin{align}\label{free-eq-4}
&\frac{d}{dt}\int_{\Omega} n(\log n- c)\, dx+\iota\int_{\Omega} \Big(\frac{\pd c}{\pd t}\Big)^2\, dx+\frac{1}{2}\frac{d}{dt}(\int_{\Omega} |\nabla c|^2+c^2)\, dx\nonumber\\
=&-\int_{\Omega} n|\nabla(\log n-c)|^2\, dx+\int_{\Omega} {\bf u}\cdot \nabla n c\, dx.
\end{align}
We multiply $(\ref{PKSNS-time-dependent})_3$ by ${\bf u}$ and integrate it by parts to find from the divergence-free that
\begin{align}\label{free-eq-3}
    \frac{1}{2}\frac{d}{dt}\int_{\Omega} |{\bf u}|^2\, dx=&\int_{\Omega} {\bf u}\cdot[\Delta {\bf u}+n\nabla c-\nabla P-({\bf u}\cdot\nabla) {\bf u}]\, dx\nonumber\\
    =&-\int_{\Omega} |\nabla {\bf u}|^2\, dx-\int_{\Omega} {\bf u}\cdot \nabla nc\, dx-\int_{\Omega} {\bf u}\cdot \nabla P\, dx.
\end{align}
Upon summing (\ref{free-eq-4}) and (\ref{free-eq-3}), we obtain
\begin{align*}
&\frac{d}{dt}\int_{\Omega} n(\log n- c)\, dx+\frac{1}{2}\cdot\frac{d}{dt}\Big(\int_{\Omega} |\nabla c|^2+c^2\Big)\, dx+\frac{1}{2}\cdot\frac{d}{dt}\int_{\Omega} |{\bf u}|^2\, dx\\
=&-\iota\int_{\Omega} \Big(\frac{\pd c}{\pd t}\Big)^2\, dx-\int_{\Omega} n|\nabla(\log n-c)|^2\, dx+\int_{\Omega} {\bf u}\cdot \nabla n c\, dx,
\end{align*}
which proves this lemma.
\end{proof}
With the help of Lemma \ref{lemma1-free}, we plan to follow the idea shown in \cite{nagai1997application1} to prove system (\ref{PKSNS-time-dependent}) admits the global-in-time solution under the subcritical mass case.  To this end, we shall develop the semigroup theory of the non-stationary Stokes operator in 2D and establish the local-in-time existence of (\ref{PKSNS-time-dependent}).

For the latter aim, we will employ the inner-outer gluing method to construct the boundary spot steady state of (\ref{PKSNS-time-dependent}) with the critical mass.  The key observation is that the forcing term $n\nabla c$ in (\ref{PKSNS-time-dependent})$_3$ can be written as a stress tensor.  Indeed, recall that chemical concentration $c$ satisfies
$$\Delta c-c+n=0,~~~\quad \quad x\in\Omega,$$
then we have
\begin{align}\label{forcing-equivalent}
n\nabla c=-\Delta c\nabla c+c\nabla c=-\nabla\cdot(\nabla c\otimes\nabla c)+\nabla\bigg( \frac{\vert \nabla c\vert^2}{2}\bigg)+\nabla\Big(\frac{c^2}{2}\Big).
\end{align}
It follows from (\ref{forcing-equivalent}) that stationary problem (\ref{PKSNS-ss}) can be written as
\begin{align}\label{PKSNS-ss-equiv}
\left\{\begin{array}{ll}
0= \De n - \nb\cdot(n\nb c) - {\bf u}\cdot\nb n, &x\in\Omega,\\
0= \De c - c + n, &x\in\Omega,\\
{\bf u}\cdot\nb {\bf u} + \nb P = \De {\bf u} - \e_0\nabla \cdot(\nabla c\otimes\nabla c)+\e_0 \nabla\Big(\frac{|\nabla c|^2}{2}\Big)+\e_0\nabla\big(\frac{c^2}{2}\big),\quad \nb\cdot {\bf u} = 0, &x\in\Omega,\\
\frac{\pd n}{\pd\boldsymbol{\nu}}= \frac{\pd c}{\pd\boldsymbol{\nu}} = 0,&x\in\partial\Omega,\\
{\bf u}\cdot\boldsymbol{\nu} = 0,\quad (\mathbb S {\bf u}\cdot\boldsymbol{\nu})_{\boldsymbol{\tau}} = 0, &x\in\partial\Omega.
\end{array}
\right.
\end{align}
One formally sees that the potentials
$$\e_0 \nabla\Big(\frac{|\nabla c|^2}{2}\Big)+\e_0\nabla\Big(\frac{c^2}{2}\Big)$$
can be absorbed into the pressure term, and we shall focus on the equivalent form (\ref{PKSNS-ss-equiv}) then discuss the existence of spots in 2D.  We wish to mention that if the fluid velocity field satisfies ${\bf u}\equiv 0$ in (\ref{PKSNS-ss-equiv}) and $(n,c)$ is the solution to the stationary minimal Keller-Segel model, $c$ should satisfies the compatibility condition 
$$\nabla P= - \e_0\nabla \cdot(\nabla c\otimes\nabla c)+\e_0 \nabla\bigg(\frac{|\nabla c|^2}{2}\bigg)+\e_0\nabla\Big(\frac{c^2}{2}\Big)$$
for the scalar function $P.$  In addition, (\ref{PKSNS-ss}) has the approximate scaling invariance property.  More precisely, if $(n,c,{\bf u},P)(x)$ is a solution to (\ref{PKSNS-ss}), then $(n_{\lambda},c_{\lambda},{\bf u}_{\lambda},P_{\lambda})(x)=(\lambda^2n,c,\lambda {\bf u}, \lambda^2 P)(\lambda x)$ satisfies the following system: 
\begin{align*}
\left\{\begin{array}{ll}
0=\Delta n-\nabla \cdot (n\nabla c)-{\bf u}\cdot\nabla n,&x\in\Omega_{\lambda},\\
0=\Delta c-\lambda^2c+n,&x\in\Omega_{\lambda},\\
{\bf u}\cdot\nabla {\bf u}+\nabla P=\Delta {\bf u}+\e_0n\nabla c, &x\in\Omega_{\lambda},\\
\frac{\partial n}{\partial \boldsymbol{\nu}}=\frac{\partial c}{\partial \boldsymbol{\nu}}=0,&x\in\partial\Omega_{\lambda},\\
{\bf u}\cdot \boldsymbol{\nu}=0,\quad ({\mathbb S}{\bf u}\cdot \boldsymbol{\nu})_{\boldsymbol{\tau}}=0,&x\in\partial\Omega_{\lambda},
\end{array}
\right.
\end{align*}
where $\Omega_{\lambda}:=\Omega/\lambda.$  The scaling invariance property causes the fully coupling issue of the linearized system associated with (\ref{PKSNS-ss-equiv}), which forces us to impose the smallness assumption on $\e_0$. Detailed discussion is presented in Section \ref{inn-out-gluing-sect}. 
{ 
Since the velocity field ${\bf u}$ solves the inhomogeneous Stokes system \eqref{PKSNS-ss-equiv}$_3$ and satisfies the Navier boundary conditions \eqref{PKSNS-ss-equiv}$_5$, we have to develop a $W^{2,p}$ theory for the Stokes operator with the Navier boundary conditions in 2D bounded domains.
When the domain $\Om$ is the half space $\R^2_+$, the last author and collaborators established an explicit solution formula for the two-dimensional non-stationary Stokes system and gave pointwise estimates of the solution in \cite{lin2023nematic} by exploiting the Fourier-Laplace transform.
It is also mentioned in \cite{lin2023nematic} that with the aid of a suitable reflection (see \eqref{sect5-eq-solution-formula-stokes-divF}) one can reduce the half-space problem to the corresponding whole-space problem.
In this manner, the desired estimates of the outer solution can be easily derived by studying the Oseen tensor. 
Whereas, the reflection technique can only be applied to domains with flat boundary, and thus it can not be employed directly in the problem of the present paper because the domain under consideration is a bounded domain. 
To tackle this difficulty, we flatten the boundary and investigate its influence on the error. 
To estimate the error, we  
study the generalized inhomogeneous Stokes system of non-solenoidal velocity field with nontrivial right hand side in the Navier slip boundary conditions, \eqref{sect0-abstract-stokes-system}, by adopting the approach carried out in \cite{AR-JDE2014} of 3D Stokes system to give a delicate estimate near the boundary and establish a global $W^{2,p}$ regularity result for the 2D Stokes system.
See Section \ref{sect6-W2p-estimate-Stokes} for details.
}
%Moreover, due to the presence of velocity field ${\bf u}$ in (\ref{PKSNS-ss-equiv}) and Navier type boundary conditions, we have to develop the $W^{2,p}$ theory of the Stokes operator in 2D and the argument will be exhibited in Section \ref{sect6-W2p-estimate-Stokes}.  

Now, we state the main results of this paper.  Concerning the global regularity of the two-dimensional Stokes operator, we establish the following $W^{2,p}$ theory for the generalized inhomogeneous Stokes system of non-solenoidal velocity field with nontrivial right hand side in the Navier slip boundary conditions.
%have
%In addition, the scaling invariance causes the fully coupling of linearized operator  
% In addition, assume that $(n,c)$ is the non-trivial solution to the minimal Keller-Segel model, we find $(n,c,0)$ 
%\begin{equation}\label{eq-1.1ss}
%\left\{\,
%\begin{aligned}
%0&= \De n - \nb\cdot(n\nb c) - u\cdot\nb n,\\
%0&= \De c - c + n,\qquad\qquad\qquad\qquad\qquad\qquad\quad\, \frac{\pd n}{\pd\boldsymbol{\nu}}\Big|_{\pd\Om} = \frac{\pd c}{\pd\boldsymbol{\nu}}\Big|_{\pd\Om} = 0,\\
%u&\cdot\nb u + \nb p = \De u + \e_0 n\nb c,\quad \nb\cdot u = 0,\qquad  %u\cdot\nu\Big|_{\pd\Om} = 0,\quad (Su\cdot\nu)_\tau\Big|_{\pd\Om} = 0,
%\end{aligned}
%\right.
%\end{equation}
\begin{theorem}\label{thm-global-regularity-NS}
Consider the following abstract problem:
\begin{align}\label{sect0-abstract-stokes-system}
\left\{\begin{array}{ll}
-\De{\bf u} + \nb P = {\bf f}, &x\in\Om,\\
\div{\bf u}=\eta,&x\in \Om,\\
{\bf u}\cdot\boldsymbol{\nu} = g, &x\in\pd\Om,\\
2[\mathbb S({\bf u})\boldsymbol{\nu}]_{\boldsymbol{\tau}}= h\boldsymbol{\tau},&x\in\pd\Om,
\end{array}
\right.
\end{align}
where $\Om\subset\R^2$ is a bounded domain with $C^{1,1}$ boundary $\pd\Om$.  Let ${\bf f}\in{\bf L}^p(\Om)$, $\eta\in W^{1,p}(\Om)$, $g\in W^{2-\frac1p,p}(\pd\Om)$, and $h\in W^{1-\frac1p,p}(\pd\Om)$ satisfy the compatibility conditions \eqref{eq-2.14-AR} and \eqref{eq-3.15-AR}.
Then for $1<p<\infty$, system \eqref{sect0-abstract-stokes-system} has a unique solution $({\bf u},P)\in({\bf W}^{2,p}(\Om)\times W^{1,p}(\Om))/\boldsymbol{\mathcal N}(\Om)$.
Further, the solution satisfies the following estimate:
\EQN{
&\norm{\bf u}_{{\bf W}^{2,p}(\Om)/\boldsymbol{\mathcal T}(\Om)} + \norm{P}_{W^{1,p}(\Om)/\R}\\
\le & C\bke{ \norm{\bf f}_{{\bf L}^p(\Om)} +  \norm{\eta}_{W^{1,p}(\Om)} + \norm{g}_{W^{2-\frac1p,p}(\pd\Om)} + \norm{h}_{W^{1-\frac1p,p}(\pd\Om)} },
}
where $C>0$ is a constant, $\mathcal T$ and $\mathcal N$ are given by \eqref{sect6-kernel-Tp} and \eqref{sect6-N-kernel}, respectively.
\end{theorem}
Focusing on the global well-posedness of nonstationary problem (\ref{PKSNS-time-dependent}), we perform the energy estimate and apply the Moser-Alikakos iteration to obtain the following results:
\begin{theorem}\label{thm-global-existence-PKS-NS}
Assume that initial data $(n_0,c_0,{\bf u}_0)\in C^{0}(\bar\Omega)\times W^{1,\infty}(\Omega)\times {\bf{D}}(A_2)$ with ${\bf{D}}(A_2)$ defined by \eqref{sect2-semigroup-domain}, $n_0,c_0\geq 0,\not\equiv0$ and
\begin{align}\label{sect0-M-M0-less}
M:=\int_{\Omega}n_0\, dx<M_0,
\end{align}
where $M_0$ is defined by (\ref{sect0-M0-critical}).  Then (\ref{PKSNS-time-dependent}) admits the classical global-in-time solution $(n, c,{\bf u})$.
\end{theorem}
For Theorem \ref{thm-global-existence-PKS-NS}, we give some remarks as follows:
\begin{remark}
~

We have proved the global existence of the solution to (\ref{PKSNS-time-dependent}) with the subcritical mass.  For the critical mass case, our conjecture is that the solution also exists globally and the idea of proof may follow from \cite{nagai2016global} directly.  We believe at least for the parabolic-elliptic counterpart of (\ref{PKSNS-time-dependent}), their approach is durable with the slight modification. 
\end{remark}
Considering the concentration phenomenon, we assume $\e_0\ll 1$ is sufficiently small but fixed, then construct the solution with the striking structure to (\ref{PKSNS-ss-equiv}) via the inner-outer gluing method, which are summarized as
\begin{theorem}\label{thm11}
There exists a sufficiently small $\e_0>0$ such that for all sufficiently small $\e>0$, (\ref{PKSNS-ss-equiv}) admits a family of solutions $(n_{\varepsilon},c_{\varepsilon},u_{\varepsilon})$ satisfying the following forms:
\begin{align}\label{mainn}
 n_{\varepsilon}(x)=\frac{1}{\varepsilon^2}W\bigg(\frac{x - \xi_{\varepsilon}}{\varepsilon}\bigg) + o(1); 
\end{align}
\begin{align}\label{mainc}
  {c_{\varepsilon}}(x) =\big[\Gamma(x) +   H^{\varepsilon}(x, \xi_{\varepsilon})-4\log \varepsilon\big] + o(1),
 \end{align}
where $W$ and $\Gamma$ are defined by
\begin{align*}
W=\frac{8\mu^2}{(\mu^2+|y|^2)^2},~~\Gamma=\log\frac{8\mu^2}{(\mu^2+ \vert y\vert^2)^2},~~y=\frac{x-\xi_{\e}}{\e}, \ \ \xi_{\e}\in\pd\Omega;
\end{align*}
$H^{\varepsilon}$ is the correction term of $\Gamma,$ which satisfies
\begin{align*}
\left\{\begin{array}{ll}
-\Delta H^{\varepsilon}+H^{\varepsilon}=-\Gamma,&x\in\Omega,\\
\frac{\partial H^{\varepsilon}}{\partial\boldsymbol{\nu}}=-\frac{\partial\Gamma}{\partial \boldsymbol{\nu}},&x\in\partial\Omega.
\end{array}
\right.
\end{align*}
In particular, as $\e\rightarrow 0,$ $\xi_{\varepsilon}$ converges to the critical point of the following energy functional:
 \begin{align}\label{jm}
 \mathcal J_m(\xi)=16\pi^2 H(\xi,\xi),
 \end{align}
where the critical point of $\mathcal J_m$ is assumed to be non-degenerate and $H$ denotes the regular part of the following Neumann Green's function:
\begin{align*}
\left\{\begin{array}{ll}
\Delta G-G=-\delta (x-\xi),&x\in\Omega,\\
\frac{\partial G}{\partial \boldsymbol{\nu}}=0,&x\in\partial\Omega;
\end{array}
\right.
\end{align*}
$\mu$ is determined by 
 \begin{equation*}
   \log 8\mu^2 = 4\pi H(\xi, \xi).
 \end{equation*} 
\end{theorem}
We give the following remarks to explain our results shown in Theorem \ref{thm11}:
\begin{remark}
~
\begin{itemize}
    \item  Our subsequent proof of Theorem \ref{thm11} implies that $\e_0>0$ can be chosen as a universal constant independent of $\e$ given in the conditions and the domain size $|\Omega|.$
    \item Due to the approximate scaling invariance property of system (\ref{PKSNS-ss-equiv}), the linearized inner problem is fully coupled whenever $\e_0>0$ is sufficiently small.  In fact, the inner problem is given by
    \begin{align*}
    \left\{\begin{array}{ll}
    h_1+{\bf u}\cdot \nabla \phi=\Delta\phi-\nabla\cdot (\phi\nabla \Gamma)-\nabla \cdot(W\nabla \psi),\\
    h_2=\Delta\psi-\psi+\phi,\\
    \boldsymbol{h_3}+\nabla P=\Delta {\bf u}-\e_0\nabla\cdot(\nabla \Gamma \otimes \nabla \psi),
    \end{array}
    \right.
    \end{align*}
    where $h_1$, $h_2$ and $\boldsymbol{h_3}$ are generic error terms.
    \item The smallness of $\e_0$ is needed to guarantee the fixed point argument and the detailed discussion is shown in Section \ref{inn-out-gluing-sect}.
\end{itemize}
\end{remark}
The theoretical tool what we mainly use to show Theorem \ref{thm11} is the inner-outer gluing method, which is powerful and has been successfully applied on plenty of elliptic and parabolic problems \cite{del2016introduction,juncheng2022parabolic}.  We observe that system (\ref{PKSNS-ss-equiv}) can be naturally understood as the coupling of classical Keller-Segel models with transport effect and incompressible Navier-Stokes equations.  When $\e_0$ is sufficiently small, the construction of ansatz is in spirit of the pattern formation within the following minimal Keller-Segel models:
\begin{align}\label{standard-KS-SS}
\left\{\begin{array}{ll}
0=\Delta n-\nabla\cdot(n\nabla c),&x\in\Omega,\\
0=\Delta c-c+n,&x\in\Omega,\\
\frac{\partial n}{\partial \boldsymbol{\nu}}=\frac{\partial c}{\partial \boldsymbol{\nu}}=0,&x\in\partial\Omega.
\end{array}
\right.
\end{align}
In fact, del Pino and Wei \cite{del2006collapsing} showed the existence of spots to (\ref{standard-KS-SS}) rigorously via the Lyapunov-Schmidt reduction method.  Davila et al. \cite{davila2020existence} further constructed the infinite time blow-up solution to the non-stationary counterpart of (\ref{standard-KS-SS}) in the whole space $\R^2.$  We find that the vital step in the proof of Theorem \ref{thm11} is the formation of inner and outer linear theories, which crucially relies on the arguments shown in \cite{KWX2022}. 
 In \cite{KWX2022}, the authors borrowed the ideas from \cite{del2006collapsing,davila2020existence} and developed the linear theory successfully applied on the minimal Keller-Segel models with logistic growth.       

Although \cite{del2006collapsing}, \cite{davila2020existence} and \cite{KWX2022} provide many useful ideas that can be applied on the proof of Theorem \ref{thm11}, the fully coupling between transported Keller-Segel models and incompressible Navier-Stokes equations forces us to develop new ingredients in the inner-outer gluing procedure, which are shown as follows:
\begin{itemize}
    \item  Similarly as in \cite{lai2022finite}, the transport term ${\bf u}\cdot \nabla n$ in \eqref{PKSNS-ss-equiv}$_{1}$ cannot be regarded as a small perturbation term in the linearized inner problem.  Indeed, after scaling with inner variable $y=\frac{x-\xi}{\e}$, the linearized operator in the inner region becomes
    \begin{align}\label{sect1-linearized-problem-rmk}
    \left\{\begin{array}{ll}
    L_{W}[\Phi]-u[\Phi]\cdot\nabla_y W=h,\\
    -\Delta_y \bar\Psi=\Phi,
    \end{array}
    \right.
    \end{align}
    where $L_{W}[\Phi]$ is given by (\ref{sect2-linearized-inner-operator}).  The order of ${\bf u}[\Phi]\cdot\nabla_y W$ is precisely the same as the leading order term in error $h$.  To tackle with this issue, we adjust $\e_0>0$ in the forcing term $\e_0\nabla\cdot(\nabla c\otimes\nabla c)$ of (\ref{PKSNS-ss-equiv})$_{3}$ such that the smallness of ${\bf u}[\Phi]$ is provided.  As a consequence, the transport term ${\bf u}\cdot \nabla_y W$ can be truly realized as a perturbation term in the linearized inner problem (\ref{sect1-linearized-problem-rmk}).  
    \item To guarantee the construction of boundary spots, we impose Navier boundary conditions with zero friction rather than no-slip boundary conditions.  Motivated by this, we have to develop the new $W^{2,p}$ theory subject to Navier boundary conditions in 2D.
    \item Of concern the boundary spot, the location is assumed to be at the boundary rather than in the interior of domain.  As a consequence, we have to flatten the boundary and explore its influence on the error.  Moreover, we must develop the new inner linear theory restricted in the half space $\mathbb R^2_{+}$. 
    \item Unlike the parabolic stokes operator, the stationary counterpart has non-trivial kernels and some solvability conditions are needed to be imposed such that the existence and uniqueness of the solution are provided.  Thus, it is necessary to solve the corresponding reduced problems.  
\end{itemize}
%Thanks to the previous work of the authors  the formation of inner and outer linearies crucially relies on \cite{KWX2022}. to prove Theorem \ref{thm11}, we extend the inner and outer linear theories established in \cite{KWX2022} and develop the linear theory of Stokes operator subject to Navier boundary conditions with zero friction.0ur

The paper is organized as follows.  In Section \ref{sect6-W2p-estimate-Stokes}, we formulate the $W^{2,p}$ theory of the Stokes operator subject to Navier boundary conditions.  The section \ref{sect3-free-global-existence} is devoted to the global well-posedness of (\ref{PKSNS-time-dependent}) with the subcritical mass.  Section \ref{sect2}--\ref{inn-out-gluing-sect} focus on the construction of the boundary spot steady state with the mass exactly being $M_0$ given by (\ref{sect0-M0-critical}).  In detail, Section \ref{sect2} shows the idea of the choice of ansatz and the error computations.  Section \ref{sect3} is devoted to the effect of boundary on the error estimates.  Next, we establish the inner and outer linear theories modes by modes in Section \ref{sect4}.  Section \ref{sect-model-stokes} focuses on the model problem of Stokes operator.  In Section \ref{inn-out-gluing-sect}, we construct the boundary spot via the inner-outer gluing method and fixed point argument.

\medskip

Throughout the paper, we shall use the symbol ``$\lesssim$" to denote ``$\leq C$" for a positive constant $C$ independent of $x$ and $t$. Here $C$ might be different from line to line.  For convenience, we shall replace location $\xi_{\e}$ by $\xi$ without confusing readers in Section \ref{sect2}--\ref{inn-out-gluing-sect}.

\section{$W^{2,p}$ theory for Stokes system}\label{sect6-W2p-estimate-Stokes}
In this section, we consider the abstract problem (\ref{sect0-abstract-stokes-system}), which is the generalized inhomogeneous Stokes system of non-solenoidal velocity field with nontrivial right hand side in the Naiver slip boundary conditions.
%\footnote{\crr{suggestion: use boldface for vectors and blackboard-boldface for tensors in the whole paper}}
%\begin{subequations}
%\label{eq-stokes-general}
%\begin{align}
%-\De{\bf u} + \nb\pi& = {\bf f}\ \text{ in } \Om,\label{eq-stokes-general-a}\\
%\div{\bf u}&=\eta\ \text{ in } \Om,\label{eq-stokes-general-b}\\
%{\bf u}\cdot\boldsymbol{\nu} &= g\ \text{ on }\pd\Om,\label{eq-stokes-general-c}\\
%2[\mathbb S({\bf u})\boldsymbol{\nu}]_{\boldsymbol{\tau}}&= h\boldsymbol{\tau}\ \text{ on }\pd\Om,\label{eq-stokes-general-d}
%\end{align}
%\end{subequations}
Note that $\boldsymbol{\varphi}_{\boldsymbol{\tau}} := (\boldsymbol{\varphi}\cdot\boldsymbol{\tau})\boldsymbol{\tau}$ so that we may assume the right hand side of \eqref{sect0-abstract-stokes-system}$_4$ to be parallel to $\boldsymbol{\tau}$.  Before discussing the $W^{2,p}$ theory and prove Theorem \ref{thm-global-regularity-NS}, we give some preliminary notations and results.

\medskip

\subsection{Preliminaries}
~

Define 
\EQN{
{\bf H}^p(\div,\Om) = \bket{{\bf v}\in{\bf L}^p(\Om): \div{\bf v}\in L^p(\Om)},\quad
{\bf H}^p(\curl,\Om) = \bket{{\bf v}\in{\bf L}^p(\Om): \curl{\bf v}\in L^p(\Om)},
}
equipped with the norms
\EQN{
\norm{\bf v}_{{\bf H}^p(\div,\Om)} = \bke{\norm{\bf v}_{{\bf L}^p(\Om)}^p + \norm{\div{\bf v}}_{L^p(\Om)}^p}^{1/p},\quad
\norm{\bf v}_{{\bf H}^p(\curl,\Om)} = \bke{\norm{\bf v}_{{\bf L}^p(\Om)}^p + \norm{\curl{\bf v}}_{L^p(\Om)}^p}^{1/p}.
}
Thanks to \cite{AS-MMMAS2013}, the spaces ${\boldsymbol{ \mathcal D}}(\overline\Om)$ is dense in ${\bf H}^p(\div,\Om)$ and ${\bf H}^p(\curl,\Om)$.
Let the closures of ${\boldsymbol{ \mathcal D}}(\Om)$ in ${\bf H}^p(\div,\Om)$ and ${\bf H}^p(\curl,\Om)$ be ${\bf H}_0^p(\div,\Om)$ and ${\bf H}_0^p(\curl,\Om)$, respectively.
Then we have the characterization
\EQN{
{\bf H}_0^p(\div,\Om) &= \bket{{\bf v}\in{\bf H}^p(\div,\Om),\ {\bf v}\cdot\boldsymbol{\nu} = 0 \text{ on }\pd\Om},\\
{\bf H}_0^p(\curl,\Om) &= \bket{{\bf v}\in{\bf H}^p(\curl,\Om),\ {\bf v}\cdot\boldsymbol{\nu} = 0 \text{ on }\pd\Om}.
}
Define the space 
\[
{\bf X}^p(\Om) =  {\bf H}^p(\div,\Om) \cap {\bf H}^p(\curl,\Om)
\]
with the norm
\[
\norm{\bf v}_{{\bf X}^p(\Om)} = \bke{\norm{\bf v}_{{\bf L}^p(\Om)}^p + \norm{\div{\bf v}}_{L^p(\Om)}^p + \norm{\curl{\bf v}}_{L^p(\Om)}^p}^{1/p}.
\] 
Let 
\EQN{
{\bf X}^p_T(\Om) = &\bket{{\bf v}\in{\bf X}^p:{\bf v}\cdot\boldsymbol{\nu} = 0\text{ on }\pd\Om},\quad
{\bf X}^p_N(\Om) = \bket{{\bf v}\in{\bf X}^p:{\bf v}\cdot\boldsymbol{\tau} = 0\text{ on }\pd\Om},\\
&\qquad\qquad\qquad {\bf X}^p_0(\Om) = {\bf X}^p_T(\Om) \cap {\bf X}^p_N(\Om).
}

Recall the trace spaces, for $k\ge0$ and $p\in(1,\infty)$,
\[
{\bf W}^{k-\frac1p,p}(\pd\Om) = {\rm tr}({\bf W}^{k,p}(\Om))
= \bket{{\bf v}\in{\bf W}^{k-1,p}(\pd\Om):\exists{\bf w}\in{\bf W}^{k,p}(\Om)\text{ such that }{\rm tr}({\bf w}) = {\bf v}},
\]
and ${\bf H}^{\frac12}(\pd\Om) := {\bf W}^{1-\frac12,2}(\pd\Om)$.

For $p\in(1,\infty)$, denote its dual exponent by $p'$, i.e. $1/p + 1/p' = 1$.
Let $[{\bf H}^{p'}_0(\div,\Om)]'$ be the dual space of ${\bf H}^{p'}_0(\div,\Om)$ with the pairing 
\EQ{\label{eq-Hp-pairing}
\bka{\cdot,\cdot}_{\Om,p} := \bka{\cdot,\cdot}_{[{\bf H}^{p'}_0(\div,\Om)]'\times{\bf H}_0^{p'}(\div,\Om)},
}
where we use the parameter $p$, instead of $p'$, on the left hand side for notational convenience.
Let ${\bf W}^{-\frac1p,p}(\pd\Om)$ denote the dual space of ${\bf W}^{\frac1p,p'}(\pd\Om) = {\bf W}^{1-\frac1{p'},p'}(\pd\Om)$ with the pairing
\EQ{\label{eq-Wp-pairing}
\bka{\cdot,\cdot}_{\pd\Om,p} := \bka{\cdot,\cdot}_{{\bf W}^{-\frac1p,p}(\pd\Om)\times{\bf W}^{\frac1p,p'}(\pd\Om)}.
}

Define
\[
{\bf V}^p(\Om) = \bket{{\bf v}\in{\bf W}^{1,p}(\Om): \div{\bf v} = 0 \text{ in }\Om,\ {\bf v}\cdot\boldsymbol{\nu} = 0 \text{ on }\pd\Om}
\]
equipped with the ${\bf W}^{1,p}(\Om)$-norm and
\[
{\bf E}^p(\Om) = \bket{{\bf v}\in{\bf W}^{1,p}(\Om): \De{\bf v} \in [{\bf H}^{p'}_0(\div,\Om)]'},
\]
with the norm
\[
\norm{\bf v}_{{\bf E}^p(\Om)} = \norm{\bf v}_{{\bf W}^{1,p}(\Om)} + \norm{\De{\bf v}}_{[{\bf H}^{p'}_0(\div,\Om)]'}.
\]
By the same argument as in the proof of \cite[Lemma 4.2.1]{Seloula-thesis2010}, ${\boldsymbol{ \mathcal D}}(\overline\Om)$ is dense in ${\bf E}^p(\Om)$.

Introduce the kernel spaces
\EQN{
{\bf K}^p_T(\Om) &= \bket{{\bf v}\in{\bf L}^p(\Om): \div{\bf v} = \curl{\bf v} = 0\text{ in }\Om\text{ and }{\bf v}\cdot\boldsymbol{\nu} = 0\text{ on }\pd\Om},\\
{\bf K}^p_N(\Om) &= \bket{{\bf v}\in{\bf L}^p(\Om): \div{\bf v} = \curl{\bf v} = 0\text{ in }\Om\text{ and }{\bf v}\cdot\boldsymbol{\tau} = 0\text{ on }\pd\Om}.
}

Now we state the key identity in our analysis on the boundary.

\begin{lemma}\label{lem-2.1-CMR}
For any ${\bf v}\in{\bf W}^{2,p}(\Om)$ with ${\bf v}\cdot\boldsymbol{\nu} = 0$ on $\pd\Om$, we have 
\[
2[\mathbb S({\bf v})\boldsymbol{\nu}]_{\boldsymbol{\tau}} = (\curl{\bf v})\boldsymbol{\tau} - 2\ka{\bf v}_{\boldsymbol{\tau}}\ \text{ on } \pd\Om,
\]
where $\ka$ is the curvature on $\pd\Om$.
\end{lemma}
\begin{proof}
The lemma follows directly from the proof of \cite[Lemma 2.1]{CMR-Nonlinearity1998} using the density of the space $\boldsymbol{\mathcal D}(\overline\Om)$ in $\bket{{\bf v}\in{\bf W}^{2,p}(\Om): {\bf v}\cdot\boldsymbol{\nu} = 0\text{ on }\pd\Om}$.
\end{proof}
Next, we establish the Green identity of Stokes system.

\medskip

\subsection{Green formulas}
~

We first derive two useful Green formulas in the following two lemmas.

\begin{lemma}\label{lem-2.3-AR}
Let $\Om\subset\R^2$ be a bounded domain with $C^{1,1}$ boundary $\pd\Om$ and $1<p<\infty$.
The linear mapping $\ga:{\bf v}\mapsto\curl{\bf v}|_{\pd\Om}$ defined on ${\boldsymbol{ \mathcal D}}(\overline\Om)$ can be extended to a linear and continuous mapping 
\[
\ga: {\bf E}^p(\Om) \to W^{-\frac1p,p}(\pd\Om).
\]
Moreover, we have the Green formula: For any ${\bf v}\in{\bf E}^p(\Om)$, $\boldsymbol{\varphi}\in{\bf V}^{p'}(\Om)$,
\EQ{\label{eq-2.7-AR}
-\bka{\De{\bf v}, \boldsymbol{\varphi}}_{\Om,p} = \int_\Om (\curl{\bf v})(\curl\boldsymbol{\varphi})\, dx - \bka{(\curl{\bf v})\boldsymbol{\tau},\boldsymbol{\varphi}}_{\pd\Om,p},
}
where $\bka{\cdot,\cdot}_{\Om,p}$ and $\bka{\cdot,\cdot}_{\pd\Om,p}$ are the pairings defined in \eqref{eq-Hp-pairing} and \eqref{eq-Wp-pairing}.
\end{lemma}

\begin{proof}
By the density argument as in the proof of \cite[Corollary 4.2.2]{Seloula-thesis2010} we may assume ${\bf v}\in{\boldsymbol{ \mathcal D}}(\overline\Om)$ and $\boldsymbol{\varphi}\in{\bf W}^{1,p'}(\Om)\cap{\bf X}^{p'}_T(\Om)$.
By the integration by parts formula and using $\div\boldsymbol{\varphi}=0$ in $\Om$ and $\boldsymbol{\varphi}\cdot\boldsymbol{\nu}=0$ on $\pd\Om$, we gets
\[
\int_\Om (\curl{\bf v})(\curl{\boldsymbol{\varphi}})\, dx = - \int_\Om \De{\bf v}\cdot\boldsymbol{\varphi}\, dx + \int_{\pd\Om} (\curl{\bf v})(\varphi_2 \nu_1 - \varphi_1 \nu_2)\, dS,\quad \boldsymbol{\varphi} = (\varphi_1,\varphi_2),\ \boldsymbol{\nu} = (\nu_1,\nu_2).
\]
Since for $\boldsymbol{\tau} := (\tau_1,\tau_2) = (-\nu_2,\nu_1)$, one has
$\varphi_2 \nu_1 - \varphi_1 \nu_2 = \varphi_2 \tau_2 + \varphi_1 \nu_1 = \boldsymbol{\varphi}\cdot\boldsymbol{\tau}$.
Thus, we duduce
\[
\int_\Om (\curl{\bf v})(\curl{\boldsymbol{\varphi}})\, dx = - \int_\Om \De{\bf v}\cdot\boldsymbol{\varphi}\, dx + \int_{\pd\Om} (\curl{\bf v})\boldsymbol{\tau}\cdot\boldsymbol{\varphi}\, dS.
\]
Rearranging the above equation, the lemma is proved.
\end{proof}

\begin{lemma}\label{lem-2.4-AR}
Let $\Om\subset\R^2$ be a bounded domain with $C^{1,1}$ boundary $\pd\Om$ and $1<p<\infty$.
The linear mapping $\Th:{\bf v}\mapsto[\mathbb S({\bf v})\boldsymbol{\nu}]_{\boldsymbol{\tau}}|_{\pd\Om}$ defined on ${\boldsymbol{ \mathcal D}}(\overline\Om)$ can be extended to a linear and continuous mapping 
\[
\Th: {\bf E}^p(\Om) \to W^{-\frac1p,p}(\pd\Om).
\]
Moreover, we have the Green formula: For any ${\bf v}\in{\bf E}^p(\Om)$, $\boldsymbol{\varphi}\in{\bf V}^{p'}(\Om)$,
\EQ{\label{eq-2.8-AR}
-\bka{\De{\bf v}, \boldsymbol{\varphi}}_{\Om,p} = 2\int_\Om \mathbb S({\bf v}): \mathbb S(\boldsymbol{\varphi})\, dx - \bka{2[\mathbb S({\bf v})\boldsymbol{\nu}]_{\boldsymbol{\tau}},\boldsymbol{\varphi}}_{\pd\Om,p},
}
where $\bka{\cdot,\cdot}_{\Om,p}$ and $\bka{\cdot,\cdot}_{\pd\Om,p}$ are the pairings defined in \eqref{eq-Hp-pairing} and \eqref{eq-Wp-pairing}.
\end{lemma}
\begin{proof}
The lemma follows straightforwardly from the proof of \cite[Lemma 2.4]{AR-JDE2014} using the identity
\[
\De{\bf v} = 2\div\mathbb S({\bf v}) - \nb(\div{\bf v}).
\]
\end{proof}

It follows from above that one can extend Lemma \ref{lem-2.1-CMR} to a statement in ${\bf W}^{-\frac1p,p}(\pd\Om)$. 
Indeed, we have the following corollary.

\begin{corollary}\label{cor-2.5-AR}
For any ${\bf v}\in{\bf E}^p(\Om)$ with ${\bf v}\cdot\boldsymbol{\nu} = 0$ on $\pd\Om$, we have
\[
2[\mathbb S({\bf v})\boldsymbol{\nu}]_{\boldsymbol{\tau}} = (\curl{\bf v})\boldsymbol{\tau} - 2\ka{\bf v}_{\boldsymbol{\tau}}\ \text{ in } {\bf W}^{-\frac1p,p}(\pd\Om),
\]
where $\ka$ is the curvature on $\pd\Om$.
\end{corollary}

\medskip

\subsection{An auxiliary problem}
~

In view of Corollary \ref{cor-2.5-AR}, we consider the following auxiliary problem:
\begin{subequations}
\label{eq-stokes-auxiliary}
\begin{align}
-\De{\bf u} + \nb P& = {\bf f}\ \text{ in } \Om,\label{eq-stokes-auxiliary-a}\\
\div{\bf u}&=\eta\ \text{ in } \Om,\label{eq-stokes-auxiliary-b}\\
{\bf u}\cdot\boldsymbol{\nu} &= g\ \text{ on }\pd\Om,\label{eq-stokes-auxiliary-c}\\
\curl{\bf u}&= H\ \text{ on }\pd\Om.\label{eq-stokes-auxiliary-d}
\end{align}
\end{subequations}
According to Lemma \ref{lem-2.1-CMR}, the auxiliary problem \eqref{eq-stokes-auxiliary} is equivalent to the generalized Stokes system \eqref{sect0-abstract-stokes-system}
when $H = h + 2\ka({\bf u}\cdot{\boldsymbol{\tau}})$, where $\ka$ represents the curvature on $\pd\Om$.

We establish the existence and uniqueness of weak solution for the auxiliary problem \eqref{eq-stokes-auxiliary} in the following theorem.

\begin{theorem}\label{thm-2.7-AR}
Let $\Om\subset\R^2$ be a bounded domain with $C^{1,1}$ boundary $\pd\Om$ and $1<p<\infty$.
Suppose ${\bf f}\in[{\bf H}_0^{p'}(\div,\Om)]'$, $\eta\in L^p(\Om)$, $g\in W^{1-\frac1p,p}(\pd\Om)$, $H\in W^{-\frac1p,p}(\pd\Om)$ satisfy the following compatibility conditions:

for any $\boldsymbol{\varphi}\in{\bf K}^{p'}_T(\Om)$,
\EQ{\label{eq-2.13-AR}
\bka{{\bf f},\boldsymbol{\varphi}}_{\Om,p} + \bka{H\boldsymbol{\tau},\boldsymbol{\varphi}}_{\pd\Om,p} = 0,
}
\EQ{\label{eq-2.14-AR}
\int_\Om \eta\, dx = \int_{\pd\Om} g\, dS.
}
Then the auxiliary problem \eqref{eq-stokes-auxiliary} has a unique solution $({\bf u},P)\in{\bf W}^{1,p}(\Om)\times (L^p(\Om)/\R)$.
Moreover, the solution satisfies the estimate
\[
\norm{\bf u}_{{\bf W}^{1,p}(\Om)} + \norm{P}_{L^p(\Om)/\R}
\lesssim\bke{ \norm{\bf f}_{[{\bf H}^{p'}_0(\div,\Om)]'} + \norm{\eta}_{L^p(\Om)} + \norm{g}_{W^{1-\frac1p,p}(\pd\Om)} + \norm{H}_{W^{-\frac1p,p}(\Om)} }.
\]
\end{theorem}
\begin{proof}
Decompose ${\bf u} = {\bf z} + \nb\th$, where
\begin{subequations}
\label{eq-stokes-auxiliary-z}
\begin{align}
-\De{\bf z} + \nb P& = {\bf f} + \nb\eta\ \text{ in } \Om,\label{eq-stokes-auxiliary-z-a}\\
\div{\bf z}&=0\ \text{ in } \Om,\label{eq-stokes-auxiliary-z-b}\\
{\bf z}\cdot\boldsymbol{\nu} &= 0\ \text{ on }\pd\Om,\label{eq-stokes-auxiliary-z-c}\\
\curl{\bf z}&= H\ \text{ on }\pd\Om,\label{eq-stokes-auxiliary-z-d}
\end{align}
\end{subequations}
and 
\begin{subequations}
\label{eq-stokes-auxiliary-theta}
\begin{align}
\De\th& = \eta\ \text{ in } \Om,\label{eq-stokes-auxiliary-theta-a}\\
\nb\th\cdot\boldsymbol{\nu} &= g\ \text{ on }\pd\Om.\label{eq-stokes-auxiliary-theta-b}
\end{align}
\end{subequations}
By the $W^{2,p}$ theory of Neumann problem of Poisson's equation, there exists a unique solution $\th\in W^{2,p}(\Om)/\R$ of \eqref{eq-stokes-auxiliary-theta} with
\[
\norm{\th}_{W^{2,p}(\Om)/\R} \lesssim \bke{ \norm{\eta}_{L^p(\Om)} + \norm{g}_{W^{1-\frac1p,p}(\pd\Om)} }.
\]
For the system \eqref{eq-stokes-auxiliary-z}, by Lemma \ref{lem-2.3-AR} every solution of \eqref{eq-stokes-auxiliary-z} also solves
\[
\int_\Om (\curl{\bf z})(\curl\boldsymbol{\varphi})\, dx = \bka{{\bf f},\boldsymbol{\varphi}}_{\Om,p} + \bka{H\boldsymbol{\tau},\boldsymbol{\varphi}}_{\pd\Om,p},\ \text{ for all }\boldsymbol{\varphi}\in{\bf V}^{p'}(\Om).
\]
The rest of the proof follows from the same argument in the proof of \cite[Theorem 4.2.4]{Seloula-thesis2010} (see also \cite[Theorem 4.4]{AS-DEA2011}).
\end{proof}

\medskip

\subsection{Weak solutions of the generalized Stokes system \eqref{sect0-abstract-stokes-system}}
~

We consider the inhomogeneous Stokes system of non-solenoidal velocity field with nontrivial right hand side in the slip-Naiver boundary condition \eqref{sect0-abstract-stokes-system}.

\medskip

\noindent{\bf The case of $\eta=0$ and $g=0$.}
The following proposition provides a weak formulation of the generalized Stokes system \eqref{sect0-abstract-stokes-system} for the case of $\eta=0$ and $g=0$. 

\begin{proposition}\label{prop-3.1-AR}
Suppose $\eta=0$ and $g=0$. 
Let ${\bf f}\in[{\bf H}^{p'}_0(\div,\Om)]'$ and $h\in W^{-\frac1p,p}(\pd\Om)$.
Then the problem of finding a distributional solution $({\bf u},P)\in{\bf W}^{1,p}(\Om)\times L^p(\Om)$ of the generalized Stokes system \eqref{sect0-abstract-stokes-system} is equivalent to the problem of finding ${\bf u}\in{\bf V}^p(\Om)$ such that
\EQ{\label{eq-3.2-AR}
2\int_\Om \mathbb S({\bf u}): \mathbb S(\boldsymbol{\varphi})\, dx = \bka{{\bf f},\boldsymbol{\varphi}}_{\Om,p} + \bka{h\boldsymbol{\tau},\boldsymbol{\varphi}}_{\pd\Om,p},\ \text{ for all }\boldsymbol{\varphi}\in{\bf V}^{p'}(\Om).
}
\end{proposition}

\begin{proof}
It is a direct consequence of the Green formula \eqref{eq-2.8-AR} in Lemma \ref{lem-2.4-AR} via the same proof of \cite[Proposition 3.1]{AR-JDE2014}.
\end{proof}

We now introduce the kernel $\boldsymbol{\mathcal T}(\Om)$. 
Define
\begin{align}\label{sect6-kernel-Tp}
\boldsymbol{\mathcal T}^p(\Om) = \bket{{\bf v}\in{\bf W}^{1,p}(\Om): \mathbb S({\bf v}) = {\bf O}\text{ in }\Om,\text{ and }{\bf v}\cdot\boldsymbol{\nu}=0\text{ on }\pd\Om}.
\end{align}
The following result characterizes the kernel $\boldsymbol{\mathcal T}^p(\Om) $.
\begin{lemma}
\[
\boldsymbol{\mathcal T}^p(\Om) =
\begin{cases}
{\rm span}\bket{\boldsymbol{\be}},\ \boldsymbol{\be} = c{\bf x}^\perp + {\bf b},\ \text{for some constant $c\neq0$ and ${\bf b}$,}&\ \text{if $\Om$ is a disk},\\
\{{\bf 0}\},&\ \text{otherwise},
\end{cases}
\]
where ${\bf x}^\perp := (-x_2,x_1)$.
In particular, $\boldsymbol{\mathcal T}^p(\Om)$ does not depend on $p$ so that we can denote it by $\boldsymbol{\mathcal T}(\Om)$.
\end{lemma}
\begin{proof}
For ${\bf v} = (v_1,v_2)\in\boldsymbol{\mathcal T}^p(\Om)$, one has $\pd_1v_1 = 0$, $\pd_2v_2=0$, and $\pd_1v_2 = -\pd_2v_1$.
Integrating the first two equations, we deduce $v_1(x_1,x_2) = c_1(x_2)$ and $v_2(x_1,x_2) = c_2(x_1)$.
It then follows from $\pd_1v_2 = -\pd_2v_1$ that $\pd_1c_2(x_1) = -\pd_2c_1(x_2) = c$ for some constant $c$.
Thus, $c_1(x_2) = cx_1 + b_2$ and $c_2(x_1) = -cx_2 + b_1$ for some constants $b_1$, $b_2$.
Therefore, 
\[
{\bf v}({\bf x}) = c{\bf x}^\perp + {\bf b},\quad {\bf x}^\perp=(-x_2,x_1),
\]
for some constant scalar $c$ and vector ${\bf b} = (b_1,b_2)$.

Suppose $c\neq0$. By the change of variables of translation $(y_1,y_2) = \boldsymbol{\psi}({\bf x}) := (x_1 + (b_2/c), x_2 - (b_1/c))$ we have that ${\bf v}({\bf y}) = c(-y_2,y_1)$.
Since the normal vector is invariant under translation, we have $\boldsymbol{\nu} = \boldsymbol{\nu}_{\bf y}$ where $\boldsymbol{\nu}_{\bf y}$ is the outward normal of $\Om_{\bf y} = \boldsymbol{\psi}(\Om)$.
Then ${\bf v}({\bf x})\cdot\boldsymbol{\nu} = 0$ on $\pd\Om$ implies that $0 = {\bf v}({\bf y})\cdot\boldsymbol{\nu}_{\bf y} = c(-y_2({\nu}_{\bf y})_1 + y_1({\nu}_{\bf y})_2)$ on $\pd\Om_{\bf y}$.
This means that ${\nu}_{\bf y}$ is parallel to $(y_1,y_2)$, implying that $\Om_{\bf y}$, hence $\Om$ is a disk.

If $c=0$, then ${\bf v} = {\bf b}$. But since ${\bf v}({\bf x})\cdot\boldsymbol{\nu} = 0$ on $\pd\Om$, we must have ${\bf b} = {\bf 0}$. 
This completes the proof of the lemma. 
\end{proof}

We also denote
\begin{align}\label{sect6-N-kernel}
\boldsymbol{\mathcal N}(\Om) := \boldsymbol{\mathcal T}(\Om)\times\R 
= \bket{({\bf v},c): {\bf v}\in\boldsymbol{\mathcal T}(\Om),\ c\in\R }. 
\end{align}

\begin{remark}
Since $\div\boldsymbol{\be} = 0$ and $\boldsymbol{\be}\cdot\boldsymbol{\nu} = 0$ on $\pd\Om$,
$\boldsymbol{\be}\in{\bf V}^{p'}(\Om)$.
By choosing $\boldsymbol{\varphi} = \boldsymbol{\be}$ in \eqref{eq-3.2-AR}, we deduce the following compatibility condition:
\EQ{\label{eq-3.15-AR}
\bka{{\bf f}, \boldsymbol{\be}}_{\Om,p} + \bka{h\boldsymbol{\tau},\boldsymbol{\be}}_{\pd\Om,p} = 0
}
for solving the generalized Stokes system \eqref{sect0-abstract-stokes-system} when $\eta=0$ and $g=0$.
\end{remark}

The following lemma is stated in \cite[(2.6)]{Verfurth-NumerMath1987}) for general dimension without proof.

\begin{lemma}[Poincar\'e-Morrey inequality]\label{lem-3.3-AR}
Let $\Om$ be a Lipschitz bounded domain in $\R^2$.
Then, we have
\EQ{\label{eq-3.7-AR}
\inf_{{\bf v}\in\boldsymbol{\mathcal T}(\Om)} \norm{{\bf u} + {\bf v}}_{{\bf L}^2(\Om)}^2 \le C \bke{ \norm{\mathbb S({\bf u})}_{{\bf L}^2(\Om)}^2 + \int_{\pd\Om} |{\bf u}\cdot\boldsymbol{\nu}|^2\, dS },\ \text{ for all }{\bf u}\in{\bf H}^1(\Om),
}
where the constant $C$ depends only on $\Om$.
In particular, the seminorm $\norm{\mathbb S({\bf u})}_{{\bf L}^2(\Om)}$ is a norm equivalent to $\norm{\bf u}_{{\bf H}^1(\Om)}$ if ${\bf u}\in{\bf H}^1(\Om)$, ${\bf u}\cdot\boldsymbol{\nu} = 0$ on $\pd\Om$, and $\int_\Om{\bf u}\cdot\boldsymbol{\be}\, dx = 0$.
\end{lemma}
\begin{proof}
The proof of the inequality \eqref{eq-3.7-AR} is similar to that of \cite[(3.7)]{AR-JDE2014} with a slightly modification to the two-dimensional case.

Suppose on the contrary of the lemma. 
There exists a sequence $\{{\bf u}_k\}_k$ in ${\bf H}^1(\Om)$ such that
\[
\norm{{\bf u}_k - \mathcal P{\bf u}_k}_{{\bf L}^2(\Om)}^2 > k \bke{ \norm{\mathbb S({\bf u}_k)}_{{\bf L}^2(\Om)}^2 + \int_{\pd\Om} |{\bf u}_k\cdot\boldsymbol{\nu}|^2\, dS },
\]
where $\mathcal P$ is the orthogonal projection from ${\bf L}^2(\Om)$ onto $\boldsymbol{\mathcal T}(\Om)$.
We may assume $\norm{{\bf u}_k - \mathcal P{\bf u}_k}_{{\bf L}^2(\Om)}^2 = 1$. 
So
\EQ{\label{eq-lem-3.3-pf-AR}
\frac1k > \norm{\mathbb S({\bf u}_k)}_{{\bf L}^2(\Om)}^2 + \int_{\pd\Om} |{\bf u}_k\cdot\boldsymbol{\nu}|^2\, dS,\quad k = 1,2,\ldots.
}
Set ${\bf w}_k = {\bf u}_k - \mathcal P{\bf u}_k$.
Then, ${\bf w}_k$ is bounded in ${\bf H}^1(\Om)$ by the Korn inequality.
In particular, ${\bf w}_k$ is bounded in ${\bf W}^{1,p}(\Om)$ for all $1\le p \le2$.
By Sobolev embedding ${\bf W}^{1,p}(\Om)\hookrightarrow {\bf L}^q(\Om)$ for all $1\le q<2p/(2-p)$.
Choosing $1<p<2$ and $q=2$, we get ${\bf H}^1(\Om)\hookrightarrow {\bf L}^2(\Om)$.
So, by Rellich-Kondrachov compactness theorem, ${\bf w}_k$ converges, up to a subsequence, to ${\bf w}$ in $L^2(\Om)$ and weakly in ${\bf H}^1(\Om)$.
Thus, it follows by taking the limit in \eqref{eq-lem-3.3-pf-AR} that $\norm{\mathbb S({\bf w})}_{{\bf L}^2(\Om)} = 0$ and ${\bf w}\cdot\boldsymbol{\nu} = 0$ on $\pd\Om$, ${\bf w}\in\boldsymbol{\mathcal T}(\Om)$.
On the other hand, ${\bf w} = \lim_{k\to\infty}({\bf u}_k - \mathcal P{\bf u}_k)\in\boldsymbol{\mathcal T}(\Om)^\perp$, where $\boldsymbol{\mathcal T}(\Om)^\perp$ is the orthogonal complement of $\boldsymbol{\mathcal T}(\Om)$ in ${\bf L}^2(\Om)$. 
So, we must have ${\bf w} = {\bf 0}$.
This contradicts with the relation $\norm{{\bf w}_k}_{{\bf L}^2(\Om)} = 1$ for all $k$ and completes the proof of the inequality \eqref{eq-3.7-AR}.
The equivalence of the norms follows from the Korn's second inequality: $\norm{\nb{\bf u}}_{\mathbb L^2(\Om)}^2 \le C\norm{\mathbb S({\bf u})}_{{\bf L}^2(\Om)}^2$ if ${\bf u}\in{\bf H}^1(\Om)$ such that $\int_\Om{\bf u}\cdot\boldsymbol{\be}\, dx = 0$.
This proves the lemma.
\end{proof}

We are now in a position of prove the existence and uniqueness of weak solution to the generalized Stokes system \eqref{sect0-abstract-stokes-system} for $\eta=0$, $g=0$.
We first consider the Hilbert case, $p=2$.

\begin{theorem}[$\eta=0$, $g=0$, $p=2$]\label{thm-3.4-AR}
Suppose $\eta=0$ and $g=0$. 
Let ${\bf f}\in[{\bf H}^2_0(\div,\Om)]'$ and $h\in H^{-\frac12}(\pd\Om)$ satisfy the compatibility condition \eqref{eq-3.15-AR} with $p=2$.
Then, the generalized Stokes system \eqref{sect0-abstract-stokes-system} has a unique solution $({\bf u},P)\in({\bf H}^1(\Om)\times L^2(\Om))/\boldsymbol{\mathcal N}(\Om)$.
Moreover, we have the estimate
\EQ{\label{eq-3.10-AR}
\norm{\bf u}_{{\bf H}^1(\Om)/\boldsymbol{\mathcal T}(\Om)} + \norm{P}_{L^2(\Om)/\R}
\lesssim\bke{ \norm{\bf f}_{[{\bf H}^2_0(\div,\Om)]'} + \norm{h}_{H^{-\frac12}(\pd\Om)} }.
}
\end{theorem}
\begin{proof}
The proof is exactly the same as that of \cite[Theorem 3.4]{AR-JDE2014} using the Poincar\'e-Morrey inequality \eqref{eq-3.7-AR} and the weak formulation in Proposition \ref{prop-3.1-AR} via an application of Lax-Milgram theorem on the bilinear form $a$ on ${\bf H}^1(\Om)$ defined by
\[
a({\bf u},\boldsymbol{\varphi}) = \int_\Om \mathbb S({\bf u}):\mathbb S(\boldsymbol{\varphi})\, dx.
\]
\end{proof}

Next, we extend the result in Theorem \ref{thm-3.4-AR} to the case $p\ge2$ in the following theorem.

\begin{theorem}[$\eta=0$, $g=0$, $p\ge2$]\label{thm-3.7-AR}
Suppose $\eta=0$, $g=0$, and $2\le p<\infty$.
Let ${\bf f}\in[{\bf H}^p_0(\div,\Om)]'$ and $h\in W^{-\frac1p,p}(\pd\Om)$ satisfy the compatibility condition \eqref{eq-3.15-AR}.
Then, the generalized Stokes system \eqref{sect0-abstract-stokes-system} has a unique solution $({\bf u},P)\in({\bf W}^{1,p}(\Om)\times L^p(\Om))/\boldsymbol{\mathcal N}(\Om)$.
Moreover, we have the following estimate:
\EQ{\label{eq-est-thm3.7}
\norm{\bf u}_{{\bf W}^{1,p}(\Om)/\boldsymbol{\mathcal T}(\Om)} + \norm{P}_{L^p(\Om)/\R}
\lesssim \bke{ \norm{\bf f}_{[{\bf H}^{p'}_0(\div,\Om)]'} + \norm{h}_{W^{-\frac1p,p}(\pd\Om)} }.
}
\end{theorem}

\begin{proof}
The proof is similar to that of \cite[Theorem 3.7]{AR-JDE2014} with a slightly modification to the two-dimensional case.

Note that for $p\ge2$
\EQ{\label{eq-1-pf-thm3.7}
[{\bf H}^{p'}_0(\div,\Om)]' \hookrightarrow [{\bf H}^2_0(\div,\Om)]'\ \text{ and }\ 
{\bf W}^{-\frac1p,p}(\pd\Om) \hookrightarrow {\bf H}^{-\frac12}(\pd\Om).
}
By Theorem \ref{thm-3.4-AR} for the case of $p=2$, 
the generalized Stokes system \eqref{sect0-abstract-stokes-system}, when $\eta=0$ and $g=0$, has a unique solution $({\bf u},P)\in({\bf H}^1(\Om)\times L^2(\Om))/\boldsymbol{\mathcal N}(\Om)$.
Applying Corollary \ref{cor-2.5-AR}, we have
\[
(\curl{\bf u})\boldsymbol{\tau} 
= 2[\mathbb S({\bf u})\boldsymbol{\nu}]_{\boldsymbol{\tau}} + 2\ka{\bf u}_{\boldsymbol{\tau}}
= (h + 2\ka({\bf u}\cdot\boldsymbol{\tau}) ) \boldsymbol{\tau}
\ \text{ on }{\bf H}^{-\frac12}(\pd\Om),\quad \ka\text{ is the curvature on }\pd\Om,
\]
because ${\bf u}\in{\bf E}^2(\Om)$ and ${\bf u}\cdot\boldsymbol{\nu}|_{\pd\Om} = 0$.
Thus, $({\bf u},P)$ is the solution to the auxiliary problem \eqref{eq-stokes-auxiliary} with $H = h + 2\ka({\bf u}\cdot\boldsymbol{\tau})$.
It then follows from the Green formula \eqref{eq-2.7-AR} in Lemma \ref{lem-2.3-AR} that $({\bf u},P)$ solves the following variational problem:
For all $\boldsymbol{\varphi}\in{\bf V}^2(\Om)$,
\[
\int_\Om (\curl{\bf u})(\curl\boldsymbol{\varphi})\, dx 
= \bka{{\bf f},\boldsymbol{\varphi}}_{\Om,2} + \bka{ (h + 2\ka({\bf u}\cdot\boldsymbol{\tau}))\boldsymbol{\tau}, \boldsymbol{\varphi} }_{\pd\Om,2}.
\]
In particular
\[
\bka{{\bf f},\boldsymbol{\varphi}}_{\Om,2} + \bka{ (h + 2\ka({\bf u}\cdot\boldsymbol{\tau}))\boldsymbol{\tau}, \boldsymbol{\varphi} }_{\pd\Om,2} = 0,\ \text{ for all }\boldsymbol{\varphi}\in{\bf K}^2_T(\Om).
\]
More generally, for $p\ge2$ and $\boldsymbol{\varphi}\in{\bf K}^{p'}_T(\Om)$,
\EQN{
\bka{{\bf f},\boldsymbol{\varphi}}_{\Om,p} + \bka{ (h + 2\ka({\bf u}\cdot\boldsymbol{\tau}))\boldsymbol{\tau}, \boldsymbol{\varphi} }_{\pd\Om,p} = 0,
}
which verifies the compatibility condition \eqref{eq-2.13-AR} because ${\bf u}\cdot\boldsymbol{\tau} \in H^{\frac12}(\pd\Om) \hookrightarrow W^{-\frac1p,p}(\pd\Om)$ for $2\le p<\infty$ so that $H = h + 2\ka({\bf u}\cdot\boldsymbol{\tau})\in W^{-\frac1p,p}(\pd\Om)$.
The compatibility condition \eqref{eq-2.14-AR} is also satisfied since $\eta=0$ and $g=0$. 
By Theorem \ref{thm-2.7-AR}, we have $({\bf u},P)\in {\bf W}^{1,p}(\Om)\times (L^p(\Om)/\R)$, $2\le p<\infty$, by uniqueness, and
\EQ{\label{eq-2-pf-thm3.7}
\norm{\bf u}_{{\bf W}^{1,p}(\Om)} + \norm{P}_{L^p(\Om)/\R}
&\lesssim  \bke{ \norm{\bf f}_{[{\bf H}^{p'}_0(\div,\Om)]'} + \norm{h + 2\ka({\bf u}\cdot\boldsymbol{\tau})}_{W^{-\frac1p,p}(\pd\Om)} }\\
&\lesssim \bke{ \norm{\bf f}_{[{\bf H}^{p'}_0(\div,\Om)]'} + \norm{h}_{W^{-\frac1p,p}(\pd\Om)} + 2\ka\norm{\bf u}_{{\bf W}^{-\frac1p,p}(\pd\Om)} }.
}
%Now, repeated application of Theorem \ref{thm-2.7-AR} enables us to assume that ${\bf u}|_{\pd\Om}\in{\bf W}^{1-\frac1p,p}(\pd\Om)$.
We now establish the estimate \eqref{eq-est-thm3.7}.
For $p=2$, Theorem \ref{thm-3.7-AR} is proved in Theorem \ref{thm-3.4-AR}.
For $2<p<\infty$, by Morrey's inequality $W^{1,p}(\Om) \hookrightarrow C^{0,\ga}(\Om) \hookrightarrow L^q(\Om)$, $\ga = 1-(2/p)$, for all $q\in[1,\infty]$.
Hence, we have
\[
\norm{\bf u}_{{\bf W}^{-\frac1p,p}(\pd\Om)} \lec \norm{\bf u}_{{\bf W}^{1-\frac1q,q}(\pd\Om)},\ \text{ for all }q\in[1,\infty].
\] 
Choosing $q=2$ and applying trace theorem, we deduce
\EQ{\label{eq-3-pf-thm3.7}
\norm{\bf u}_{{\bf W}^{-\frac1p,p}(\pd\Om)} 
\lec \norm{\bf u}_{{\bf H}^\frac12(\pd\Om)}
&\lec \norm{\bf u}_{{\bf H}^1(\Om)}\\
&\lec \norm{\bf f}_{{\bf H}^2_0(\div,\Om)]'} + \norm{h}_{H^{-\frac12}(\pd\Om)}\\
&\lec \norm{\bf f}_{{\bf H}^{p'}_0(\div,\Om)]'} + \norm{h}_{W^{-\frac1p,p}(\pd\Om)},
}
where we used the estimate \eqref{eq-3.10-AR} (for $p=2$) in the second last inequality and the embedding \eqref{eq-1-pf-thm3.7} in the last inequality.
The desired estimate \eqref{eq-est-thm3.7} follows by using \eqref{eq-3-pf-thm3.7} in \eqref{eq-2-pf-thm3.7}.
This proves Theorem \ref{thm-3.7-AR}.
\end{proof}

\medskip

\noindent{\bf The case of general $\eta$ and $g$.}

\begin{corollary}\label{cor-3.8-AR}
For $2\le p<\infty$,
let ${\bf f}\in[{\bf H}^{p'}_0(\div,\Om)]'$, $\eta\in L^p(\Om)$, $g\in W^{1-\frac1p,p}(\pd\Om)$, and $h\in W^{-\frac1p,p}(\pd\Om)$ satisfy the compatibility conditions \eqref{eq-2.14-AR} and \eqref{eq-3.15-AR}.
Then, the generalized Stokes system \eqref{sect0-abstract-stokes-system} has a unique solution $({\bf u},P)\in({\bf W}^{1,p}(\Om)\times L^p(\Om))/\boldsymbol{\mathcal N}(\Om)$.
In addition, $({\bf u},P)$ satisfies the estimate
\[
\norm{\bf u}_{{\bf W}^{1,p}(\Om)/\boldsymbol{\mathcal T}(\Om)} + \norm{P}_{L^p(\Om)/\R}
\le C\bke{ \norm{\bf f}_{[{\bf H}^{p'}_0(\div,\Om)]'} +  \norm{\eta}_{L^p(\Om)} + \norm{g}_{W^{1-\frac1p,p}(\pd\Om)} + \norm{h}_{W^{-\frac1p,p}(\pd\Om)} }.
\]
\end{corollary}
\begin{proof}
The proof follows from an application of Theorem \ref{thm-3.7-AR} and Green formula \eqref{eq-2.8-AR} via the same argument in the proof of \cite[Corollary 3.8]{AR-JDE2014} using the decomposition ${\bf u} = {\bf z} + \nb\th$,
were ${\bf z}$ solves
\begin{subequations}
\label{eq-3.18-AR}
\begin{align}
-\De{\bf z} + \nb P& = {\bf f} + \nb\eta\ \text{ in } \Om,\label{eq-3.18-AR-a}\\
\div{\bf z}&=0\ \text{ in } \Om,\label{eq-3.18-AR-b}\\
{\bf z}\cdot\boldsymbol{\nu} &= 0\ \text{ on }\pd\Om,\label{eq-3.18-AR-c}\\
2[\mathbb S({\bf z})\boldsymbol{\nu}]_{\boldsymbol{\tau}}&= H\boldsymbol{\tau}\ \text{ on }\pd\Om,\label{eq-3.18-AR-d}
\end{align}
\end{subequations}
where $H = h - 2(\mathbb S(\nb\th) \boldsymbol{\nu})\cdot\boldsymbol{\tau}$,
and $\th$ solves \eqref{eq-stokes-auxiliary-theta}.
\end{proof}

\medskip

\subsection{Strong solutions of the generalized Stokes system \eqref{sect0-abstract-stokes-system}}

\begin{theorem}\label{thm-4.1-AR}
For $2\le p<\infty$,
let ${\bf f}\in{\bf L}^p(\Om)$, $\eta\in W^{1,p}(\Om)$, $g\in W^{2-\frac1p,p}(\pd\Om)$, and $h\in W^{1-\frac1p,p}(\pd\Om)$ satisfy the compatibility conditions \eqref{eq-2.14-AR} and \eqref{eq-3.15-AR}.
Then, the generalized Stokes system \eqref{sect0-abstract-stokes-system} has a unique solution $({\bf u},P)\in({\bf W}^{2,p}(\Om)\times W^{1,p}(\Om))/\boldsymbol{\mathcal N}(\Om)$.
Further, the solution satisfies the following estimate:
\EQ{\label{eq-est-thm4.1}
\norm{\bf u}_{{\bf W}^{2,p}(\Om)/\boldsymbol{\mathcal T}(\Om)} + \norm{P}_{W^{1,p}(\Om)/\R}
\lesssim \bke{ \norm{\bf f}_{{\bf L}^p(\Om)} +  \norm{\eta}_{W^{1,p}(\Om)} + \norm{g}_{W^{2-\frac1p,p}(\pd\Om)} + \norm{h}_{W^{1-\frac1p,p}(\pd\Om)} }.
}
\end{theorem}
\begin{proof}
To begin with, it follows from Corollary \ref{cor-3.8-AR} that \eqref{sect0-abstract-stokes-system} has a unique solution $({\bf u},P)\in({\bf W}^{1,p}(\Om)\times L^p(\Om))/\boldsymbol{\mathcal N}(\Om)$. 
It remains to improve the regularity of the solution and derive the estimate \eqref{eq-est-thm4.1}.

Adopting the same idea as in the proof of Corollary \ref{cor-3.8-AR}, we decompose
${\bf u} = {\bf z} + \nb\th$,
where ${\bf z}$ solves \eqref{eq-3.18-AR} with $H = h - 2(\mathbb S(\nb\th) \boldsymbol{\nu})\cdot\boldsymbol{\tau}$, and $\th$ solves \eqref{eq-stokes-auxiliary-theta}.
For $\th$ solving \eqref{eq-stokes-auxiliary-theta}, we use the classical elliptic theory to get
\EQ{\label{eq-theta-W3p}
\norm{\th}_{W^{3,p}(\Om)/\R} \lec \norm{\eta}_{W^{1,p}(\Om)} + \norm{g}_{W^{2-\frac1p,p}(\pd\Om)}.
}
For ${\bf z}$ satisfying \eqref{eq-3.18-AR} with $H = h - 2(\mathbb S(\nb\th) \boldsymbol{\nu})\cdot\boldsymbol{\tau}$, we set $\om = \curl{\bf z}$.
Since ${\bf z}\in{\bf E}^p$ and ${\bf z}\cdot\boldsymbol{\nu}=0$ on $\pd\Om$,
we have that $\curl{\bf z} = 2(\mathbb S({\bf z})\boldsymbol{\nu})\cdot\boldsymbol{\tau} + 2\ka({\bf z}\cdot\boldsymbol{\tau})$ on $\pd\Om$ by Corollary \ref{cor-2.5-AR}.
Then $\om$ solves
\begin{subequations}
\label{eq-om}
\begin{align}
-\De\om& = \curl{\bf f}\qquad\qquad\qquad\qquad\qquad\ \ \, \text{ in } \Om,\label{eq-om-a}\\
\om &= h - 2(\mathbb S(\nb\th)\boldsymbol{\nu})\cdot\boldsymbol{\tau} + 2\ka({\bf z}\cdot\boldsymbol{\tau})\ \text{ on }\pd\Om.\label{eq-om-b}
\end{align}
\end{subequations}
The classical elliptic theory then gives
\EQ{\label{eq-om-W1p}
\norm{\om}_{W^{1,p}(\Om)} 
&\lec \norm{\curl{\bf f}}_{[H^{p'}_0(\Om)]'} + \norm{h - 2(\mathbb S(\nb\th)\boldsymbol{\nu})\cdot\boldsymbol{\tau} + 2\ka({\bf z}\cdot\boldsymbol{\tau})}_{W^{1-\frac1p,p}(\pd\Om)}\\
&\lec \norm{\bf f}_{{\bf L}^p(\Om)} + \norm{h}_{W^{1-\frac1p,p}(\pd\Om)} +  \norm{\pd^2\th}_{W^{1-\frac1p,p}(\pd\Om)} + \norm{\bf z}_{W^{1-\frac1p,p}(\pd\Om)}.
}
By trace theorem, $\norm{\pd^2\th}_{W^{1-\frac1p,p}(\pd\Om)}\lec\norm{\pd^2\th}_{W^{1,p}(\Om)} \lec \norm{\th}_{W^{3,p}(\Om)}$ and $\norm{\bf z}_{W^{1-\frac1p,p}(\pd\Om)} \lec \norm{\bf z}_{W^{1,p}(\Om)}$.
Thus, 
\EQN{
\norm{\om}_{W^{1,p}(\Om)} 
\lec \norm{\curl{\bf f}}_{[H^{p'}_0(\Om)]'} + \norm{h}_{W^{1-\frac1p,p}(\pd\Om)} +  \norm{\th}_{W^{3,p}(\Om)} + \norm{\bf z}_{W^{1,p}(\Om)}.
}
Moreover, applying Theorem \ref{thm-3.7-AR} on ${\bf z}$, noting that the compatibility condition \eqref{eq-3.15-AR} can be checked with the aid of Green formula \eqref{eq-2.8-AR} as in the proof of \cite[Corollary 3.8]{AR-JDE2014}, one has
\EQ{\label{eq-z-W1p}
\norm{\bf z}_{W^{1,p}(\Om)} 
&\lec \norm{{\bf f} + \nb\eta}_{[{\bf H}^{p'}_0(\div,\Om)]'} + \norm{h - 2(\mathbb S(\nb\th) \boldsymbol{\nu})\cdot\boldsymbol{\tau} }_{W^{-\frac1p,p}(\pd\Om)}\\
&\lec \norm{\bf f}_{[{\bf H}^{p'}_0(\div,\Om)]'} + \norm{\nb\eta}_{[{\bf H}^{p'}_0(\div,\Om)]'} + \norm{h}_{W^{-\frac1p,p}(\pd\Om)} + \norm{\pd^2\th}_{W^{-\frac1p,p}(\pd\Om)}\\
&\lec \norm{\bf f}_{{\bf L}^p(\Om)} + \norm{\eta}_{W^{1,p}(\Om)} + \norm{h}_{W^{1-\frac1p,p}(\pd\Om)} + \norm{\th}_{W^{2-\frac1p,p}(\pd\Om)}\\
&\lec \norm{\bf f}_{{\bf L}^p(\Om)} + \norm{\eta}_{W^{1,p}(\Om)} + \norm{h}_{W^{1-\frac1p,p}(\pd\Om)} + \norm{\th}_{W^{2,p}(\Om)},
}
where we used the trace theorem again in the last inequality.
Thus, by using \eqref{eq-theta-W3p} and \eqref{eq-z-W1p} in \eqref{eq-om-W1p} we obtain
\[
\norm{\om}_{W^{1,p}(\Om)} \lec \norm{\bf f}_{{\bf L}^p(\Om)} + \norm{h}_{W^{1-\frac1p,p}(\pd\Om)} + \norm{\eta}_{W^{1,p}(\Om)} + \norm{g}_{W^{2-\frac1p,p}(\pd\Om)}.
\]
Therefore, 
\[
{\bf z}\in{\bf X}^{2,p}(\Om):=\bket{{\bf v}\in{\bf L}^p(\Om): \div{\bf v}\in W^{1,p}(\Om),\, \curl{\bf v}\in W^{1,p}(\Om),\ {\bf v}\cdot\boldsymbol{\nu}\in W^{2-\frac1p,\frac1p}(\pd\Om)}
\]
since $\div{\bf z} = 0$, $\curl{\bf z} = \om\in W^{1,p}$, ${\bf z}\cdot\boldsymbol{\nu} = 0$ on $\pd\Om$.
Thus, ${\bf u} = {\bf z} + \nb\th\in {\bf X}^{2,p}(\Om)$ since $\nb\th$ also lies in ${\bf X}^{2,p}(\Om)$.
We can now apply the embedding ${\bf X}^{2,p}(\Om) \hookrightarrow {\bf W}^{2,p}(\Om)$ in iv) of Remark 2 of \cite[Corollary 1, p.p. 212--213]{DL-book1990v3} (see also \cite[(1.15)]{Mitrea-DIE2005} and \cite[Lemma 2.2]{Fan-Li-Li-ARMA2022}) to show that ${\bf u}\in{\bf W}^{2,p}(\Om)$ and derive the estimate \eqref{eq-est-thm4.1} for ${\bf u}$.
Finally, for the pressure $P$, the regularity $P\in W^{1,p}(\Om)$ and its estimate in \eqref{eq-est-thm4.1}
follows from the equation $\nb P = \De{\bf u} + {\bf f} \in {\bf L}^p(\Om)$.
This completes the proof of the theorem.
\end{proof}

\begin{remark}[The case $1<p<2$]\label{sect2-rmk-psmall}
Theorem \ref{thm-4.1-AR} (as well as Theorem \ref{thm-3.7-AR} and Corollary \ref{cor-3.8-AR}) can be proved for $1<p<2$ by the duality argument performed in the proof of \cite[Theorem 3.9]{AR-JDE2014}.  We skip the detailed discussion for small $p$. %since the present paper only uses the result for large $p$ to apply Sobolev embedding to get $L^\infty$ bound for the velocity and pressure in the outer region.
\end{remark}
Combining Theorem \ref{thm-4.1-AR} and Remark \ref{sect2-rmk-psmall}, we have proved Theorem \ref{thm-global-regularity-NS}.  Theorem \ref{thm-global-regularity-NS} establishes the global regularity of the two dimenional stationary Stokes operator with the Navier boundary condition.  It plays the vital role on the construction of the outer solution to the velocity field ${\bf u}$ shown in Section \ref{sect-model-stokes}.  Furthermore, by using Theorem \ref{thm-global-regularity-NS}, the semigroup estimate of the non-stationary Stokes operator is also be developed in Subsection \ref{subsect-stokes-semigroup}. 

\section{Global Existence: Subcritical Mass Case}\label{sect3-free-global-existence}
In this section, we shall discuss the global well-posedness of system (\ref{PKSNS-time-dependent}) and prove Theorem \ref{thm-global-existence-PKS-NS}.  Before this, we have to study the local-in-time existence of the solution.  Whereas, noting that the velocity ${\bf u}$ satisfies the incompressible Navier-Stokes equation subject to the Navier boundary condition rather than the no-slip one, we are driven to develop the corresponding semigroup theory and establish the desired semigroup estimate.  The detailed discussion will be shown in Subsection \ref{subsect-stokes-semigroup}.

\subsection{Analyticity of Stokes semigroup in ${\bf L}^p(\Om)$}\label{subsect-stokes-semigroup}
~

In this subsection, we prove the Stokes operator with Navier slip boundary conditions generates a bounded analytic semigroup on ${\bf L}^p_\si(\Om)$ for all $1<p<\infty$, where
\[
{\bf L}^p_\si(\Om) = \bket{{\bf v}\in{\bf L}^p(\Om): \div{\bf v} = 0\text{ in }\Om\text{ and }{\bf v}\cdot\boldsymbol{\nu} = 0\text{ on }\pd\Om}.
\]

\medskip

\noindent{\bf The case $p=2$.}

\begin{theorem}\label{thm-resolvent-p=2}
Let $\ve\in(0,\pi)$ be fixed, ${\bf f}\in{\bf L}^2(\Om)$ and $\la\in\Si_\ve = \bket{\la\in\CC^*: |\arg\la| < \pi - \ve}$. Then we have the following:

(i) Assume that $\Om$ is of class $C^{1,1}$. Then, the resolvent problem
\begin{subequations}
\label{eq-stokes-resolvent}
\begin{align}
&\la{\bf u} - \De{\bf u} + \nb P = {\bf f},\quad \div{\bf u} = 0,\ \text{ in } \Om,\label{eq-stokes-resolvent-a}\\
&{\bf u}\cdot\boldsymbol{\nu} = 0,\qquad\quad [\mathbb S({\bf u})\boldsymbol{\nu}]_{\boldsymbol{\tau}} = {\bf 0},\ \text{ on }\pd\Om,\label{eq-stokes-resolvent-b}
\end{align}
\end{subequations}
has a unique solution $({\bf u},P)\in{\bf H}^1(\Om)\times(L^2(\Om))/\R$.
Moreover, the solution satisfies the estimates
\EQ{\label{eq-semigroup-3.4}
\norm{\bf u}_{{\bf L}^2(\Om)} \le \frac{C_\ve}{|\la|} \norm{\bf f}_{{\bf L}^2(\Om)},
}
\EQ{\label{eq-semigroup-3.5}
\norm{\mathbb S({\bf u})}_{\mathbb L^2(\Om)} \le \frac{C_\ve}{\sqrt{2|\la|}} \norm{\bf f}_{{\bf L}^2(\Om)},
}
for some constant $C_\ve>0$ independent of ${\bf f}$ and $\la$.

(ii) If $\Om$ is of class $C^{2,1}$, then $({\bf u},P)\in{\bf H}^2(\Om)\times H^1(\Om)$, and ${\bf u}$ satisfies the estimate
\[
\norm{\bf u}_{{\bf H}^2(\Om)} \le \frac{C(\Om,\la,\ve)}{|\la|} \norm{\bf f}_{{\bf L}^2(\Om)},
\]
for some $C(\Om,\la,\ve)>0$.
\end{theorem}

\begin{proof}
The theorem follows from the same argument as in the proof of \cite[Theorem 3.2]{AAR-semigroup2016} via the variational problem of finding ${\bf u}\in{\bf V}^2(\Om)$ such that
\[
a({\bf u},\boldsymbol{\varphi}) = \int_{\Om} {\bf f}\cdot\overline{\boldsymbol{\varphi}}\,dx,\ \text{ for all }\boldsymbol{\varphi}\in{\bf V}^2(\Om),\qquad
a({\bf u}, \boldsymbol{\varphi}) := \la\int_\Om{\bf u}\cdot\overline{\boldsymbol{\varphi}}\,dx + 2\int_\Om\mathbb S({\bf u}):\mathbb S(\overline{\boldsymbol{\varphi}})\, dx,
\]
which is equivalent to finding a distributional solution $({\bf u},P)\in{\bf H}^1(\Om)\times L^2(\Om)$ of \eqref{eq-stokes-resolvent} thanks to the Green formula \eqref{eq-2.8-AR} in Lemma \ref{lem-2.4-AR}.
\end{proof}

\medskip

\noindent{\bf The general case $1<p<\infty$.}
We now extend Theorem \ref{thm-resolvent-p=2} to the general case $1<p<\infty$. 
To begin with, we first establish the following existence and uniqueness theorem.
\begin{theorem}\label{thm-semigroup-3.4}
Assume that $\Om$ is of $C^{1,1}$.
Let $\ve\in(0,\pi)$ be fixed, ${\bf f}\in{\bf L}^2(\Om)$ and $\la\in\Si_\ve$. 
Then the resolvent problem \eqref{eq-stokes-resolvent} has a unique solution $({\bf u},P)\in {\bf W}^{1,p}(\Om)\times(L^p(\Om)/\R)$. 
If $\Om$ is of class $C^{2,1}$, then $({\bf u},P)\in {\bf W}^{2,p}(\Om)\times W^{1,p}(\Om)$. 
\end{theorem}

\begin{proof}
By using a duality argument with Theorem \ref{thm-resolvent-p=2} and embedding theorems as in the proof of Theorem \ref{thm-3.7-AR}, the theorem follows.
We omit the details for brevity.
\end{proof}

Now, we proceed to prove resolvent estimate for $1<p<\infty$.
For this purpose, we use the following lemma whose proof is the same as for \cite[Lemma 2.5]{AAE-CM2016} with a slightly modification to the two-dimensional case.

\begin{lemma}
Let ${\bf u}\in{\bf W}^{1,p}(\Om)$ such that $\De{\bf u}\in{\bf L}^p(\Om)$. Then
\EQ{\label{eq-semigroup-2.8}
-\int_\Om |{\bf u}|^{p-2}&\De{\bf u}\cdot\overline{\bf u}\, dx 
= \int_\Om |{\bf u}|^{p-2}|\nb{\bf u}|^2\, dx + 4\,\frac{p-2}{p^2} \int_\Om \abs{\nb|{\bf u}|^{p/2}}^2 dx\\
& + i(p-2)\sum_{k=1}^2 \int_\Om |{\bf u}|^{p-4} \Re\bke{\frac{\pd{\bf u}}{\pd x_k}\cdot\overline{\bf u}} \Im\bke{\frac{\pd{\bf u}}{\pd x_k}\cdot\overline{\bf u}} dx - \left<\frac{\pd{\bf u}}{\pd\boldsymbol{\nu}}, |{\bf u}|^{p-2}{\bf u}\right>_{\pd\Om,p}.
}
\end{lemma}

We also need the following lemma whose proof is is the same as for \cite[Lemma 2.1]{AR-JDE2014} with a slightly modification to the two-dimensional case.

\begin{lemma}
For any ${\bf v}\in{\bf W}^{2,p}(\Om)$, we have 
\EQ{\label{eq-semigroup-2.3}
2 \bkt{\mathbb S({\bf v})\boldsymbol{\nu}}_{\boldsymbol{\tau}} = \nb_{\boldsymbol{\tau}}({\bf v}\cdot\boldsymbol{\nu}) + \bke{\frac{\pd{\bf v}}{\pd\boldsymbol{\nu}}}_{\boldsymbol{\tau}} - \bke{{\bf v}_{\boldsymbol{\tau}}\cdot\frac{\pd{\boldsymbol{\nu}}}{\pd s}}{\boldsymbol{\tau}},
}
where $s$ is the arc length parameter of $\pd\Om$.
\end{lemma}

\begin{remark}
If $\Om$ is of class $C^{1,1}$, and if ${\bf u}\in{\bf W}^{1,p}(\Om)$ such that ${\bf u}\cdot\boldsymbol{\nu} = 0$ and $[\mathbb S({\bf u})\boldsymbol{\nu}]_{\boldsymbol{\tau}} = {\bf 0}$ on $\pd\Om$, then, thanks to \eqref{eq-semigroup-2.3},$(\pd{\bf u}/\pd\boldsymbol{\nu})_{\boldsymbol{\tau}} = \bke{{\bf u}_{\boldsymbol{\tau}}\cdot(\pd\boldsymbol{\nu}/\pd s)} \boldsymbol{\tau}$ belongs to ${\bf W}^{1-1/p,p}(\pd\Om)\hookrightarrow{\bf L}^{p'}(\pd\Om)$. 
Consequently, the integral
\EQ{\label{eq-semigroup-2.9}
\int_{\pd\Om} |{\bf u}|^{p-2} \bke{\frac{\pd{\bf u}}{\pd\boldsymbol{\nu}}}_{\boldsymbol{\tau}}\cdot\overline{\bf u}\, dS
}
is well-defined, and the term $\left<\pd{\bf u}/\pd\boldsymbol{\nu}, |{\bf u}|^{p-2}{\bf u}\right>_{\pd\Om,p}$ in \eqref{eq-semigroup-2.8} can be replaced by \eqref{eq-semigroup-2.9}.
That is, we have for ${\bf u}\in{\bf W}^{1,p}(\Om)$ such that $\De{\bf u}\in{\bf L}^p(\Om)$ that 
\EQ{\label{eq-semigroup-2.8-new}
-\int_\Om &|{\bf u}|^{p-2}\De{\bf u}\cdot\overline{\bf u}\, dx 
= \int_\Om |{\bf u}|^{p-2}|\nb{\bf u}|^2\, dx + 4\,\frac{p-2}{p^2} \int_\Om \abs{\nb|{\bf u}|^{p/2}}^2 dx\\
& + i(p-2)\sum_{k=1}^2 \int_\Om |{\bf u}|^{p-4} \Re\bke{\frac{\pd{\bf u}}{\pd x_k}\cdot\overline{\bf u}} \Im\bke{\frac{\pd{\bf u}}{\pd x_k}\cdot\overline{\bf u}} dx - \int_{\pd\Om} |{\bf u}|^{p-2} \bke{\frac{\pd{\bf u}}{\pd\boldsymbol{\nu}}}_{\boldsymbol{\tau}}\cdot\overline{\bf u}\, dS.
}
\end{remark}

We are now ready to estimate of the solution of \eqref{eq-stokes-resolvent} for $p>2$.
To begin with, we consider the case when $\la$ is away from zero.

\begin{proposition}\label{prop-semigroup-3.8}
Let $1<p<\infty$ and $\la\in\CC^*$ with $\Re\la\ge0$.
Let ${\bf f}\in{\bf L}^p_\si(\Om)$, and let ${\bf u}\in{\bf W}^{2,p}(\Om)$ be the unique solution of the resolvent problem \eqref{eq-stokes-resolvent}. 
Then ${\bf u}$ satisfies the estimate
\EQ{\label{eq-semigroup-3.9}
\norm{\bf u}_{{\bf L}^p(\Om)} \le \frac{\ka}{|\la|} \norm{\bf f}_{{\bf L}^p(\Om)}.
}
for some $\ka = \ka(p,\Om) > 0$ independent of $\la$ and ${\bf f}$.
\end{proposition}

\begin{proof}
For the case when $\la$ is away from zero, the same argument as in the proof of \cite[Proposition 3.6]{AAR-semigroup2016}, with the aid of the formula \eqref{eq-semigroup-2.8-new} for two-dimensional case, enable us to find $\la_0 = \la_0(\Om,p)>0$ such that 
\eqref{eq-semigroup-3.9} holds for all $\la\in\CC^*$ with $\Re\la\ge0$ and $|\la|\ge\la_0$.

For the case when $0<|\la|\le\la_0$, \eqref{eq-semigroup-3.9} follows as in \cite[Remark 3.7]{AAR-semigroup2016} using \eqref{eq-semigroup-3.4} and \eqref{eq-semigroup-3.5} together with the Sobolev embeddings.

We skip the details and only note that the embeddings for dimension three used in the proofs of \cite[Proposition 3.6]{AAR-semigroup2016} (e.g. ${\bf H}^2(\Om)\hookrightarrow{\bf W}^{1,p}(\Om)$ for $2\le p\le4$ and ${\bf W}^{2,4}(\Om)\hookrightarrow{\bf W}^{1,\infty}(\Om)$) and in \cite[Remark 3.7]{AAR-semigroup2016} (e.g. ${\bf H}^1(\Om)\hookrightarrow{\bf L}^p(\Om)$ for $2\le p\le6$) also apply to our 2-D case.
Indeed, for dimension two we have that ${\bf H}^2(\Om)\hookrightarrow{\bf H}^s(\Om)\hookrightarrow{\bf W}^{1,p}(\Om)$ for $2\le p\le4$ and $s=2-(2/p)\in[1,3/2]$ and that ${\bf H}^1(\Om)\hookrightarrow{\bf H}^s(\Om)\hookrightarrow{\bf L}^p(\Om)$ for $2\le p\le 6$ and $s=1-(2/p)\in[0,2/3]$.
\end{proof}

Let $A_p$ be the Stokes operator on $L^p_\si(\Om)$ with Navier slip boundary conditions given by 
\[
A_p{\bf u} := -\mathbb P_p\De{\bf u},\ \text{ where }\mathbb P_p:{\bf L}^p(\Om)\to{\bf L}^p_\si(\Om)\text{ is the Helmholtz projection,}
\]
with domain
\begin{align}\label{sect2-semigroup-domain}
{\bf D}(A_p) = \bket{{\bf u}\in{\bf W}^{2,p}(\Om)\cap{\bf L}^p_\si(\Om): \bkt{\mathbb S({\bf u})\boldsymbol{\nu}}_{\boldsymbol{\tau}}=0\ \text{ on }\pd\Om}.
\end{align}

\begin{theorem}
The operator $-A_p$ generates a bounded analytic semigroup on ${\bf L}^p_\si(\Om)$.
\end{theorem}

\begin{proof}
It follows from Theorem \ref{thm-semigroup-3.4} and Proposition \ref{prop-semigroup-3.8} that $-A_p$ is a sectorial operator on ${\bf L}^p_\si(\Om)$. Therefore, $-A_p$ is the infinitesimal generator of an analytic semigroup by \cite[Theorem 1.3.4, p.20]{Henry-book1981}.
\end{proof}

We now derive an estimate for the semigroup $\{e^{-A_pt}\}_{t\ge0}$.
To this end, we first obtain a bound on pure imaginary powers of $-A_p$.

\begin{theorem}\label{thm-ghosh-3.5.1}
There exists an angle $\th\in(0,\pi/2)$ and a constant $C>0$ such that 
\EQ{\label{eq-ghosh-3.44}
\norm{A_p^{is}} \le Ce^{|s|\th},\ \text{ for all }s\in\R.
}
\end{theorem}

\begin{proof}
The proof is similar to that for \cite[(3.44)]{Ghosh-thesis2018} using the interpolation-extrapolation theory (see \cite[V.1.5]{Amann-book1995}) via the Stokes operator for Navier-type boundary conditions ${\bf u}\cdot{\bf n} = 0$, $\boldsymbol{\curl}\,{\bf u}\times{\bf n} = {\bf 0}$ on the boundary in dimension three.
To adapt the proof to our two-dimensional case, we only need to replace the boundary condition $\boldsymbol{\curl}\,{\bf u}\times{\bf n} = {\bf 0}$ by $\boldsymbol{\curl}\,{\bf u} = 0$, and the vector-valued $\boldsymbol{\curl}\,{\bf u}$ for dimension three by scalar-valued $\curl{\bf u}$ for dimension two.
\end{proof}

As a corollary of Theorem \ref{thm-ghosh-3.5.1}, we have the following Sobolev type embedding theorem for domains of fractional powers.

\begin{theorem}\label{thm-ghosh-3.5.4}
For all $1<p<\infty$ and for $\th\in(0,1/p)$,
\[
{\bf D}(A_p^\th)\hookrightarrow{\bf L}^q(\Om),\ \text{ where }\, \frac1q = \frac1p - \th.
\]
\end{theorem}

\begin{proof}
In view of the bound \eqref{eq-ghosh-3.44} on pure imaginary powers of $A_p$, we can apply \cite[Theorem 1.15.3]{Triebel-book1995} to determine the domains of definition of $A_p^\th$ for $\th\in(0,1)$:
\[
{\bf D}(A_p^\th) = \bkt{{\bf L}^p_\si(\Om), {\bf D}(A_p)}_\th,
\]
the complex interpolation space.
By definition of ${\bf D}(A_p)$ and Sobolev embedding, we have
\[
\bkt{{\bf L}^p_\si(\Om), {\bf D}(A_p)}_\th 
\hookrightarrow \bkt{{\bf L}^p(\Om), {\bf W}^{2,p}(\Om)}_\th 
\hookrightarrow {\bf W}^{2\th,p}(\Om)
\hookrightarrow {\bf L}^q(\Om),\ \text{ for }\frac1q = \frac1p - \th,
\]
provided $0<\th<1/p$.
\end{proof}

We are now in the position to estimate the semigroup $\{e^{-A_pt}\}_{t\ge0}$.

\begin{theorem}\label{sect3-semigroup-estimate-pre}
For all $p\le q<\infty$, there exists $\de>0$ such that for all $t>0$
\EQ{\label{eq-ghosh-3.65}
\norm{e^{-A_pt}{\bf u}_0}_{{\bf L}^q(\Om)} \le C(\Om,p)\, e^{-\de t} t^{-(1/p - 1/q)} \norm{{\bf u}_0}_{{\bf L}^p(\Om)},
}
\EQ{\label{eq-ghosh-3.66}
\norm{\mathbb S(e^{-A_pt}{\bf u}_0)}_{\mathbb L^q(\Om)} \le C(\Om,p)\, e^{-\de t} t^{-\bke{1/p-1/q}-1/2} \norm{{\bf u}_0}_{{\bf L}^p(\Om)},
}
and for all $m,n\in\N$
\EQ{\label{eq-ghosh-3.67}
\norm{\frac{\pd^m}{\pd t^m}\, A_p^ne^{-A_pt}{\bf u}_0}_{{\bf L}^q(\Om)} \le C(\Om,p)\, e^{-\de t} t^{-(m+n)-\bke{1/p-1/q}} \norm{{\bf u}_0}_{{\bf L}^p(\Om)}.
}
\end{theorem}

\begin{proof}
The proof is similar to that of \cite[Theorem 3.6.3]{Ghosh-thesis2018} with a slightly modification to the two-dimensional case.

For the case $p=q$, the estimates \eqref{eq-ghosh-3.65}, \eqref{eq-ghosh-3.66} and \eqref{eq-ghosh-3.67} follow from \cite[Theorem 6.13, Chapter 2]{Pazy-book1983}.

Suppose that $p<q$. Let $s\in(1/p-1/q,1/p)$ and set $1/p_0 = 1/p - s$.
Obviously, $p<q<p_0$ so that $1/q = \th/p_0 + (1-\th)/p$ for $\th\in(0,1)$ and $\th = \frac{1/p-1/q}{1/p-1/p_0} = \frac{1/p-1/q}s$.
For all $t>0$, we have $e^{-A_pt}{\bf u}_0\in{\bf D}(A_p^s)\hookrightarrow{\bf L}^{p_0}(\Om)$ by Theorem \ref{thm-ghosh-3.5.4}.
Thus, $e^{-A_pt}{\bf u}_0\in{\bf L}^q(\Om)$ and
\EQN{
\norm{e^{-A_pt}{\bf u}_0}_{{\bf L}^q(\Om)} 
&\le C\norm{e^{-A_pt}{\bf u}_0}_{{\bf L}^{p_0}(\Om)}^\th \norm{e^{-A_pt}{\bf u}_0}_{{\bf L}^p(\Om)}^{1-\th}
\le C\norm{A_p^s \bke{e^{-A_pt}{\bf u}_0} }_{{\bf L}^p(\Om)}^\th \norm{e^{-A_pt}{\bf u}_0}_{{\bf L}^p(\Om)}^{1-\th}\\
&\le C (e^{-\de t} t^{-s})^\th (e^{-\de t})^{1-\th} \norm{{\bf u}_0}_{{\bf L}^p(\Om)}\\
&= C\, e^{-\de t} t^{-\th s} \norm{{\bf u}_0}_{{\bf L}^p(\Om)} = C\, e^{-\de t} t^{-\bke{1/p-1/q}} \norm{{\bf u}_0}_{{\bf L}^p(\Om)},
}
proving \eqref{eq-ghosh-3.65}.

The estimates \eqref{eq-ghosh-3.66} and \eqref{eq-ghosh-3.67} follows from the same proof for \cite[(3.66), (3.67)]{Ghosh-thesis2018}. We leave out the details for the sake of brevity.
\end{proof}

\medskip

\subsection{Local and global existence}
~

Invoking Lemma \ref{sect3-semigroup-estimate-pre}, we study the local-in-time existence and summarize the results as follows:
\begin{lemma}\label{sect3-lemma-local-in-time-1}
Suppose $(n_0,c_0,{\bf u}_0)\in C^{0}(\bar\Omega)\times W^{1,\infty}(\Omega)\times {\bf{D}}(A_2)$ with ${\bf{D}}(A_2)$ defined by \eqref{sect2-semigroup-domain}; $n_0$ and $c_0$ are nonnegative but not identically equal to zero in $\Omega.$  Then there exists $T\leq \infty$ such that the unique pair $(n,c,u,P)$ with positive $n$ and $c$ solve (\ref{PKSNS-time-dependent}) in a classical sense.  Moreover, if $T<\infty,$ we have
 \begin{align}\label{eq-in-lemma3.1}
    \lim_{t\rightarrow T^-}\big(\Vert n\Vert_{L^{\infty}(\Omega)}+\Vert c\Vert_{W^{1,\infty}(\Omega)}+\Vert A^{\alpha}_2{\bf u}\Vert_{W^{1,q}(\Omega)}\big)=\infty,
 \end{align}
for some $q>2.$
\end{lemma}
\begin{proof}
The proof is shown in Appendix \ref{appen-local-in-time}.   
\end{proof}
With the local existence result, we focus on the global well-posedness of (\ref{PKSNS-time-dependent}).  A useful observation is the mass conservation of cellular density $n$.  Indeed, we have from the integration by parts that
\begin{align}\label{freen1bound}
\frac{d}{dt}\int_{\Omega} n\, dx=\int_{\Omega} n_t\, dx=\int_{\Omega} {\bf u}\cdot \nabla n\, dx=-\int_{\Omega} (\nabla\cdot {\bf u}) n\, dx+\int_{\partial\Omega} {\bf u}\cdot \boldsymbol{\nu}\,  dS=0.
\end{align}
{An immediate consequence of (\ref{freen1bound}) is the boundedness of $c$ in $L^1$.  In fact, we integrate the $c$-equation by parts and obtain that
\begin{align*}
\frac{d}{dt}\int_{\Omega }c\, dx+\int_{\Omega} c\, dx=\int_{\Omega} n\, dx \lesssim 1.
\end{align*}
By solving the Gr\"{o}nwall's inequality, one has $\Vert c\Vert_{L^1(\Omega)}\lesssim 1.$}  Summarizing the results above, we have
\begin{lemma}
Let $(n,c,u)$ be the solution of time-dependent system (\ref{PKSNS-time-dependent}) and the conditions in Theorem \ref{thm-global-existence-PKS-NS} hold.  Then we have
\begin{align*}
\int_{\Omega} n\, dx+\int_{\Omega} c\,dx \lesssim1.
\end{align*}
\end{lemma}
Now, we state several useful preliminary lemmas for the arguments below on the proof of global well-posedness.
\begin{lemma}[ Cf. Lemma 2.4 in \cite{liu2015global} ]\label{sect3-preliminary-lemma33}
Assume that $\Omega\subset\R^2$ is a bounded domain with the smooth boundary.  Let $p\in(1,\infty)$ and $r\in(0,p).$  Then there exists $C>0$ such that for any $\delta>0$,
\begin{align*}
\Vert f\Vert_{L^p}^p\leq \delta\Vert \nabla f\Vert_{L^2}^{p-r}\Vert f\log |f|\Vert_{L^r}^r+C\Vert f\Vert_{L^r}^p+C_{\delta},
\end{align*}
where $f\in W^{1,2}(\Omega)$ and constant $C_{\delta}>0$ depending on $\delta.$
\end{lemma}
\begin{lemma}[ Cf. Lemma 2.2 in \cite{chenkongwang2021global} ]\label{sect3-pre-chenkongwang2021global}
Let $(n,c,u,P)$ be the solution of (\ref{PKSNS-time-dependent}) and assume constant $\iota>0.$  Then there exists constant $C$ depending on $\Vert\nabla c_0\Vert_{L^q}$ and $|\Omega|$ such that
\begin{align*}
\Vert c(\cdot,t)\Vert_{W^{1,q}}\leq C\Big(1+\sup_{s\in(0,t)}\Vert n(\cdot,s)\Vert_{L^p}\Big),
\end{align*}
where $q\in\Big[1,\frac{Np}{N-p}\Big)$ if $p\in[1,N)$; $q\in[1,\infty)$ if $p=N$ and $q=\infty$ if $p>N.$
\end{lemma}
\begin{proof}
We follow the argument in \cite{chenkongwang2021global} to rewrite the $c$-equation as the abstract form then perform the heat Neumann semigroup estimate to prove the lemma.  Noting the steps are the same as in \cite{chenkongwang2021global}, we omit the details.
\end{proof}
\begin{lemma}[ Cf. \cite{nirenberg1959elliptic} ]\label{sect3-pre-GN-inequality}
Let $\Omega\subset R^N$, $N\geq 1$, be a bounded smooth domain.  Let $j \geq 0$, $k \geq 0$ be integers and $p$, $q$, $r$, $s>1$.  Then there is a constant $C > 0$ such that for any function $w \in L^q(\Omega)\cap L^s(\Omega)$ with $D^{k}w\in L^r(\Omega)$ such that
\begin{align*}
\Vert D^jw\Vert_{L^p}\leq C\Vert D^k w\Vert_{L^r}^{\alpha}\Vert w\Vert_{L^q}^{1-\alpha}+C\Vert w\Vert_{L^s},
\end{align*}
where $\frac{1}{p}=\frac{j}{N}+\big(\frac{1}{r}-\frac{k}{N})\alpha+\frac{1-\alpha}{q}$ with $\frac{j}{k}\leq \alpha <1.$
\end{lemma}
We shall follow the ideas in \cite{nagai1997application1} to show the global existence of the solution.  Firstly, we cite the following Moser-Trudinger type's inequalities given in \cite{nagai1997application1}:
\begin{lemma}\label{lemma2-free}
Assume $\Omega\subset \mathbb R^2$ with the smooth boundary $\partial \Omega$, then for any small $\epsilon>0$ and $w\in W^{1,2}(\Omega),$
\begin{align*}
\int_{\Omega} e^{|w|}\, dx\leq C_{\Omega}\exp\Big\{\Big(\frac{1}{8\theta_{\Omega}}+\epsilon\Big)\Vert \nabla w\Vert_{L^2(\Omega)}^2+\frac{2}{|\Omega|}\Vert w\Vert_{L^{1}(\Omega)}\Big\},
\end{align*}
where $C_{\Omega}$ is some positive constant depending on $|\Omega|$ and $\theta_{\Omega}$ denotes the minimum interior
angle at the vertices of $\Omega$.  In particular $\theta_{\Omega}=\pi$ if there is no corner on $\Omega.$
\end{lemma}
\begin{lemma}\label{lemma2-free1}
Let $\Omega=\{x\in \mathbb R^N:\vert x\vert\leq R\}$ $(N\geq 2)$ and $w\in W^{1,N}(\Omega)$ with $w=w(|x|).$  Then for any $\epsilon>0$, there exists $C=C(|\Omega|,\epsilon)$ such that 
\begin{align*}
\int_{\Omega} e^{|w|}dx\leq C\exp\Big\{\Big(\frac{1}{\beta_N}+\epsilon\Big) \Vert\nabla w\Vert_{L^N(\Omega)}^N+\frac{2^N}{N|\Omega|}\Vert w\Vert_{L^1(\Omega)}\Big\},
\end{align*}
where $\beta_N$ is given by
\begin{align*}
\beta_N=N\Big(\frac{N\alpha_N}{N-1}\Big)^{N-1},~~\alpha_N=N\omega_{N-1}^{1/(N-1)},
\end{align*}
and $\omega_{N-1}$ denotes the surface area of the unit sphere in $\R^N.$
\end{lemma}

By using Lemma \ref{lemma2-free} and Lemma \ref{lemma2-free1}, we have from the energy dissipation given in Lemma \ref{lemma1-free} that
\begin{lemma}\label{free-lemma3}
Assume condition \eqref{sect0-M-M0-less} holds.  Then $(n,c,{\bf u},P)$ satisfies
\begin{align*}
\int_{\Omega} n c\, dx\lesssim 1,~~|\mathcal J(t)|\lesssim 1.
\end{align*}
\end{lemma}
\begin{proof}
Firstly, we have the free-energy $\mathcal J(t)$ given by (\ref{sect0-free-energy-calculate}) can be rewritten as
{{\begin{align}\label{lemma23-1}
\mathcal J(t)=\int_{\Omega} n\log n\, dx-(1+\delta)\int_{\Omega} nc\, dx+\int_{\Omega}[ \delta nc+\frac{1}{2}(|\nabla c|^2+c^2)]\, dx+\frac{1}{2}\int_{\Omega}\vert {\bf u}\vert^2\, dx,
\end{align}}}
where $\delta>0$ is any small number.  From (\ref{lemma23-1}), we find
{{\begin{align*}
\mathcal J(t)\geq-\int_{\Omega} n\log \frac{e^{(1+\delta)c}}{n}\, dx+\int_{\Omega} \delta nc\, dx +\frac{1}{2}\int_{\Omega}(|\nabla c|^2+c^2)\, dx.
\end{align*}}}
On the other hand, noting $n$ satisfies (\ref{freen1bound}), one has
$$\int_{\Omega} \frac{n}{M}\, dx=1.$$
Then by using the fact that $-\log x$ is convex, we apply Jensen's inequality to obtain
\begin{align}\label{sect2-jesen-after}
-\log \Big\{\frac{1}{M}\int_{\Omega} e^{(1+\delta)c}\, dx\Big\}=-\log \int_{\Omega} \frac{e^{(1+\delta)c}}{n}\frac{n}{M}\, dx\leq &\int_{\Omega} \bigg(-\log\frac{e^{(1+\delta)c}}{n}\bigg)\frac{n}{M}\,dx\nonumber\\
=&-\frac{1}{M}\int_{\Omega} n\log\frac{e^{(1+\delta )c}}{n}\, dx.
\end{align}
In light of Lemma \ref{lemma2-free}, one gets
\begin{align}\label{Sect2-global-tmineq}
\int_{\Omega} e^{(1+\delta) c}\, dx\lesssim \exp\Big\{\Big(\frac{1}{2M_0}+\epsilon\Big)(1+\delta)^2\Vert \nabla c\Vert_{L^2(\Omega)}^2+\frac{2(1+\delta)}{|\Omega|}\Vert c\Vert_{L^1(\Omega)}\Big\},
\end{align}
where $\epsilon>0$ is small.  Upon substituting (\ref{Sect2-global-tmineq}) into (\ref{sect2-jesen-after}), we arrive at
\begin{align*}
\log\Big(\frac{1}{M}\int_{\Omega} e^{(1+\delta) c}\, dx\Big)\lesssim \log \frac{1}{M} +\Big(\frac{1}{2M_0}+\epsilon\Big)(1+\delta)^2\Vert \nabla c\Vert_{L^2}^2+\frac{2(1+\delta)}{|\Omega|}\Vert c\Vert_{L^1},
\end{align*}
which yields 
\begin{align}\label{freeenergy-eq111}
    \mathcal J(t)\geq& -M\log \Big\{\frac{1}{M}\int_{\Omega} e^{(1+\delta)c}\, dx\Big\}+\int_{\Omega} \Big[\delta nc +\frac{1}{2}\Big(|\nabla c|^2+c^2\Big)\Big]\, dx\nonumber\\
    \gec&  -M\Big[\log \frac{1}{M} +\Big(\frac{1}{2M_0}+\epsilon\Big)(1+\delta)^2\Vert \nabla c\Vert_{L^2}^2+\frac{2(1+\delta)}{|\Omega|}\Vert c\Vert_{L^1}\Big]+\int_{\Omega} \Big[\delta nc +\frac{1}{2}(|\nabla c|^2+c^2)\Big].
\end{align}
We rearrange (\ref{freeenergy-eq111}) to find
\begin{align}\label{Sect2-readily-obtain-before}
&\Big\{\frac{1}{2}-M\Big(\frac{1}{2M_0}+\epsilon\Big)(1+\delta)^2\Big\}\Vert \nabla c\Vert_{L^2}^2+\delta \int_{\Omega} nc\, dx\nonumber\\
\leq & M\Big\{\log \frac{1}{M}+\frac{2(1+\delta)}{|\Omega|}\Vert c\Vert_{L^1}\Big\}+\mathcal J(t)\lec 1+\mathcal J(0)\lec 1.
\end{align}
To obtain the boundedness of $\Vert \nabla c\Vert_{L^2(\Omega)}$ and $\Vert nc\Vert_{L^1(\Omega)}$, we require for sufficiently small $\epsilon,\delta>0$ that
\begin{align}\label{sect2-before-take-limit}
\frac{1}{2}-M\Big(\frac{1}{2M_0}+\epsilon\Big)(1+\delta)^2>0.
\end{align}
Letting $\epsilon\rightarrow 0^+$ and $\delta\rightarrow 0^+,$ one finds (\ref{sect2-before-take-limit}) becomes
$$M<M_0.$$
With this condition, we readily obtain from (\ref{Sect2-readily-obtain-before}) that
\begin{align*}
 \int_{\Omega} n\log n\, dx\lesssim 1,~~~\int_{\Omega} \vert \nabla c\vert^2\,  dx\lesssim 1.
\end{align*}
This completes the proof of our lemma.
\end{proof}
It is worthy mentioning that $\int_{\Omega}n \log n\, dx$ is bounded from below.  Indeed, according to the lower bound of function $x\log x$, one has
$$\int_{\Omega} n\log n \, dx \geq -\frac{1}{e}|\Omega|.$$
With the help of key Lemma \ref{free-lemma3}, we can show the boundedness of $\Vert n\Vert_{L^2}$, which is
\begin{lemma}\label{sect3-lemma-n2-bounded}
Assume the condition (\ref{sect0-M-M0-less}) holds.  Then we have
\begin{align*}
{\int_{\Omega} n^2dx\lesssim 1.}
\end{align*}
\end{lemma}
\begin{proof}
Test the $n$-equation in (\ref{PKSNS-time-dependent}) against $n$, then we integrate it by parts and apply the $c$-equation to get
\begin{align*}
\frac{1}{2}\frac{d}{dt} \int_{\Omega}n^2\, dx=&-\int_{\Omega}|\nabla n|^2\, dx-\int_{\Omega}n\nabla \cdot (n\nabla c)\, dx-\int_{\Omega} {\bf u}\cdot\nabla n n\, dx \\
=&-\int_{\Omega}|\nabla n|^2\, dx-\int_{\Omega} {\bf u}\cdot \nabla n n\, dx+\int_{\Omega} n\nabla n\cdot\nabla c\, dx\\
=&-\int_{\Omega}|\nabla n\vert^2\, dx-\int_{\Omega}{\bf u}\cdot\nabla n n\, dx-\frac{1}{2}\int_{\Omega} n^2\Delta c\, dx\\
=&-\int_{\Omega}|\nabla n|^2\, dx-\int_{\Omega}{\bf u}\cdot \nabla n n\, dx-\frac{1}{2}\int_{\Omega}n^2\Big(\iota\frac{\pd c}{\pd t}+c-n\Big)\, dx\\
\leq &-\int_{\Omega}|\nabla n|^2\, dx-\frac{\iota}{2}\int_{\Omega} n^2 \frac{\pd c}{\pd t}\, dx+\frac{1}{2}\int n^3\, dx,
\end{align*}
where we have used the boundary conditions of ${\bf u}.$  Thanks to Lemma \ref{sect3-preliminary-lemma33}, one has $n$ satisfies
\begin{align}\label{free-eqGn}
\Vert n\Vert_{L^3}\leq \epsilon \Vert\nabla n\Vert_{L^2}^{2/3}\Vert n\log n\Vert_{L^1}^{1/3}+C_{\epsilon}\Big(\Vert  n\log n\Vert_{L^1}+\Vert n\Vert_{L^1}^{1/3}\Big),
\end{align}
where $\epsilon>0$ is any small constant.  Invoking the H\"{o}lder's inequality and Lemma \ref{sect3-pre-GN-inequality}, we obtain
\begin{align}\label{free-eq2}
\int_{\Omega} \Big|n^2 \frac{\pd c}{\pd t}\Big|\, dx\leq &\Big\Vert \frac{\pd c}{\pd t}\Big\Vert_{L^2}\Vert n\Vert_{L^4}^2\nonumber\\
\lesssim & \Big\Vert\frac{\pd c}{\pd t}\Big\Vert_{L^2}\big[(\Vert \nabla n\Vert_{L^2}^{1/2}+\Vert n\Vert_{L^2}^{1/2})\Vert n\Vert_{L^2}^{1/2}\big]^2\nonumber\\
\lesssim & \Big\Vert \frac{\pd c}{\pd t}\Big\Vert_{L^2}(\Vert \nabla n\Vert_{L^2}\Vert n\Vert_{L^2}+\Vert n\Vert_{L^2}^2)\nonumber\\
\leq &\epsilon \Vert \nabla n\Vert_{L^2}^2+C_{\epsilon}\Big(\Big\Vert \frac{\pd c}{\pd t}\Big\Vert_{L^2}^2+\Big\Vert \frac{\pd c}{\pd t}\Big\Vert_{L^2}\Big)\Vert n\Vert_{L^2}^2.
\end{align}
Combining (\ref{free-eqGn}) and (\ref{free-eq2}), one gets
\begin{align}\label{free-eqq}
\frac{d}{dt}\Vert n\Vert_{L^2}^2+2\Vert \nabla n\Vert_{L^2}^2\leq &\epsilon \iota \Vert \nabla n\Vert_{L^2}^2+\iota\Big(\Big\Vert \frac{\pd c}{\pd t }\Big\Vert_{L^2}^2+\Big\Vert \frac{\pd c}{\pd t}\Big\Vert_{L^2}\Big)\Vert n\Vert_{L^2}^2\nonumber\\
+&\epsilon^3\Vert \nabla n\Vert_{L^2}^2\Vert n\log n\Vert_{L^1}+C_{\epsilon}(\Vert n\log n\Vert_{L^1}^3+\Vert n\Vert_{L^1}).
\end{align}
We rearrange (\ref{free-eqq}) to arrive at
\begin{align*}
&\frac{d}{dt}\Vert n\Vert_{L^2}^2+(2-\epsilon\iota-\epsilon^3\Vert n\log n\Vert_{L^1})\Vert\nabla n\Vert_{L^2}^2\\
\leq &C_{\epsilon}\Big[\iota\Big(\Big\Vert \frac{\pd c}{\pd t}\Big\Vert_{L^2}^2+\Big\Vert \frac{\pd c}{\pd t}\Big\Vert_{L^2}\Big)\Vert n\Vert_{L^2}^2+\Vert n\log n\Vert_{L^1}^3+\Vert n\Vert_{L^1}\Big].
\end{align*}
In light of Lemma \ref{free-lemma3}, we have $\Vert n\log n\Vert_{L^1}$ is bounded.  As a consequence, we choose $\epsilon$ small such that
\begin{align}\label{boundbefore}
  &\frac{d}{dt}\Vert n\Vert_{L^2}^2+\Vert\nabla n\Vert_{L^2}^2\nonumber\\
\leq &C_{\epsilon}\Big[\iota\Big(\Big\Vert \frac{\pd c}{\pd t}\Big\Vert_{L^2}^2+\Big\Vert \frac{\pd c}{\pd t}\Big\Vert_{L^2}\Big)\Vert n\Vert_{L^2}^2+\Vert n\log n\Vert_{L^1}^3+\Vert n\Vert_{L^1}\Big].  
\end{align}
On the other hand, one can conclude from Lemma \ref{sect3-pre-GN-inequality} that
\begin{align*}
(\iota+1)\Vert n\Vert_{L^2}\leq \bar C_1(\iota+1)\Big(\Vert \nabla n\Vert_{L^2}^{1/2}\Vert n\Vert_{L^1}^{1/2}+\Vert n\Vert_{L^1}\Big),
\end{align*}
where ${\bar C}_1>0$ is some constant only depending on $|\Omega|$.  It follows that
%\begin{align*}
%\Vert n\Vert^2_{2}\leq& C\Big(\Vert \nabla n\Vert_2\Vert n\Vert_{1}+\Vert n\Vert^2_{1}\Big)\\
%\leq &\Big(\Vert \nabla n\Vert_2^2+C\Vert n\Vert^2_{2}\Big).
%\end{align*}
%Hence,
\begin{align}\label{sect3-lemma-GN-use1}
\Vert\nabla n\Vert_{L^2}^2\geq (\iota+1)\Vert n\Vert_{L^2}^2-\bar C_2 \Vert n\Vert_{L^1}^2,
\end{align}
where positive constant ${\bar C}_2$ only depends on $|\Omega|$ and $\iota$.  Upon substituting (\ref{sect3-lemma-GN-use1}) into (\ref{boundbefore}), one finds
\begin{align}\label{sect3-differential-inequality-1}
\frac{d}{dt}\Vert n\Vert_{L^2}^2+(\iota+1)\Vert n\Vert_{L^2}^2\leq \bar C_3\Big[\iota\Big(\Big\Vert \frac{\pd c}{\pd t}\Big\Vert_{L^2}^2+\Big\Vert \frac{\pd c}{\pd t}\Big\Vert_{L^2}\Big)\Vert n\Vert_{L^2}^2+B\Big],
\end{align}
where positive constant $\bar C_3$ only depends on $|\Omega|$, $\iota$ and $\epsilon$; moreover, $B$ is defined by
\begin{align*}
B:=\sup_{t>0}(\Vert n\log n\Vert_{L^1}^3+\Vert n\Vert_{L^1}+\Vert n\Vert_{L^1}^2).
\end{align*}
Noting that energy dissipation (\ref{dissipation}) gives us
\begin{align}\label{boundedl2n2}
\iota\int_0^t\Big\Vert \frac{\pd c}{\pd t}(\cdot,s)\Big\Vert_{2}^2ds\lesssim 1,
\end{align}
then we define $y_1(t):=\Vert n(\cdot, t)\Vert_{L^2}^2$ and $y_2(t)=\big\Vert \frac{\pd c}{\pd t}(\cdot,t)\big\Vert_{L^2}^2$, then rewrite (\ref{sect3-differential-inequality-1}) as the following differential inequality:
\begin{align}\label{diffinequal}
    y_1'(t)+(\iota+1) y_1(t)\lesssim \iota(y_2+\sqrt{y_2})y_1+B.
\end{align}
By Young's inequality, we have $y_2$ satisfies
\begin{align*}
\sqrt{y_2}\leq \frac{1}{2}+\frac{y_2}{2}.
\end{align*}
Substituting it into (\ref{diffinequal}), one finds
\begin{align}\label{sect-combine-1-ineq}
y_1'+\Big(\frac{\iota}{2}+1-\frac{3\iota}{2}y_2\Big)y_1\leq B.
\end{align}
Combining (\ref{sect-combine-1-ineq}) with (\ref{boundedl2n2}), we obtain that
$$y_1(t)\lesssim 1,$$
i.e. $\Vert n(\cdot,t)\Vert_{L^2}$ is uniformly bounded in time.
\end{proof}
Lemma \ref{sect3-lemma-n2-bounded} implies $\Vert n(\cdot,t)\Vert_{L^2}$ is bounded.  Then if $\iota=0,$ one has for any $q\in[1,+\infty),$ $c\in W^{1,q}$ by the standard elliptic estimate.  Moreover, if $\iota>0,$ we find from Lemma \ref{sect3-pre-chenkongwang2021global} that
\begin{align*}
\Vert c(\cdot,t)\Vert_{W^{1,q}}\lesssim 1, \ \ \forall q\in[1,+\infty).
\end{align*}
With the boundedness of $\Vert n\Vert_{L^2}$ and $\Vert c\Vert_{W^{1,q}}$, we next prove the boundedness of $\Vert n\Vert_{L^\infty}$ and $\Vert c\Vert_{W^{1,\infty}}$, which are
\begin{lemma}\label{lemma310-sect3}
Assume that all conditions in Theorem \ref{thm-global-existence-PKS-NS} hold. 
 Then we have
\begin{align}\label{sect3-ncomponent-estimate}
\Vert n(\cdot,t)\Vert_{L^\infty}+\Vert c(\cdot,t)\Vert_{W^{1,\infty}}\lesssim 1.
\end{align}
\end{lemma}
\begin{proof}
For $p>1$, we multiply the $n$-equation in (\ref{PKSNS-time-dependent}) by $n^{p-1}$, then apply the integration by parts and the divergence-free property of $\bf u$ to get
\begin{align}\label{sect3-after-young-substitute-before-0}
\frac{1}{p}\cdot\frac{d}{dt}\int_{\Omega} n^{p}\, dx=&-(p-1)\int_{\Omega}n^{p-2}|\nabla n|^2\, dx+(p-1)\int_{\Omega} n^{p-1}\nabla n\cdot \nabla c\, dx-\int_{\Omega} n^{p-1}{\bf u}\cdot\nabla n\, dx\nonumber\\
=&-\frac{4(p-1)}{p^2}\int_{\Omega}|\nabla n^{\frac{p}{2}}|^2\, dx+(p-1)\int_{\Omega} n^{p-1}\nabla n\cdot \nabla c\, dx, 
\end{align}
where we have used the boundary conditions of ${\bf u}.$
%\begin{align*}
%p\int_{\Omega} n^p\nabla n\cdot \nabla c=\frac{p}{p+1}\int_{\Omega}\nabla n^{p+1} \cdot \nabla c
%\end{align*}
On the other hand, by Young's inequality, one has
\begin{align}\label{sect3-after-young-substitute-before}
\int_{\Omega} n^{p-1}\vert \nabla n\vert \vert \nabla c\vert \, dx\leq &\frac{p-1}{2}\int_{\Omega} n^{p-2}\vert \nabla n\vert^2\, dx+C_{p}\int_{\Omega}  n^{p}\vert \nabla c\vert^2\, dx\nonumber\\
\leq &\frac{2(p-1)}{p^2}\int_{\Omega}|\nabla n^{\frac{p}{2}}|^2\, dx+\frac{1}{p}\int_{\Omega} n^{p+1}\, dx+C_p\int_{\Omega} \vert\nabla c\vert^{2(p+1)}\, dx.
\end{align}
In light of $c\in W^{1,q}$ for any $q\in[1,\infty)$, we find
\begin{align}\label{combine1}
\int_{\Omega} \vert\nabla c\vert^{2(p+1)}\, dx\leq C_p.
\end{align}
On the other hand, by using Gagliardo-Nirenberg's inequality given in Lemma \ref{sect3-pre-GN-inequality}, one has
\begin{align}\label{combine2}
2\int_{\Omega} n^{p+1}\, dx\leq \frac{2(p-1)}{p}\int_{\Omega}|\nabla n^{\frac{p}{2}}|^2\, dx+C_p.
\end{align}

We collect (\ref{sect3-after-young-substitute-before-0}), (\ref{sect3-after-young-substitute-before}), (\ref{combine1}) and (\ref{combine2}) to arrive at
\begin{align*}
\frac{d}{dt}\int_{\Omega} n^p\, dx+\int_{\Omega} n^p\, dx\leq C_p,
\end{align*}
where $p>1$ is arbitrary.  By the standard Moser--Alikakos iteration, we obtain that $n(\cdot,t)$, $\nabla c(\cdot,t)\in L^{\infty}(\Omega).$
\end{proof}

\medskip

{\bf{Proof of Theorem \ref{thm-global-existence-PKS-NS}:}}\\
With the help of Lemma \ref{lemma310-sect3}, it is left to study the estimate of velocity ${\bf u}$.  Firstly, we define $\omega:=\text{curl}~{\bf u}$ and obtain from (\ref{PKSNS-time-dependent})$_3$ that the vorticity equation is
\begin{align}\label{omegaeq}
\partial_{t}\omega+({\bf u}\cdot \nabla)\omega=\Delta \omega+\nabla\times (n\nabla c),
\end{align}
where $\nabla \cdot {\bf u}=0$ in $\Omega$.  Next, we estimate the vorticity $\omega$.  To this end, multiplying the (\ref{omegaeq}) by $\omega$, we integrate it by parts to obtain
\begin{align}\label{sect3-final-rearrange-before}
\frac{1}{2}\frac{d}{dt}\Vert \omega\Vert_{L^2}^2+\Vert \nabla \omega\Vert_{L^2}^2=\int_{\Omega} \nabla\times (n\nabla c)\omega\, dx =-\int_{\Omega} n\nabla c\cdot \nabla^{\perp}\omega\, dx.
\end{align}
We integrate (\ref{sect3-final-rearrange-before}) over $(0,t)$ and apply the Young's inequality to find 
\begin{align*}
\frac{1}{2}\Vert\omega\Vert_{L^2}^2+\int_0^t\Vert\nabla \omega\Vert_{L^2}^2\, d\tau=&\frac{1}{2}\Vert\omega_0\Vert_{L^2}^2-\int_0^t\int_{\Omega}n\nabla c\cdot (\nabla^{\perp} \omega)dxd\tau\\
\leq &\frac{1}{2}\Vert\omega_0\Vert_{L^2}^2+\frac{1}{2}\int_0^t ({\Vert n\nabla c\Vert_{L^2}^2}+\Vert\nabla^{\perp}\omega\Vert_{L^2}^2)\, d\tau,
\end{align*}
which implies
\begin{align*}
\Vert \omega\Vert_2^2+\int_0^t\Vert\nabla \omega \Vert_2^2\, d\tau\leq \Vert \omega_0\Vert_2^2+C(T),
\end{align*}
where we have used (\ref{sect3-ncomponent-estimate}) and positive constant $C$ satisfies $C(T)\rightarrow \infty$ as $T\rightarrow \infty.$  Moreover, noting that $\Vert\nabla {\bf u}\Vert_{{\bf L}^2}=\Vert \nabla\times {\bf u}\Vert_{{\bf L}^2}=\Vert\omega\Vert_{{\bf L}^2},$ we have
\begin{align}\label{sect3-combine-final-before-1-global-exsitence}
\Vert\nabla {\bf u}\Vert_{{\bf L}^2}\leq C(T).
\end{align}
By using the Gagliardo-Nirenberg-Sobolev's inequality, i.e. Lemma \ref{sect3-pre-GN-inequality}, we obtain
\begin{align}\label{sect3-combine-final-before-1-global-exsitence-2}
\Vert {\bf u}\Vert_{{\bf L}^\infty}\lesssim (\Vert {\bf u}\Vert_{{\bf L}^2}^{1-\alpha}\Vert\nabla {\bf u}\Vert_{{\bf L}^2}^{\alpha}+\Vert {\bf u}\Vert_{{\bf L}^2}),
\end{align}
where $\alpha\in(0,1)$.  Due to energy dissipation (\ref{dissipation}), one has ${\bf u}\in  {\bf L}^2$.  Hence, combining (\ref{sect3-combine-final-before-1-global-exsitence}) and (\ref{sect3-combine-final-before-1-global-exsitence-2}), we finish the proof of Theorem \ref{thm-global-existence-PKS-NS}.
\qed

\medskip

\begin{remark}
~

Lemma \ref{lemma310-sect3} establishes the uniformly-in-time boundedness of $\Vert n\Vert_{L^\infty}$ and $\Vert c\Vert_{W^{1,\infty}}$.  However, due to the growth bound of $\Vert {\bf u}\Vert_{{\bf L}^\infty}$ as $t\rightarrow T^{-}$, we can not rule out the infinite time blow-up of (\ref{PKSNS-time-dependent}).  We conjecture that velocity ${\bf u}$ possesses the uniformly-in-time bound in the subcritical mass case. 
\end{remark}
Theorem \ref{thm-global-existence-PKS-NS} demonstrates that (\ref{PKSNS-time-dependent}) admits the global-in-time solution with the subcritical mass.  Next, in Section \ref{sect2}--\ref{inn-out-gluing-sect}, we shall focus on the critical mass case and construct the stationary solution with the striking structure.

\section{The Choice of Ansatz and Error Estimates}\label{sect2}
First of all, we borrow the idea from \cite{del2006collapsing} to determine the approximate solution of (\ref{PKSNS-ss-equiv}).  Next, we reduce the transported Keller-Segel system as the single equation with a nonlocal term.  To begin with, we find with the absence of transport term $u\cdot \nabla n,$ the $n$-equation of (\ref{PKSNS-ss-equiv}) implies
\begin{align}\label{sect2-nc-relation}
n=\varepsilon^2 e^{c},
\end{align}
 where $\varepsilon>0$ is any small constant.  Upon substituting (\ref{sect2-nc-relation}) into the $c$-equation, one has the standard minimal Keller-Segel model is reduced as 
\begin{align}\label{sect2-standard-reduced}
\left\{\begin{array}{ll}
    \Delta c-c+\varepsilon^2 e^c=0,&x\in\Omega,\\
    \frac{\partial c}{\partial\boldsymbol{\nu}}=0,&x\in\partial\Omega.
    \end{array}
    \right.
\end{align}

In this paper, we plan to construct the boundary spot, i.e. the location is assumed to be at the boundary $\partial\Omega$.  In fact, similarly as shown in \cite{del2006collapsing} and \cite{KWX2022}, the construction of boundary spot is the modification of interior counterpart.  Thus, we first give the ansatz of single interior spot and compute its error.  Assume that the solution is concentrated at $\xi\in \Omega,$ then we introduce the stretched variable $y=\frac{x-\xi}{\varepsilon}$, define $c(x)=\bar c(\xi+\varepsilon y)-4\log \varepsilon$ and obtain from (\ref{sect2-standard-reduced}) that the limiting problem of $\bar c$ is 
$$\Delta_y \bar c+e^{\bar c}=0,~~y\in\mathbb R^2.$$
After imposing the following integral constraint
$$\int_{\mathbb R^2}e^{\bar c}dy<+\infty,$$
we arrive at
\begin{align}\label{sect2-limiting-problem-barc}
\left\{\begin{array}{ll}
\Delta_y \bar c+e^{\bar c}=0,&y\in\mathbb R^2,\\
\int_{\mathbb R^2}e^{\bar c}dy<+\infty.
\end{array}
\right.
\end{align}
It is well-known that (\ref{sect2-limiting-problem-barc}) has a family of solutions as follows:
\begin{align}\label{sect2-barcmu}
\bar c_{\mu}=\log\frac{8\mu^2}{(\mu^2+|y |^2)^2},~~\mu>0.
\end{align}
By using \eqref{sect2-barcmu}, we find the ``rough" ansatz of $c$ is given by
\begin{align}\label{sect2-c0-leading}
c_0=\Gamma\Big( \frac{ x-\xi}{\varepsilon}\Big)-4\log\varepsilon, \ \ \Gamma=\log\frac{8\mu^2}{(\mu^2+ \vert y\vert^2)^2}, \ \ y=\frac{x-\xi}{\e},
\end{align}
where $\xi$ represents the location of the local spot.  Moreover, noting that $n=\e^2e^{c}$, one finds from (\ref{sect2-c0-leading}) that the basic ansatz $n_0$ of $n$ is 
\begin{align}\label{sect2-n0-W0-leading}
n_0=\frac{1}{\varepsilon^2}W\Big(\frac{x-\xi}{\varepsilon}\Big),~~W=\frac{8\mu^2}{(\mu^2+|y|^2)^2}.
\end{align}
Since $c$ must satisfy the Neumann boundary condition, we set $H^{\varepsilon}$ as the correction term of $c_0$, which is determined by
\begin{align}\label{sect2-Hvarepsilon-eq}
\left\{\begin{array}{ll}
-\Delta H^{\varepsilon}+H^{\varepsilon}=-\Gamma,&x\in\Omega,\\
\frac{\partial H^{\varepsilon}}{\partial\boldsymbol{\nu}}=-\frac{\partial\Gamma}{\partial\boldsymbol{ \nu}},&x\in\partial\Omega.
\end{array}
\right.
\end{align}
It is easy to find that there exists the solution to (\ref{sect2-Hvarepsilon-eq}) satisfying $H^{\varepsilon}\in C^{1,\alpha}$ in 2D.  We summarize and choose the approximate solution of $(n,c)$ as 
\begin{align}\label{sect2-n-and-c}
n=\frac{1}{\varepsilon^2}W+\phi,~~c=-4\log \varepsilon+\Gamma+H^{\varepsilon}+\psi,
\end{align}
where $(\phi,\psi)$ is the remainder term. 

We have obtained the desired approximate solution defined in \eqref{mainn} and \eqref{mainc}.  Next, we shall compute and analyze the error generated by (\ref{sect2-n-and-c}).  Similarly as in \cite{KWX2022}, the coupled $n$-equation and $c$-equation can be reduced as the single form $S(n)=0$, where $S(n)$ is given by
$$S(n):=\Delta_x n+\nabla_x\cdot( n\nabla (\Delta_x-1)^{-1}n)-{\bf u}\cdot \nabla_x n.$$
Upon substituting (\ref{sect2-n-and-c}) into $S(n)=0$, we have
\begin{align}\label{S-of-n1}
S(n)=&\frac{1}{\varepsilon^4}\Delta_y W-\frac{1}{\varepsilon^4}\nabla_y\cdot(W\nabla_y\Gamma)\nonumber\\
&-\frac{1}{\varepsilon^3}\nabla_y\cdot\big(W\nabla_x H^{\varepsilon}\big)-\frac{1}{\varepsilon^3}{\bf u}\cdot\nabla_y W\nonumber\\
&+\Delta_x\phi-\frac{1}{\varepsilon^2}\nabla_x\cdot (W\nabla_x\psi)-\nabla_x\cdot(\phi\nabla_x\Gamma)\nonumber\\
&-\nabla_x\cdot(\phi\nabla_x\psi)-\nabla_x\cdot(\phi\nabla_x H^{\varepsilon})\nonumber\\
&-{\bf u}\cdot\nabla_x\phi=0.
\end{align}
Noting that $\Gamma$ and $W$ satisfy (\ref{sect2-c0-leading}) and (\ref{sect2-n0-W0-leading}), one further obtains from (\ref{S-of-n1}) that 
\begin{align*}
S(n)=&-\frac{1}{\varepsilon^3}\nabla_y\cdot\big(W\nabla_x H^{\varepsilon}\big)-\frac{1}{\varepsilon^3}{\bf u}\cdot\nabla_y W\\
&+\Delta_x\phi-\frac{1}{\varepsilon^2}\nabla_x\cdot (W\nabla_x\psi)-\nabla_x\cdot(\phi\nabla_x\Gamma)\\
&-\nabla_x\cdot(\phi\nabla_x\psi)-\nabla_x\cdot(\phi\nabla_x H^{\varepsilon})\\
&-{\bf u}\cdot\nabla_x\phi=0.
\end{align*}
Our goal is to show the existence of $(\phi,\psi)$ via the inner-outer gluing method and fixed point theorem.  To this end, we decompose $\phi(x)$ as
\begin{align}\label{phi-decomposition1}
\phi(x)=\frac{1}{\e^2}\Phi^{\text{i}}(y)\chi(y)+\varphi^{\text{o}},~~y=\frac{x-\xi}{\varepsilon},
\end{align}
where radial function $\chi$ is defined by
%\begin{align}
%\chi(x)=\chi_0\Big(\e|{x-\xi}|\Big),
%\end{align}
%and radial function $\chi_0$ is defined as
\begin{align*}
\chi(r):=\left\{\begin{array}{ll}
1,&r\leq \frac{\delta}{\e},\\
0,&r\geq \frac{2\delta}{\e},
\end{array}
\right.
\end{align*}
with $\delta>0$ is a small constant.  Upon substituting (\ref{phi-decomposition1}) into (\ref{S-of-n1}), we have
\begin{align}\label{sect2-S-of-n-2}
S(n)=&-\frac{1}{\varepsilon^3}\nabla_y\cdot\big(W\nabla_x H^{\varepsilon}\big)-\frac{1}{\varepsilon^3}{\bf u}\cdot\nabla_y W\nonumber\\
&+\frac{1}{\varepsilon^4}\Delta_y\Phi^{\text{i}}\chi-\frac{1}{\varepsilon^4}\nabla_y\cdot(\Phi^{\text{i}}\nabla_y\Gamma)\chi-\frac{1}{\varepsilon^4}\nabla_y\cdot\Big(W\nabla_y\bar\Psi^{\text{i}}\Big)\chi\nonumber\\
&+\Delta_x\varphi^{\text{o}}-\nabla_x\cdot(\varphi^{\text{o}}\nabla_x(\Gamma+ H^{\varepsilon}))-\frac{1}{\varepsilon^2}\nabla_x\cdot(W\nabla_x\psi^{\text{o}})\nonumber\\
&+\frac{2}{\varepsilon^3}\nabla_y\Phi^{\text{i}}\cdot\nabla_x\chi+\frac{1}{\varepsilon^2}\Phi^{\text{i}}\Delta_x\chi-\frac{1}{\varepsilon^4}\Phi^{\text{i}}\nabla_y\Gamma\cdot\nabla_y\chi\nonumber\\
&+\frac{1}{\varepsilon^4}\nabla_y\cdot\big(W\nabla_y\bar\Psi^{\text{i}}\big)\chi-\frac{1}{\varepsilon^4}\nabla_y\cdot(W\nabla_y\Psi^{\text{i}})\chi\nonumber\\
&+\frac{1}{\varepsilon^4}\nabla_y\cdot(W\nabla_y\Psi^{\text{i}})\chi-\frac{1}{\varepsilon^4}\nabla_y\cdot(W\nabla_y\hat\Psi^{\text{i}})\nonumber\\
&-\frac{1}{\varepsilon^2}\nabla_y\cdot\Big(\Big(\frac{1}{\varepsilon^2}\Phi^{\text{i}}\chi+\varphi^{\text{o}}\Big)\nabla_y\psi\Big)-\frac{1}{\varepsilon^3}\nabla_y\cdot\Big(\Phi^{\text{i}}\nabla_x H^{\varepsilon}\chi \Big)\nonumber\\
&-\frac{1}{\varepsilon}{\bf u}\cdot\nabla_y\Big(\frac{1}{\varepsilon^2}\Phi^{\text{i}}\chi+\varphi^{\text{o}}\Big),
\end{align}
where $\psi=-(\Delta_x-1)^{-1}\phi$,
\begin{align}\label{sect3-hatpsi-inn}
\hat\Psi^{\text{i}}:=-\Big(\Delta_y-\varepsilon^2\Big)^{-1}\big(\Phi^{\text{i}}\chi\big),~~ \bar\Psi^{\text{i}}:=-\Delta_y^{-1}\Phi^{\text{i}}
\end{align}
and
\begin{align}\label{sect3-psi-inn-varphi-o}
 \Psi^{\text{i}}:=-\Big(\Delta_y-\varepsilon^2\Big)^{-1}\Phi^{\text{i}},~~\psi^{\text{o}}:=-\Big(\frac{1}{\varepsilon^2}\Delta_y-1\Big)^{-1}\varphi^{\text{o}}.
 \end{align}
Before formulating the inner-outer gluing scheme, we define the inner operator $L_{W}$ and outer operator $L^{\text{o}}$ as
\begin{align}\label{sect2-linearized-inner-operator}
L_{W}[\Phi]=\Delta_y\Phi-\nabla_y\cdot(W\nabla_y \Psi)-\nabla_y\cdot(\Phi\nabla_y\Gamma),~~-\Delta_y^{-1} \Phi=\Psi,
\end{align}
and
\begin{align}\label{sect2-linearized-outer-operator}
L^{\text{o}}[\varphi]=\Delta_x\varphi-\nabla_x\varphi\cdot\nabla_x(\Gamma+ H^{\varepsilon}))-\varphi(\Gamma+H^{\varepsilon}).
\end{align}
By using (\ref{sect2-linearized-inner-operator}) and (\ref{sect2-linearized-outer-operator}), we simplify (\ref{sect2-S-of-n-2}) to get
\begin{align}
&\e\nabla_y\cdot\big(W\nabla_x H^{\varepsilon}\big)+\e {\bf u}\cdot\nabla_y W+{\varepsilon^2}\nabla_x\cdot(W\nabla_x\psi^{\text{o}})\nonumber\\
&-2{\varepsilon}\nabla_y\Phi^{\text{i}}\cdot\nabla_x\chi-{\varepsilon^2}\Phi^{\text{i}}\Delta_x\chi+\Phi^{\text{i}}\nabla_y\Gamma\cdot\nabla_y\chi\nonumber\\
&+{\varepsilon^2}\nabla_y\cdot\Big(\Big(\frac{1}{\varepsilon^2}\Phi^{\text{i}}\chi+\varphi^{\text{o}}\Big)\nabla_y\psi\Big)+{\varepsilon}\nabla_y\cdot\Big(\Phi^{\text{i}}\nabla_x H^{\varepsilon}\chi \Big)\nonumber\\
&+{\varepsilon^3}{\bf u}\cdot\nabla_y\Big(\frac{1}{\varepsilon^2}\Phi^{\text{i}}\chi+\varphi^{\text{o}}\Big)\nonumber\\
&-\nabla_y\cdot\big(W\nabla_y\bar\Psi^{\text{i}}\big)\chi+\nabla_y\cdot(W\nabla_y\Psi^{\text{i}})\chi\nonumber\\
&-\nabla_y\cdot(W\nabla_y\Psi^{\text{i}})\chi+\nabla_y\cdot(W\nabla_y\hat\Psi^{\text{i}})\nonumber\\
=&L_{W}[\Phi^{\text{i}}]\chi+\e^4L^{\text{o}}[\varphi^{\text{o}}]\nonumber.
\end{align}
Now, the construction of interior spot is transformed to show the existence of remainder term $(\phi,\psi)$.  However, focusing on the construction of boundary pattern, we need to study the effect of boundary on the ansatz and error analysis.  Our strategy is to strengthen the boundary and analyze the corresponding new error.   

It is worthy mentioning that we mainly compute the error arising from the transported Keller-Segel system here and the study of Stokes operator and $u$-equation will be shown in Section \ref{sect-model-stokes}.   
\section{The Effect of Boundary on Spots}\label{sect3}
In this section, we flatten the boundary and study the influence of boundary curvature on the error estimate arising from $S(n)=0$.  Since $\pd \Omega$ is assumed to be smooth, we can find near the location $\xi=(\xi_1,\xi_2)$ the boundary is the graph of a smooth function $\rho: (-R_{\xi},R_{\xi})\rightarrow\mathbb R$ with $\rho(0)=\rho'(0)=0.$  In other words, $\pd \Omega$ can be represented by $(x_1,\rho(x_1))$ locally.  Denote $X_1:=x_1-\xi_1$ and $X_2=x_2-\xi_2-\rho(x_1-\xi_1)$, then one has the gradient, Laplace and Neumann boundary operators become
\begin{equation*}
    \nabla_x = \Big(\frac{\partial}{\partial X_1} - \frac{\partial}{\partial X_2}\cdot \rho'(X_1), \frac{\partial}{\partial X_2} \Big), 
\end{equation*}
\begin{equation*}
\Delta_x=\Delta_{X}+(\rho'(X_1))^2\partial_{X_2X_2}-2\rho'(X_1)\partial_{X_2X_1}-\rho''(X_1)\partial_{X_2}, 
\end{equation*}
and
 \begin{equation*}
 (1+(\rho'(X_1))^2)\frac{\partial }{\partial\boldsymbol{\nu}}=(\rho'(X_1))\partial_{X_1}-\partial_{X_2}-(\rho'(X_1))^2\partial_{X_2}.
 \end{equation*}
Moreover, in the inner coordinate $Y_i=\frac{X_i}{\e}$, $i=1,2,$ we utilize $\rho(0)=\rho'(0)=0$ to get $\rho(X_1):=\rho(\e Y_1)$ can be expanded as 
\begin{align}\label{sect3-rho-X1}
\rho(X_1)=\frac{1}{2}\rho''(0)\varepsilon^2 Y_1^2+O(\varepsilon^3).
\end{align}
 By using (\ref{sect3-rho-X1}), one finds  
 \begin{equation}\label{sect3-Delta-w}
 \begin{split}
     \Delta_x w&= \frac{1}{\varepsilon^2}\Delta_Y w + (\rho'(\varepsilon Y_1))^2 \frac{1}{\varepsilon^2}\partial_{Y_2Y_2} w- 
      \frac{2}{\varepsilon^2} (\rho'(\varepsilon Y_1)) \partial_{Y_1Y_2}w - \frac{1}{\varepsilon}\rho''(\varepsilon Y_1) \partial_{Y_2}w\\
      & = \frac{1}{\varepsilon^2}\Delta_Y w + (\rho''(0))^2 Y^2_1\partial_{Y_2 Y_2} w - \frac{2}{\varepsilon}\rho''(0)Y_1 \partial_{Y_1Y_2} w - \frac{1}{\varepsilon}\rho''(0) \partial_{Y_2}w+O(1),
      \end{split}  
 \end{equation}
 and
  \begin{equation}\label{sect3-nabla-w}
  \begin{split}
      \nabla_x w_1\cdot \nabla_x w_2 &=  \frac{1}{\varepsilon^2} \nabla_Y w_1 \cdot \nabla_Y w_2   + \frac{1}{\varepsilon^2}\frac{\partial w_1}{\partial Y_2} \cdot \frac{\partial w_2}{\partial Y_2} (\rho'(\varepsilon Y_1))^2   \\
      &  -  \frac{1}{\varepsilon^2}\Big(\frac{\partial w_1}{\partial Y_1} \cdot \frac{\partial w_2}{\partial Y_2} + \frac{\partial w_1}{\partial Y_2} \cdot \frac{\partial w_2}{\partial Y_1} \Big)\rho'(\varepsilon Y_1)  \\
        &  = \frac{1}{\varepsilon^2} \nabla_Y w_1 \cdot \nabla_Y w_2  + \frac{\partial w_1}{\partial Y_2} \cdot \frac{\partial w_2}{\partial Y_2} (\rho''(0))^2Y_1^2  \\
         &  -  \frac{1}{\varepsilon}\Big(\frac{\partial w_1}{\partial Y_1} \cdot \frac{\partial w_2}{\partial Y_2} + \frac{\partial w_1}{\partial Y_2} \cdot \frac{\partial w_2}{\partial Y_1} \Big)\rho''(0)Y_1+O(1).
      \end{split}
      \end{equation}

With the help of (\ref{sect3-Delta-w}) and (\ref{sect3-nabla-w}), we are able to compute the error generated by the approximate solution \eqref{sect2-n-and-c} involving the boundary curvature term.  Upon substituting (\ref{sect2-n-and-c}) into $\e^4 S(n)=0$, we obtain
\begin{align*}
\varepsilon^4S(u)=&\varepsilon^4\Big[\nabla_x\cdot\Big(\nabla_x\Big(\frac{1}{\varepsilon^2}W+\phi\Big)-\Big(\frac{1}{\varepsilon^2} W+\phi\Big)\nabla_x(\Gamma+H+\psi)\Big)\Big]-\varepsilon^4{\bf u}\cdot \nabla_x\Big(\frac{1}{\varepsilon^2}W+\phi\Big)\\
=&\varepsilon^2 \Delta_x  W+\varepsilon^4 \Delta_x\phi-\nabla_x(\varepsilon^2W+\varepsilon^4\phi)\cdot \nabla_x (\Gamma+H+\psi)-(\varepsilon^2W+\varepsilon^4\phi)\Delta_x (\Gamma+H+\psi)-{\bf u}\cdot \nabla_x\big({\varepsilon^2}W+\varepsilon^4\phi\big)\\
=&\Delta_Y W-2\frac{\partial^2 W}{\partial Y_1 \partial Y_2}\rho'(\varepsilon Y_1)+\frac{\partial^2 W}{\partial Y_2^2}[\rho'(\varepsilon Y_1)]^2-\varepsilon\frac{\partial W}{\partial Y_2}\rho''(\varepsilon Y_1)\\
&+\varepsilon^2\Delta_Y \phi-2\varepsilon^2\frac{\partial^2\phi}{\partial Y_1\partial Y_2}\rho'(\varepsilon Y_1)+\varepsilon^2\frac{\partial^2\phi}{\partial Y_2^2}[\rho'(\varepsilon Y_1)]^2-\varepsilon^3 \frac{\partial \phi}{\partial Y_2}\rho''(\varepsilon Y_1)\\
&-(\varepsilon^2\nabla_x W\cdot\nabla_x\Gamma+\varepsilon^4\nabla_x\phi\cdot\nabla_x\Gamma+\varepsilon^4\nabla_x \phi\cdot\nabla_x\psi-\varepsilon^2\nabla_xW\cdot \nabla_x\psi):=I_{1}\\
&-{(\varepsilon^2 W\Delta_x \Gamma+\varepsilon^2 W\Delta_x \psi+\varepsilon^4\phi\Delta_x \Gamma+\varepsilon^4\phi\Delta_x\psi)}:=I_2\\
&-\varepsilon^2W\Delta_x H-\varepsilon^4\phi\Delta_x H-\varepsilon^2\nabla_xW\cdot \nabla_xH-\varepsilon^4\nabla_x\phi\cdot\nabla_x H \\
&-{\bf u}\cdot \nabla_y\big({\varepsilon} W+{\varepsilon^3}\phi\big)+{\bf u}_1\frac{\partial ({\varepsilon} W+{\varepsilon^3}\phi)}{\partial Y_2} \rho'(\varepsilon Y_1).
\end{align*}
Noting that $\phi$ can be decomposed as \eqref{phi-decomposition1} and we strengthen the boundary locally near the location $\xi,$ one finds $(\phi,\psi)(x)$ satisfies
\begin{align}\label{phi-decomposition1-bdry}
\phi(x)=\frac{1}{\e^2}\Phi^{\text{i}}(Y)\chi(Y)+\varphi^{\text{o}},~~~~-(\Delta_x-1)^{-1}\phi=\psi.
\end{align}
Next, we substitute (\ref{phi-decomposition1-bdry}) into $I_1$ to get 
\begin{align*}
I_{1}=&\varepsilon^2\nabla_x W\cdot\nabla_x\Gamma+\varepsilon^4\nabla_x\phi\cdot\nabla_x\Gamma+\varepsilon^4\nabla_x \phi\cdot\nabla_x\psi\\
=&\Bigg(\frac{\partial W}{\partial Y_1}-\frac{\partial W}{\partial Y_2}\rho'(\varepsilon Y_1),\frac{\partial W}{\partial Y_2}\Bigg)\cdot \Bigg(\frac{\partial \Gamma}{\partial Y_1}-\frac{\partial \Gamma}{\partial Y_2}\rho'(\varepsilon Y_1),\frac{\partial \Gamma}{\partial Y_2}\Bigg)\\
&+\Bigg(\frac{\partial (\Phi^{\text{i}}\chi)}{\partial Y_1}-\frac{\partial (\Phi^{\text{i}}\chi)}{\partial Y_2}\rho'(\varepsilon Y_1),\frac{\partial (\Phi^{\text{i}}\chi)}{\partial Y_2}\Bigg)\cdot \Bigg(\frac{\partial \Gamma}{\partial Y_1}-\frac{\partial \Gamma}{\partial Y_2}\rho'(\varepsilon Y_1),\frac{\partial \Gamma}{\partial Y_2}\Bigg)\\
&+\Bigg(\frac{\partial W}{\partial Y_1}-\frac{\partial W}{\partial Y_2}\rho'(\varepsilon Y_1),\frac{\partial W}{\partial Y_2}\Bigg)\cdot \Bigg(\frac{\partial \hat\Psi^{\text{i}}}{\partial Y_1}-\frac{\partial \hat\Psi^{\text{i}}}{\partial Y_2}\rho'(\varepsilon Y_1),\frac{\partial\hat\Psi^{\text{i}}}{\partial Y_2}\Bigg)\\
&+\Bigg(\frac{\partial (\Phi^{\text{i}}\chi)}{\partial Y_1}-\frac{\partial (\Phi^{\text{i}}\chi)}{\partial Y_2}\rho'(\varepsilon Y_1),\frac{\partial (\Phi^{\text{i}}\chi)}{\partial Y_2}\Bigg)\cdot \Bigg(\frac{\partial \hat\Psi^{\text{i}}}{\partial Y_1}-\frac{\partial \hat\Psi^{\text{i}}}{\partial Y_2}\rho'(\varepsilon Y_1),\frac{\partial \hat\Psi^{\text{i}}}{\partial Y_2}\Bigg)\\
&+\varepsilon^3\nabla_x\varphi^{\text{o}}\cdot\Bigg(\frac{\partial \Gamma}{\partial Y_1}-\frac{\partial \Gamma}{\partial Y_2}\rho'(\varepsilon Y_1),\frac{\partial \Gamma}{\partial Y_2}\Bigg)+\varepsilon^4\nabla_x \varphi^{\text{o}}\cdot\nabla_x\psi\nonumber\\
&+\e^2
\Bigg(\frac{\partial (\Phi^{\text{i}}\chi)}{\partial Y_1}-\frac{\partial (\Phi^{\text{i}}\chi)}{\partial Y_2}\rho'(\varepsilon Y_1),\frac{\partial (\Phi^{\text{i}}\chi)}{\partial Y_2}\Bigg)\cdot \nabla_x(\psi-\hat\Psi^{\text{i}}).
\end{align*}
We further rearrange $I_1$ and obtain
\begin{align*}
I_1=&\nabla_Y W\cdot\nabla_Y\Gamma-\Bigg(\frac{\partial W}{\partial Y_2}\frac{\partial \Gamma}{\partial Y_1}+\frac{\partial \Gamma}{\partial Y_2}\frac{\partial W}{\partial Y_1}\Bigg)\rho'(\varepsilon Y_1)+\frac{\partial W}{\partial Y_2}\frac{\partial \Gamma}{\partial Y_2}[\rho'(\varepsilon Y_1)]^2\\
&+\nabla_Y (\Phi^{\text{i}}\chi)\cdot\nabla_Y\Gamma-\Bigg(\frac{\partial (\Phi^{\text{i}}\chi)}{\partial Y_2}\frac{\partial \Gamma}{\partial Y_1}+\frac{\partial \Gamma}{\partial Y_2}\frac{\partial(\Phi^{\text{i}}\chi)}{\partial Y_1}\Bigg)\rho'(\varepsilon Y_1)+\frac{\partial (\Phi^{\text{i}}\chi)}{\partial Y_2}\frac{\partial \Gamma}{\partial Y_2}[\rho'(\varepsilon Y_1)]^2\\
&+\nabla_Y W\cdot\nabla_Y\hat\Psi^{\text{i}}-\Bigg(\frac{\partial \hat\Psi^{\text{i}}}{\partial Y_2}\frac{\partial W}{\partial Y_1}+\frac{\partial W}{\partial Y_2}\frac{\partial \hat\Psi^{\text{i}}}{\partial Y_1}\Bigg)\rho'(\varepsilon Y_1)+\frac{\partial W}{\partial Y_2}\frac{\partial \hat\Psi^{\text{i}}}{\partial Y_2}[\rho'(\varepsilon Y_1)]^2\\
&+\nabla_Y(\Phi^{\text{i}}\chi)\cdot\nabla_Y\hat\Psi^{\text{i}}-\Bigg(\frac{\partial (\Phi^{\text{i}}\chi)}{\partial Y_2}\frac{\partial  \hat\Psi^{\text{i}}}{\partial Y_1}+\frac{\partial  \hat\Psi^{\text{i}}}{\partial Y_2}\frac{\partial (\Phi^{\text{i}}\chi)}{\partial Y_1}\Bigg)\rho'(\varepsilon Y_1)+\frac{\partial(\Phi^{\text{i}}\chi)}{\partial Y_2}\frac{\partial  \hat\Psi^{\text{i}}}{\partial Y_2}[\rho'(\varepsilon Y_1)]^2\\
&+\varepsilon^3\nabla_x\varphi^{\text{o}}\cdot\Bigg(\frac{\partial \Gamma}{\partial Y_1}-\frac{\partial \Gamma}{\partial Y_2}\rho'(\varepsilon Y_1),\frac{\partial \Gamma}{\partial Y_2}\Bigg)+\varepsilon^4\nabla_x \varphi^{\text{o}}\cdot\nabla_x\psi\nonumber\\
&+\e^2
\Bigg(\frac{\partial (\Phi^{\text{i}}\chi)}{\partial Y_1}-\frac{\partial (\Phi^{\text{i}}\chi)}{\partial Y_2}\rho'(\varepsilon Y_1),\frac{\partial (\Phi^{\text{i}}\chi)}{\partial Y_2}\Bigg)\cdot \nabla_x(\psi-\hat\Psi^{\text{i}}),
\end{align*}
where $-(\Delta_x-1)^{-1}(\Phi^{\text{i}}\chi)=\hat\Psi^{\text{i}}$.  Proceeding $I_2$ with the same argument, we have
\begin{align*}
I_{2}=&\varepsilon^2 W\Delta_x \Gamma+\varepsilon^2 W\Delta_x \psi+\varepsilon^4\phi\Delta_x \Gamma+\varepsilon^4\phi\Delta_x\psi\\
=&W\Bigg[\Delta_Y \Gamma-2\frac{\partial^2\Gamma}{\partial Y_1\partial Y_2}\rho'(\varepsilon Y_1)+\frac{\partial^2\Gamma}{\partial Y_2^2}(\rho'(\varepsilon Y_1))^2-\varepsilon\frac{\partial \Gamma}{\partial Y_2}\rho''(\varepsilon Y_1)\Bigg]\\
&+\Phi^{\text{i}}\chi\Bigg[\Delta_Y \Gamma-2\frac{\partial^2\Gamma}{\partial Y_1\partial Y_2}\rho'(\varepsilon Y_1)+\frac{\partial^2\Gamma}{\partial Y_2^2}(\rho'(\varepsilon Y_1))^2-\varepsilon\frac{\partial \Gamma}{\partial Y_2}\rho''(\varepsilon Y_1)\Bigg]\\
&+W\Bigg[\Delta_Y \hat\Psi^{\text{i}}-2\frac{\partial^2\hat\Psi^{\text{i}}}{\partial Y_1\partial Y_2}\rho'(\varepsilon Y_1)+\frac{\partial^2\hat\Psi^{\text{i}}}{\partial Y_2^2}(\rho'(\varepsilon Y_1))^2-\varepsilon\frac{\partial \hat\Psi^{\text{i}}}{\partial Y_2}\rho''(\varepsilon Y_1)\Bigg]\\
&+\Phi^{\text{i}}\chi\Bigg[\Delta_Y \hat\Psi^{\text{i}}-2\frac{\partial^2\hat\Psi^{\text{i}}}{\partial Y_1\partial Y_2}\rho'(\varepsilon Y_1)+\frac{\partial^2\hat\Psi^{\text{i}}}{\partial Y_2^2}(\rho'(\varepsilon Y_1))^2-\varepsilon\frac{\partial \hat\Psi^{\text{i}}}{\partial Y_2}\rho''(\varepsilon Y_1)\Bigg]\\
&+\varepsilon^2\varphi^{o}\Bigg[\Delta_Y \Gamma-2\frac{\partial^2\Gamma}{\partial Y_1\partial Y_2}\rho'(\varepsilon Y_1)+\frac{\partial^2\Gamma}{\partial Y_2^2}(\rho'(\varepsilon Y_1))^2-\varepsilon\frac{\partial \Gamma}{\partial Y_2}\rho''(\varepsilon Y_1)\Bigg]\\
&+\varepsilon^4\varphi^{\text{o}}\cdot\Delta_x\psi+\e^2
\Phi^{\text{i}}\chi\Delta_x(\psi-\hat\Psi^{\text{i}}).
\end{align*}
Focusing on the inner region ${|x-\xi|}\lesssim {\e}$, we substitute $I_{1}$ and $I_{2}$ into $\varepsilon^4 S(u)$, then obtain
\begin{align}\label{sect3-eSu-before}
\varepsilon^4S(u)=&\Delta_Y W-\nabla_Y W\cdot\nabla_Y\Gamma-W\Delta_Y\Gamma\nonumber\\
&+\Bigg[\Big(\frac{\partial W}{\partial Y_2}\frac{\partial \Gamma}{\partial Y_1}+\frac{\partial \Gamma}{\partial Y_2}\frac{\partial W}{\partial Y_1}\Big)+2W\frac{\partial^2 \Gamma}{\partial Y_1\partial Y_2}-2\frac{\partial^2 W}{\partial Y_1\partial Y_2}\Bigg]\rho'(\varepsilon Y_1)\nonumber\\
&+\Bigg(\frac{\partial^2 W}{\partial Y_2^2}+\frac{\partial W}{\partial Y_2}\frac{\partial \Gamma}{\partial Y_2}-W\frac{\partial^2\Gamma}{\partial Y_2^2}\Bigg)[\rho'(\varepsilon Y_1)]^2-\varepsilon\Big(\frac{\partial W}{\partial Y_2}-W\frac{\partial\Gamma}{\partial Y_2}\Big)\rho''(\varepsilon Y_1)\nonumber\\
&+\Delta_Y \Phi^{\text{i}}\chi-\nabla_Y \Phi^{\text{i}}\cdot\nabla_Y\Gamma\chi-\nabla_Y W\cdot\nabla_Y\bar \Psi^{\text{i}}\chi-\Phi^{\text{i}}\Delta_Y\Gamma\chi-W\Delta_Y\bar\Psi^{\text{i}}\chi\nonumber\\
&-\rho'(\varepsilon Y_1)\chi\bigg[2\frac{\partial^2\Phi^{\text{i}}}{\partial Y_1\partial Y_2}-\Big(\frac{\partial\Phi^{\text{i}}}{\partial Y_1}\frac{\partial\Gamma}{\partial Y_2}+\frac{\partial\Phi^{\text{i}}}{\partial Y_2}\frac{\partial\Gamma }{\partial Y_1}\Big)-\Big(\frac{\partial W}{\partial Y_1}\frac{\partial\bar \Psi^{\text{i}}}{\partial Y_2}+\frac{\partial W}{\partial Y_2}\frac{\partial\bar \Psi^{\text{i}} }{\partial Y_1}\Big)\nonumber\\
&\ \ \ \ \ \ \ \ \ \ -2\Phi^{\text{i}}\frac{\partial^2\Gamma^{\text{i}}}{\partial Y_1\partial Y_2}-2\frac{\partial^2\bar\Psi^{\text{i}}}{\partial Y_1\partial Y_2}W-2\Phi^{\text{i}}\frac{\partial^2\bar\Psi^{\text{i}}}{\partial Y_1\partial Y_2}-\Bigg(\frac{\partial \Phi^{\text{i}}}{\partial Y_2}\frac{\partial \bar\Psi^{\text{i}}}{\partial Y_1}+\frac{\partial \bar\Psi^{\text{i}}}{\partial Y_2}\frac{\partial \Phi^{\text{i}}}{\partial Y_1}\Bigg)\bigg]\nonumber\\
&+[\rho'(\varepsilon Y_1)]^2\chi\bigg[\frac{\partial^2 \Phi^{\text{i}}}{\partial Y_2^2}-\frac{\partial \Phi^{\text{i}}}{\partial Y_2}\frac{\partial\Gamma}{\partial Y_2}-\frac{\partial W}{\partial Y_2}\frac{\partial\bar\Psi^{\text{i}}}{\partial Y_2}-\frac{\partial^2\Gamma}{\partial Y_2^2}\Phi^{\text{i}}-\frac{\partial^2\bar\Psi^{\text{i}}}{\partial Y_2^2}W-\frac{\partial \Phi^{\text{i}}}{\partial Y_2}\frac{\partial \bar\Psi^{\text{i}}}{\partial Y_2}-\Phi^{\text{i}}\frac{\partial^2\bar\Psi^{\text{i}}}{\partial Y_2^2}\bigg]\nonumber\\
&{-\varepsilon\nabla_YW\cdot \nabla_XH-\varepsilon {\bf u}\cdot \nabla_Y W}\nonumber\\
&-\varepsilon^2W\Delta_X H-\varepsilon  {\bf u}\cdot \nabla_Y\Phi^{\text{i}}+{u_1}\frac{\partial ({\varepsilon} W+{\varepsilon^3}\phi)}{\partial Y_2}\cdot \rho'(\varepsilon Y_1)+\text{H. O. T.}
\end{align}
Noting that $X_1=\e Y_1$, $X_2=\e Y_2$ and $\rho(0)=\rho'(0)=0$, we find $\rho(X_1)$, $\rho'(X_1)$ and $\rho''(X_1)$ can be expanded as
\begin{align}\label{sect3-rhoY1-expand}
\rho(\varepsilon Y_1)=\frac{1}{2}\rho''(0)\varepsilon^2 Y_1^2+O(\varepsilon^3),
\end{align}
\begin{align}\label{sect3-rhoprime-Y1}
\rho'(\varepsilon Y_1)=\rho''(0)\varepsilon Y_1+\frac{1}{2}\rho'''(0)\varepsilon^2 Y_1^2+O(\varepsilon^3),
\end{align}
and
\begin{align}\label{sect3-rhodouble-Y1}
\rho''(\varepsilon Y_1)=\rho''(0)+\rho'''(0)\varepsilon Y_1+\frac{1}{2}\rho^{(4)}(0)\varepsilon^2 Y_1^2+O(\varepsilon^3).
\end{align}
In addition, it follows from $W=e^{\Gamma}$ that  
\begin{align*}
\frac{\partial^2 W}{\partial Y_2^2}+\frac{\partial W}{\partial Y_2}\frac{\partial \Gamma}{\partial Y_2}-W\frac{\partial^2\Gamma}{\partial Y_2^2}=0,
\end{align*}
\begin{align*}
\Big(\frac{\partial W}{\partial Y_2}\frac{\partial \Gamma}{\partial Y_1}+\frac{\partial \Gamma}{\partial Y_2}\frac{\partial W}{\partial Y_1}\Big)+2W\frac{\partial^2 \Gamma}{\partial Y_1\partial Y_2}-2\frac{\partial^2 W}{\partial Y_1\partial Y_2}=0,
\end{align*}
and
\begin{align}\label{sect3-wgamma-3}
\frac{\partial W}{\partial Y_2}-W\frac{\partial\Gamma}{\partial Y_2}=0.
\end{align}
Upon collecting (\ref{sect3-rhoY1-expand})--(\ref{sect3-wgamma-3}), we simplify (\ref{sect3-eSu-before}) as
\begin{align}\label{sect3-inner-before}
\varepsilon^4 S(u)=&(\Delta_Y \Phi^{\text{i}}-\nabla_Y \Phi^{\text{i}}\cdot\nabla_Y\Gamma-\nabla_Y W\cdot\nabla_Y\bar\Psi^{\text{i}}-\Phi^{\text{i}}\Delta_Y\Gamma-W\Delta_Y\bar\Psi^{\text{i}})\chi\nonumber\\
&-\rho'(\varepsilon Y_1)\chi\bigg[2\frac{\partial^2\Phi^{\text{i}}}{\partial Y_1\partial Y_2}-\Big(\frac{\partial\Phi^{\text{i}}}{\partial Y_1}\frac{\partial\Gamma}{\partial Y_2}+\frac{\partial\Phi^{\text{i}}}{\partial Y_2}\frac{\partial\Gamma }{\partial Y_1}\Big)-\Big(\frac{\partial W}{\partial Y_1}\frac{\partial\bar\Psi^{\text{i}}}{\partial Y_2}+\frac{\partial W}{\partial Y_2}\frac{\partial\bar\Psi^{\text{i}} }{\partial Y_1}\Big)\nonumber\\
&\ \ \ \ \ \ \ \ \ \ -2\Phi^{\text{i}}\frac{\partial^2\Gamma}{\partial Y_1\partial Y_2}-2\frac{\partial^2\bar\Psi^{\text{i}}}{\partial Y_1\partial Y_2}W-2\Phi^{\text{i}}\frac{\partial^2\bar\Psi^{\text{i}}}{\partial Y_1\partial Y_2}-\Bigg(\frac{\partial \Phi^{\text{i}}}{\partial Y_2}\frac{\partial \bar\Psi^{\text{i}}}{\partial Y_1}+\frac{\partial \bar\Psi^{\text{i}}}{\partial Y_2}\frac{\partial \Phi^{\text{i}}}{\partial Y_1}\Bigg)\bigg]\nonumber\\
&+[\rho'(\varepsilon Y_1)]^2\chi\bigg[\frac{\partial^2 \Phi^{\text{i}}}{\partial Y_2^2}-\frac{\partial \Phi^{\text{i}}}{\partial Y_2}\frac{\partial\Gamma}{\partial Y_2}-\frac{\partial W}{\partial Y_2}\frac{\partial\bar\Psi^{\text{i}}}{\partial Y_2}-\frac{\partial^2\Gamma}{\partial Y_2^2}\Phi^{\text{i}}-\frac{\partial^2\bar\Psi^{\text{i}}}{\partial Y_2^2}W-\frac{\partial \Phi^{\text{i}}}{\partial Y_2}\frac{\partial \bar\Psi^{\text{i}}}{\partial Y_2}-\Phi^{\text{i}}\frac{\partial^2\bar\Psi^{\text{i}}}{\partial Y_2^2}\bigg]\nonumber\\
&-\varepsilon\nabla_YW\cdot \nabla_XH-\varepsilon {\bf u}\cdot \nabla_Y W+\text{H. O. T.}
\end{align}
Now, we have from (\ref{sect3-inner-before}) that in the inner region $|x-\xi|\lesssim \e$, the leading order error is $O(\e)$ and given by
\begin{align}\label{sect3-I3}
I_3:=-\varepsilon\nabla_YW\cdot \nabla_XH-\varepsilon {\bf u}\cdot \nabla_Y W.
\end{align}
It follows from (\ref{sect3-I3}) that the leading order error does not include any boundary curvature term $\rho''(0)$.  It is natural to further formulate the inner and outer problem satisfied by $\Phi^{\text{i}}$ and $\varphi^{\text{o}}.$  In fact, since we flatten the boundary locally, the form of outer operator is the same as (\ref{sect2-linearized-outer-operator}).  However, compared to (\ref{sect2-linearized-inner-operator}), the inner operator should be written in $Y$-variable and the inner problem will be solved in $\mathbb R^{2}_{+}$ rather than $\mathbb R^2$.  After finishing the error estimate to the equation, we focus on the study of boundary conditions.     
%In summary, with the help of (\ref{sect3-wgamma-1})--(\ref{sect3-wgamma-3}), one finds the boundary curvature term does not influence the leading order error.

Since $(n,c)$ satisfies the no-flux boundary condition, we similarly use (\ref{phi-decomposition1}) and calculate to get
%We next analyze the influence of the boundary operator, which is from $\phi(x)=\frac{1}{\varepsilon^2}\Phi(y)$ that
\begin{align}\label{sect3-error-analysis-boundary-before}
\varepsilon^3\bigg(\frac{\partial n}{\partial\boldsymbol{\nu}}-n\frac{\partial c}{\partial \boldsymbol{\nu}}\bigg)=&\varepsilon\frac{\partial (W+\varepsilon^2\phi)}{\partial\boldsymbol{\nu}}-\varepsilon(W+\varepsilon^2\phi)\frac{\partial(\Gamma+H+\psi)}{\partial \boldsymbol{\nu}}\nonumber\\
=&\varepsilon\frac{\partial W}{\partial \boldsymbol{\nu}}-\varepsilon W\frac{\partial\Gamma }{\partial\boldsymbol{\nu}}+\varepsilon^3 \frac{\partial\phi}{\partial\boldsymbol{\nu}}-\varepsilon W\frac{\partial \psi}{\partial\boldsymbol{\nu}}-\varepsilon^3\phi\frac{\partial\Gamma}{\partial\boldsymbol{\nu}}-\varepsilon^3\phi\frac{\partial \psi}{\partial \boldsymbol{\nu}}-\varepsilon(W+\varepsilon^2\phi)\frac{\partial H}{\partial\boldsymbol{\nu}}\nonumber\\
=&\frac{1}{\sqrt{1+[\rho'(\varepsilon Y_1)]^2}}\Bigg[-\frac{\partial W}{\partial Y_2}+W\frac{\partial \Gamma}{\partial Y_2}+\rho'(\varepsilon Y_1)\bigg(\frac{\partial W}{\partial Y_1}-W\frac{\partial \Gamma}{\partial Y_1}\bigg)-[\rho'(\varepsilon Y_1)]^2\bigg(\frac{\partial W}{\partial Y_2}-W\frac{\partial \Gamma}{\partial Y_2}\bigg)\nonumber\\
&-\frac{\partial (\Phi^{\text{i}}\chi+\e^2\varphi^{\text{o}})}{\partial Y_2}+W\frac{\partial (\hat \Psi^{\text{i}}+\e^2\psi^{\text{o}})}{\partial Y_2}+\rho'(\varepsilon Y_1)\bigg(\frac{\partial (\Phi^{\text{i}}\chi+\e^2\varphi^{\text{o}})}{\partial Y_1}-W\frac{\partial 
 (\hat \Psi^{\text{i}}+\e^2\psi^{\text{o}})}{\partial Y_1}\bigg)\nonumber\\
&-[\rho'(\varepsilon Y_1)]^2\bigg(\frac{\partial (\Phi^{\text{i}}\chi+\e^2\varphi^{\text{o}})}{\partial Y_2}-W\frac{\partial (\hat \Psi^{\text{i}}+\e^2\psi^{\text{o}})}{\partial Y_2}\bigg)\nonumber\\
&-(\Phi^{\text{i}}\chi+\e^2\varphi^{\text{o}})\Big(\frac{\partial \Gamma}{\partial Y_1}\rho'(\varepsilon Y_1)-\frac{\partial \Gamma}{\partial Y_2}-\frac{\partial\Gamma}{\partial Y_2}(\rho'(\varepsilon Y_1))^2\Big)-(\Phi^{\text{i}}\chi+\e^2\varphi^{\text{o}})\Big(\frac{\partial(\hat \Psi^{\text{i}}+\e^2\psi^{\text{o}})}{\partial Y_1}\rho'(\varepsilon Y_1)\nonumber\\
&-\frac{\partial (\hat \Psi^{\text{i}}+\e^2\psi^{\text{o}})}{\partial Y_2}-\frac{\partial(\hat \Psi^{\text{i}}+\e^2\psi^{\text{o}})}{\partial Y_2}(\rho'(\varepsilon Y_1))^2\Big)\nonumber\\
&+\varepsilon(W+\Phi^{\text{i}}\chi+\e^2\varphi^{\text{o}})\Big(\frac{\partial H}{\partial X_2}-\rho'(\varepsilon Y_1)\frac{\partial H}{\partial X_1}+[\rho'(\varepsilon Y_1)]^2\frac{\partial H}{\partial X_2}\Big)\Bigg].
\end{align}
Next, we perform the order analysis by regarding $\e$ as the variable.  On one hand, we expand $\frac{1}{\sqrt{1+[\rho'(\varepsilon Y_1)]^2}}$ as
\begin{align}\label{sect3-error-boundary-1}
\frac{1}{\sqrt{1+[\rho'(\varepsilon Y_1)]^2}}=1-\frac{1}{2}[\rho'(\varepsilon Y_1)]^2+O(\varepsilon^4).
\end{align}
On the other hand, since $W=e^{\Gamma}$, one finds
\begin{align}\label{sect3-error-boundary-2}
\frac{\partial W}{\partial Y_i}-W\frac{\partial \Gamma}{\partial Y_i}=0,~~i=1,2.
\end{align}
Upon substituting (\ref{sect3-error-boundary-1}) and (\ref{sect3-error-boundary-2}) into (\ref{sect3-error-analysis-boundary-before}), we obtain that in the inner region $|x-\xi|\lesssim \e$,
\begin{align}\label{sect3-no-flux-bdry-influence-after}
\varepsilon^3\bigg(\frac{\partial n}{\partial\boldsymbol{\nu}}-n\frac{\partial c}{\partial \boldsymbol{\nu}}\bigg)=&\varepsilon\frac{\partial (W+\varepsilon^2\phi)}{\partial\boldsymbol{\nu}}-\varepsilon(W+\varepsilon^2\phi)\frac{\partial(\Gamma+H+\psi)}{\partial \boldsymbol{\nu}}\nonumber\\
=&(1+O(\varepsilon^2))\Bigg[-\frac{\partial \Phi^{\text{i}}}{\partial Y_2}+W\frac{\partial \bar\Psi^{\text{i}}}{\partial Y_2}+\rho'(\varepsilon Y_1)\bigg(\frac{\partial \Phi^{\text{i}}}{\partial Y_1}-W\frac{\partial \bar\Psi^{\text{i}}}{\partial Y_1}\bigg)-[\rho'(\varepsilon Y_1)]^2\bigg(\frac{\partial \Phi^{\text{i}}}{\partial Y_2}-W\frac{\partial \bar\Psi^{\text{i}}}{\partial Y_2}\bigg)\nonumber\\
&-\Phi^{\text{i}}\Big(\frac{\partial \Gamma}{\partial Y_1}\rho'(\varepsilon Y_1)-\frac{\partial \Gamma}{\partial Y_2}-\frac{\partial\Gamma}{\partial Y_2}(\rho'(\varepsilon Y_1))^2\Big)-\Phi^{\text{i}}\Big(\frac{\partial \bar\Psi^{\text{i}}}{\partial Y_1}\rho'(\varepsilon Y_1)-\frac{\partial \bar\Psi^{\text{i}}}{\partial Y_2}-\frac{\partial\bar\Psi^{\text{i}}}{\partial Y_2}(\rho'(\varepsilon Y_1))^2\Big)\nonumber\\
&+\varepsilon(W+\Phi^{\text{i}})\Big(\frac{\partial H}{\partial X_2}-\rho'(\varepsilon Y_1)\frac{\partial H}{\partial X_1}+[\rho'(\varepsilon Y_1)]^2\frac{\partial H}{\partial X_2}\Big)\Bigg]\nonumber\\
=&(1+O(\varepsilon^2))\Bigg[-\frac{\partial \Phi^{\text{i}}}{\partial Y_2}+W\frac{\partial \bar\Psi^{\text{i}}}{\partial Y_2}+\rho''(0)\varepsilon Y_1\bigg(\frac{\partial \Phi^{\text{i}}}{\partial Y_1}-W\frac{\partial \bar\Psi^{\text{i}}}{\partial Y_1}\bigg)-[\rho''(0)\varepsilon Y_1]^2\bigg(\frac{\partial \Phi^{\text{i}}}{\partial Y_2}-W\frac{\partial \bar\Psi^{\text{i}}}{\partial Y_2}\bigg)\nonumber\\
&-\Phi^{\text{i}}\Big(\frac{\partial \Gamma}{\partial Y_1}\rho''(0)\varepsilon Y_1-\frac{\partial \Gamma}{\partial Y_2}-\frac{\partial\Gamma}{\partial Y_2}(\rho''(0)\varepsilon Y_1)^2\Big)\nonumber\\
&-\Phi^{\text{i}}\Big(\frac{\partial \bar\Psi^{\text{i}}}{\partial Y_1}\rho''(0)\varepsilon Y_1-\frac{\partial \bar\Psi^{\text{i}}}{\partial Y_2}-\frac{\partial\bar\Psi^{\text{i}}}{\partial Y_2}(\rho''(0)\varepsilon Y_1)^2\Big)+O(\varepsilon^2)\Bigg]\nonumber\\
&+\varepsilon(W+\Phi^{\text{i}})\Big(\frac{\partial H}{\partial X_2}-\rho'(\varepsilon Y_1)\frac{\partial H}{\partial X_1}+[\rho'(\varepsilon Y_1)]^2\frac{\partial H}{\partial X_2}\Big).
\end{align}
By checking (\ref{sect3-no-flux-bdry-influence-after}), one finds from $\Phi^{\text{i}}$ and $\bar \Psi^{\text{i}}$ are both $o(1)$ that the leading order error of boundary estimate is $O(\e)$ and given by 
\begin{align}\label{sect3-II1-boundary}
II_1:=\e W\cdot \frac{\partial H}{\partial X_2}.
\end{align}
It follows from (\ref{sect3-II1-boundary}) that the boundary curvature term $\rho''(0)$ also does not influence the leading term of boundary error.

We summarize the arguments shown in Section \ref{sect2} and Section \ref{sect3}, then obtain that the approximate solution of boundary spot is still given by (\ref{sect2-n-and-c}).  To show the existence of remainder term $(\phi,\psi)$, it is necessary to establish the linear theory of corresponding linearized operators, which will be exhibited in Section \ref{sect4}.  
%will discuss the linear theory of the stokes operator in Section \ref{stokessect}.
\section{Linear Theory: Transported Keller-Segel System}\label{sect4}
In this section, we shall discuss the properties of linearized Keller-Segel operators.  Recall that $\phi$ and $\psi$ satisfy (\ref{phi-decomposition1}), (\ref{sect3-hatpsi-inn}) and (\ref{sect3-psi-inn-varphi-o}); moreover, the inner problem satisfied by $\big(\Phi^{\text{i}},\bar\Psi^{\text{i}}\big)$ is formulated as  
%\begin{align}\label{phi-decomposition}
%\phi(x)=\Phi^{\text{inn}}(y)\chi(x)+\varphi^{\text{out}},~~y=\frac{x-\xi}{\varepsilon}.
%\end{align}
%Focusing on the $j$-th inner region, we first formulate the inner problem of $\Phi^{\text{inn}}$ in the whole space and drop ``inn" to get
\begin{align}\label{sect4-inner-operator-half}
\left\{\begin{array}{ll}
L_{W}[\Phi]:=\Delta_Y\Phi-\nabla_Y\cdot(W\nabla_Y \Psi)-\nabla_Y\cdot(\Phi\nabla_Y\Gamma)=h,&Y\in\mathbb R^2_{+},\\
(-\Delta_Y)^{-1} \Phi=\Psi,&Y\in\mathbb R^2_{+},\\
\end{array}
\right.
\end{align}
where we replace $(\Phi^{\text{i}},\bar\Psi^{\text{i}})$ by $(\Phi,\Psi)$ without confusing the reader.  Before studying (\ref{sect4-inner-operator-half}), it is necessary to establish the linear theory of inner problem (\ref{sect4-inner-operator-half}) in the whole space $\mathbb R^2$ at first.  Indeed, assume that the location $\xi\in \Omega$ and consider the inner problem formulated by
\begin{align}\label{sect4-inner-operator-whole}
\left\{\begin{array}{ll}
L_{W}[\Phi]:=\Delta_y\Phi-\nabla_y\cdot(W\nabla_y \Psi)-\nabla_y\cdot(\Phi\nabla_y\Gamma)=h,&y\in\mathbb R^2,\\
(-\Delta_y)^{-1} \Phi=\Psi,&y\in\mathbb R^2,\\
\end{array}
\right.
\end{align}
then we define the inner norm in $\mathbb R^{2}$ as 
\begin{align*}
      \|h\|_{\delta_1,\nu_1}: = \sup_{y \in \R^2}\varepsilon^{-\delta_1}\vert h\vert{(1 + |y|)^{\nu_1}},~~~~\delta_1,\nu_1>0
      \end{align*}
and have the following lemma:
\begin{lemma}\label{sect4-lemmainner-interior}
  Suppose that $h$ satisfies
  \begin{equation}\label{innercon}
      \int_{\R^2}h(y)d y = 0,  \ \ \int_{\R^2}h(y)y_jd y = 0\ \ \text{for} \ \ j =1,2,
  \end{equation}
  then for any $\Vert h\Vert_{\delta_1,4+\sigma}<\infty$ with $\delta_1>0,$ $\sigma\in(0,1)$, there exists the solution $\Phi= \mathcal{T}_{\text{i}}[h]$ to \eqref{sect4-inner-operator-whole} such that
    \begin{equation*}
      \|\Phi\|_{\delta_1, 2 + \sigma} \lesssim \|h\|_{\delta_1, 4 + \sigma}, 
    \end{equation*}
    where $\mathcal{T}_{\text{i}}[h]$ is a continuous linear operator.
\end{lemma}
The proof of Lemma \ref{sect4-lemmainner-interior} shown in \cite{KWX2022} (Cf. Lemma 3.1), crucially relies on the Fourier mode analysis since we solve (\ref{sect4-inner-operator-whole}) in 2D.  For the sake of completeness, we give the sketch of the proof here.
\begin{proof}
First of all, we define $g$ as 
\begin{align*}
g=\frac{\Phi}{W}-\Psi,
\end{align*}
then rewrite (\ref{sect4-inner-operator-whole}) as the following divergence form:
\begin{align}\label{sect4-g-problem}
\left\{\begin{array}{ll}
\nabla\cdot (W\nabla g)=h,&y\in\mathbb R^2,\\
-\Delta \Psi=\Phi.
\end{array}
\right.
\end{align}
Next, we perform the Fourier expansions of $h,$ $\Phi$, $\Psi$, $g$ and $\bar g:=Wg$.  We first write the error term $h(y)$ as
\begin{align}\label{sect4-h-decomposition-1}
h(y)=h(\rho,\theta)=\sum_{k=-\infty}^\infty {\tilde h}_k(\rho) e^{ik\theta}:=h_0(y)+h_1(y)+h_{\perp}(y),
\end{align}
where $h_k(y)={\tilde h}_k(\rho)e^{ik\theta}.$  Then, we decompose $(\Phi,\Psi)$ and $(g,\bar g)$ as 
\begin{align}
&\Phi(y)~~(\text{resp.~}\Psi(y))=\Phi(\rho,\theta)~~(\text{resp.~}\Psi(\rho,\theta))\nonumber\\
=&\sum_{k=-\infty}^{\infty}\tilde \Phi_{k}(\rho)e^{ik\theta}~~(\text{resp.~}\sum_{k=-\infty}^{\infty}\tilde \Psi_{k}(\rho)e^{ik\theta})\nonumber\\
:=&\Phi_0(y)+\Phi_1(y)+\Phi_{\perp}(y)~~(\text{resp.~}\Psi_0(y)+\Psi_1(y)+\Psi_{\perp}(y))\nonumber,
\end{align}
and
\begin{align}\label{sect4-g-and-tildeg}
&g(y)~~(\text{resp.~}\bar g(y))=g(\rho,\theta)~~(\text{resp.~}\bar g(\rho,\theta))\nonumber\\
=&\sum_{k=-\infty}^{\infty}\tilde g_{k}(\rho)e^{ik\theta}~~(\text{resp.~}\sum_{k=-\infty}^{\infty}\hat g_{k}(\rho)e^{ik\theta})\nonumber\\
:=&g_0(y)+g_1(y)+g_{\perp}(y)~~(\text{resp.~}\bar g_0(y)+\bar g_1(y)+\bar g_{\perp}(y)),
\end{align}
respectively.  Now, we construct the solution $(\Phi,\Psi)$ to (\ref{sect4-g-problem}) mode by mode.  In each mode $k$, one has $( \Phi_k,h_k)$ satisfies
\begin{align*}
\left\{\begin{array}{ll}
\nabla\cdot (W\nabla g_k)=h_k,&y\in\mathbb R^2,\\
-\Delta \Psi_k=Wg_k+W\Psi_k,
\end{array}
\right.
\end{align*}
which is equivalent to the following mode $k$ problem:
\begin{align}\label{sect4-equiv-mode-k}
\left\{\begin{array}{ll}
\mathcal L_k[\hat g_k]=h_k,&y\in\mathbb R^2,\\
-{\tilde \Psi}_{k\rho\rho}-\frac{1}{\rho}{\tilde \Psi}_{k\rho}+\frac{k^2}{\rho^2}{\tilde \Psi}_k=\bar g_k+W{\tilde \Psi}_k,
\end{array}
\right.
\end{align}
where $\mathcal L_k[\hat g_k]:=\hat g_{k\rho\rho}+\frac{1}{\rho}\hat g_{k\rho}-\frac{k^2}{\rho^2}\hat g_{k}-{(\ln W)_{\rho}}\hat g_{k\rho}+W\hat g_k.$  It is vital to study the bounded kernel functions such that 
\begin{align*}
{\tilde \Psi}_{k\rho\rho}+\frac{1}{\rho}{\tilde \Psi}_{k\rho}-\frac{k^2}{\rho^2}{\tilde \Psi}_k+W{\tilde \Psi}_k=0.
\end{align*}
Indeed, we have the fact that at mode $k=0,1$, the bounded kernel functions are given by
\begin{align}\label{sect4-boundedkernel-KS-operator}
Z_0:=\frac{\rho^2-1}{\rho^2+1},~~~~Z_j:=\partial_{y_j}\Gamma,~j=1,2,
\end{align}
where $\Gamma$ is defined in \eqref{sect2-c0-leading}.

Owing to the existence of (\ref{sect4-boundedkernel-KS-operator}), we must impose the orthogonality conditions on mode $0$ and $1$.  At mode $0$, we impose the mass condition given by
$$\int_{\mathbb R^2} hdy=0.$$
It immediately follows from (\ref{sect4-h-decomposition-1}) that 
\begin{align}\label{sect4-mass-orthogonal-condition}
\int_{\mathbb R^2}\tilde h_0(\rho)\rho d\rho=0. 
\end{align}
Then, by choosing the solution $\tilde g_0$ to \eqref{sect4-equiv-mode-k} with $k=0$ as 
\begin{align}\label{sect4-g0-sol}
\tilde g_0(\rho)=\int_{\rho}^{\infty}\frac{1}{rW(r)}\int_0^r \tilde h_0(s)sds dr,
\end{align}
we have from (\ref{sect4-mass-orthogonal-condition}), (\ref{sect4-g0-sol}) and $\Vert h\Vert_{\delta_1,4+\sigma}\lesssim 1$ that $\tilde g_0$ satisfies
\begin{align}\label{sect4-hatg0-next-before}
|\tilde g_0(\rho)|\lesssim \e^{\delta_1}\rho^{2-\sigma}.
\end{align}
In light of $\bar g=Wg$ and (\ref{sect4-g-and-tildeg}), we use (\ref{sect4-hatg0-next-before}) to get $\hat g_0$ has the fast decay property, which is
\begin{align}\label{sect4-hatg0-next}
|\hat g_0(\rho)|\lesssim \frac{\e^{\delta_1}\Vert h\Vert_{4+\sigma}}{\rho^{2+\sigma}}.
\end{align}
We further apply the variation-of-parameters formula on ${\tilde \Psi}_0$-equation to choose $\tilde \Psi_0$ satisfying 
\begin{align}\label{sect4-tilde-psi0-log}
|\tilde \Psi_0(\rho)|\lesssim \Vert h\Vert_{4+\sigma}\e^{\delta_1}\ln \rho.
\end{align}
We summarize (\ref{sect4-hatg0-next}) and (\ref{sect4-tilde-psi0-log}) to obtain that
\begin{align*}
|\tilde \Phi_0(\rho)|\lesssim \frac{\e^{\delta_1}\Vert h\Vert_{4+\sigma}}{\rho^{2+\sigma}}, 
\end{align*}
which gives us the desired estimate of mode $0$.  

Similarly, to guarantee the desired decay estimate at mode $1$, we impose the first moment orthogonality condition given in (\ref{innercon}), which is 
$$\int_{\mathbb R^2} hy_jdy=0,~~~j=1,2.$$
By direct computation, we further obtain that
\begin{align}\label{sect4-first-moment-condition-equiv}
\int_0^\infty W(\rho)\tilde g_1(\rho){\bar Z}_1(\rho)\rho d\rho=0,
\end{align}
where ${\bar Z}_1:=\frac{d}{d\rho}\Gamma.$  Recall that $\tilde \Psi_1$ satisfies
$${\tilde \Psi}_{1\rho\rho}+\frac{1}{\rho}{\tilde \Psi}_{1\rho}-\frac{1}{\rho^2}{\tilde \Psi}_1+W{\tilde \Psi}_1=0.$$
Then we apply the variation-of-parameters formula on $\tilde \Psi_1$-equation and use (\ref{sect4-first-moment-condition-equiv}) to choose the solution satisfying
\begin{align}\label{sect4-tilde-psi1-mode1}
|\tilde \Psi_1|\lesssim \frac{\e^{\delta_1}\Vert h\Vert_{4+\sigma}}{\rho^{\sigma}}.
\end{align}
On the other hand, we consider the operator $\mathcal L_k$ and employ the maximum principle to show that $\hat g_1$ satisfies
\begin{align}\label{sect4-hat-g1-mode1}
|\hat g_1|\lesssim \frac{\e^{\delta_1}\Vert h\Vert_{4+\sigma}}{\rho^{2+\sigma}}.
\end{align}
Upon collecting (\ref{sect4-tilde-psi1-mode1}) and (\ref{sect4-hat-g1-mode1}), we have the desired estimate of ${\tilde \Phi}_1$ at mode $1$.

Focusing on other modes arising from $h_{\perp},$ it is straightforward to derive the desired estimates by using the maximum principle, which was shown in \cite{KWX2022} and we omit the argument here.
\end{proof}
%The inner norm is defined by
%\begin{align*}
 %     \|h\|_{\delta_1,\nu}: = \sup_{y \in \R^2}\varepsilon^{-\delta_1}\vert h\vert{(1 + |y|)^{\nu}},
  %    \end{align*}
   %   where $\delta_1>0$, $\nu>0.$
However, Lemma \ref{sect4-lemmainner-interior} is valid only when the location of spot is assumed to be in the interior of domain $\Omega.$  If $\xi\in\partial\Omega,$ we must modify Lemma \ref{sect4-lemmainner-interior} and impose new orthogonality conditions.  To construct the boundary spot, we first define the inner norm in $\mathbb R_{+}^2$ as
  \begin{equation*}
     \| h\|_{\delta_2,\nu_2,H} = \varepsilon^{-\delta_2}\sup_{y \in \R^2_{+}} |h|(1 + |y|)^{\nu_2},
  \end{equation*}
  where $\delta_2>0$, $\nu_2>0.$
Then, with the help of even reflection and Lemma \ref{sect4-lemmainner-interior}, we establish the following linear theory in $\mathbb R_{+}^2:$
  \begin{lemma}\label{sect4-linear-theory-bdry}
Given any function $h(y)$ and $\beta(y)$ satisfying
    \begin{equation}\label{331}
       \int_{\R^2_{+}}h dy  - \int_{\partial \R^2_{+}} \beta dS_y= 0,  \ \   \int_{\R^2_+}hy_1dy  - \int_{\partial\R^2_+}\beta y_1dS_y  = 0,
    \end{equation}
and assume $\|h\|_{\delta_2,4 + \sigma, H} < \infty$ with $\delta_2>0$ and $\sigma\in(0,1)$.  Then, we have the problem
    \begin{equation}\label{sect4-problem-bdry-1}
  \begin{cases}
    L_W[\Phi] = h, &y\in\R^2_{+},\\
     W\frac{\partial g}{\partial \boldsymbol{\nu}} = \beta,~~g=\frac{\Phi}{W}-\Psi,&y\in\partial \R^2_{+}
  \end{cases}
\end{equation}
    admits a solution $\Phi ={\mathcal T^H_{\text{i}}}[h]$ satisfying the following estimate: 
      \begin{equation*}
        \|\Phi\|_{\delta_2,2 + \sigma, H} \lesssim  \|h\|_{\delta_2,4 + \sigma, H},
      \end{equation*}  
      where ${ \mathcal T_{\text{i}}^H}[h]$ is a continuous linear operator.   
\end{lemma}
\begin{proof}
We refer the reader to Lemma 3.2 in \cite{KWX2022}.  For the sake of completeness, we exhibit the sketch of proof.  Our strategy is to define the intermediate variable $\eta$ such that 
\begin{align*}
\int_{\mathbb R_{+}^2}W\nabla \eta\cdot {\bf e}_1dy=0,
\end{align*}
where ${\bf e}_1=(1,0).$  Define $g_{N}:=g-\eta$, then we transform system (\ref{sect4-problem-bdry-1}) as the following form:
  \begin{align}\label{sect4-problem-bdry-2}
\left\{\begin{array}{ll}
\nabla\cdot (W \nabla {g_N}) = h - \nabla \cdot (W\nabla \eta),& y\in \R^2_{+}\\
    W \frac{\partial {g_N}}{\partial \boldsymbol{\nu}} = 0, &y\in\partial \R^2_{+},
    \end{array}
    \right.
\end{align}
Next, we perform the even reflection and define ${\bar g}_N$ as 
\begin{equation}\label{sect4-barg-new}
   {\bar {g}}_N :=
   \begin{cases}
      {g_N}(y_1, y_2),  &  y_2 \ge 0; \\
      {g_N}(y_1, -y_2), & y_2 < 0.
   \end{cases}
\end{equation}
Thanks to (\ref{sect4-barg-new}), one has (\ref{sect4-problem-bdry-2}) can be evenly extended into $\mathbb R^2$ and the form is shown as follows:  
 \begin{equation}\label{sect4-WbargN-whole-space}
    \nabla\cdot(W \nabla {\bar{g}}_N) = \bar {h},~~~y\in \R^2,
 \end{equation}
where $\bar h$ is given by
 \begin{equation*}
      \bar {h}=
     \begin{cases}
         h(y_1, y_2) - \nabla\cdot (W\nabla \eta)(y_1, y_2), & y_2 \ge 0, \\
         h(y_1, -y  _2) - \nabla\cdot (W\nabla \eta)(y_1, -y_2), & y_2 < 0.
     \end{cases}
 \end{equation*}
We wish to apply the results of Lemma \ref{sect4-lemmainner-interior} on \eqref{sect4-WbargN-whole-space}.  To this end, it is necessary to verify the orthogonality conditions given by (\ref{innercon}).  For the mass condition, noting that $\bar h$ is the even function, one finds
 \begin{equation*}
    \int_{\R^2_-} h(y_1, -y_2) - \nabla\cdot (W\nabla \eta)(y_1, -y_2)dy = \int_{\R^2_+} h(y_1, y_2) - \nabla\cdot (W\nabla \eta)(y_1, y_2)dy.
 \end{equation*}
 %and
 %\begin{equation*}
  % \int_{\R^2_-} \big[h(y_1, -y_2) - \nabla (U\nabla \vartheta)(x_1, -x_2)\big]y_1 dy= \int_{\R^2_+} \big[h(y_1, y_2) - \nabla (U\nabla \vartheta)(y_1, y_2)\big]y_1dy.
 %\end{equation*}
 Then we use the divergence theorem to obtain from (\ref{331}) that
 \begin{equation}\label{sect4-mass-condition-equiv}
 \begin{split}
   \int_{\R^2}\bar{h} dy &=  2\int_{\R^2_+}[h - \nabla\cdot(W\nabla \eta) ]dy  = 2\int_{\R^2_+} hdy- 2\int_{\partial\R^2_+} (W \nabla \eta)\cdot \boldsymbol{\nu}dS_y  \\
        & = 2\int_{\R^2_+} h dy - 2\int_{\partial\R^2_+} W \frac{\partial \eta}{\partial \boldsymbol{\nu}}dS_y =  2\int_{\R^2_+} h dy-   2\int_{\partial\R^2_+} \beta dS_y = 0,
   \end{split}
 \end{equation}
 which verifies the mass condition.  On the other hand, we deduce from the even property of $\bar h$ that 
 \begin{equation*}
   \int_{\R^2_-} \big[h(y_1, -y_2) - \nabla\cdot (W\nabla \eta)(y_1, -y_2)\big]y_1 dy= \int_{\R^2_+} \big[h(y_1, y_2) - \nabla\cdot  (W\nabla \eta)(y_1, y_2)\big]y_1dy.
 \end{equation*}
 It similarly follows that
\begin{align}\label{sect4-center-of-mass-condition-equiv}
   \int_{\R^2_+}\bar{h}y_1dy &= 2 \int_{\R^2_+}[h - \nabla\cdot (W \nabla \eta)]y_1dy  = 2\int_{\R^2_+}h y_1 dy- 2 \int_{\partial \R^2_+}\beta y_1dS_y = 0,
\end{align}
which verifies the center of mass condition.  By using (\ref{sect4-mass-condition-equiv}) and (\ref{sect4-center-of-mass-condition-equiv}), we utilize Lemma \ref{sect4-lemmainner-interior} to find there exists the solution $(\Phi,\Psi)$ to \eqref{sect4-problem-bdry-1} such that
    \begin{equation*} 
       \|\Phi\|_{\delta_2,2 + \sigma, H} \lesssim  \|h\|_{\delta_2,4 + \sigma, H},~~~\delta_2>0,\sigma\in(0,1).
    \end{equation*}
 It completes the proof of Lemma.
\end{proof}
Lemma \ref{sect4-lemmainner-interior} and Lemma \ref{sect4-linear-theory-bdry} help us establish the inner linear theory of linearized Keller-Segel operator.  It is natural to discuss the corresponding property of outer operator $L^{\text{o}}$ defined by \eqref{sect2-linearized-outer-operator} next.  We assume $\varphi(x)=\varphi(y)$ and rewrite $L^{\text{o}}$ in $y$-variable to obtain the operator becomes
\begin{align}\label{sect4-outer-problem-new-1}
 \bar L^o[\varphi]:=\varepsilon^2 L^o[\varphi]=&\Delta_y \varphi-\nabla_y \varphi\cdot \nabla_y \bar V-\varepsilon^2\bar V \varphi,
\end{align}
where $\bar V:=\Gamma+ H^{\varepsilon}.$  By studying (\ref{sect4-outer-problem-new-1}), we define the outer norm as
 \begin{equation*}
      \|h\|_{\delta_3,\nu_3, o}: = \varepsilon^{-\delta_3}\sup_{y \in \Omega_{\varepsilon}}\frac{ |h|}{ (1 + |y - \xi'|)^{-\nu_2}},~~ \delta_3,\nu_3>0,
    \end{equation*}
then formulate the following outer linear theory:
\begin{lemma}\label{sect4-outer-problem-linear-theory}
Assume that $\|h\|_{b+2, o} < \infty$, then the problem
     \begin{equation*}
     \begin{cases}
         \bar L^{o}[\varphi]=h, &y\in \Omega_{\varepsilon}, \\
        \frac{\partial \varphi}{\partial \boldsymbol{\nu}} = 0, &y\in \partial \Omega_{\varepsilon}
        \end{cases}
     \end{equation*}
    admits the solution $\varphi=\mathcal{T}_o(h)$ satisfying 
      \begin{equation}\label{sect4-outer-problem-estimate-1}
            \|\varphi\|_{\delta_3, b, o} \lesssim \|h\|_{\delta_3, b +2, o}, 
      \end{equation}
      where $\delta_3$ and $b$ are positive constants; moreover, $\mathcal{T}_o(h)$ is a continuous linear mapping.
\end{lemma}
\begin{proof}
The proof of this Lemma is the slight modification of Lemma 3.3, \cite{KWX2022}.  We shall construct the barrier function then apply the maximum principle to show the estimate (\ref{sect4-outer-problem-estimate-1}).

Define the barrier function $w$ as
  \begin{equation*}
      w =  w_{b} + w_0+\bar w_1 =   \frac{C_1}{(\mu^2+|y - \xi'|)^{b}} +  C_2 w_0+C_3\bar w_1.
  \end{equation*}
Here $C_1$, $C_2$, $C_3$ are positive constants and $w_0$, $\bar w_1$ are functions.  We will explain and determine them later on.  Since the location satisfies $\xi'_j \in \partial \Omega_\e$, we rewrite $\partial \Omega_\e$ near $\xi'_j$ as the graph $(y_1, y_2) = (y_1, \rho(y_1))$ with $\rho(0) = 0$ and $\rho'(0) = 0$, then find 
 \begin{align}\label{sect4-outer-problem-barrier-function-1}
       \frac{\partial w_b}{\partial \boldsymbol{\nu}} =&- \frac{b}{\vert y-\xi'_j\vert(\mu^2+|y-\xi'|)^{1 + b}}\cdot\frac{(y_2-\xi_{2} ')- (y_1-\xi_{1}') \rho'(y_1)}{\sqrt{1 + |\rho'(y)|^2}}\nonumber\\
       =&O(\e^{b+1}).
     \end{align}
To guarantee the boundary condition, we choose $w_0$ as the unique solution to the following problem
  \begin{align}\label{sect4-outer-problem-NBC-1}
     \left\{\begin{array}{ll}
      -\Delta w_0 + \nabla\bar V \cdot  \nabla w_0 + \e^2G(y, \xi_j') w_0 =0, &y\in \Omega_{\e},\\
     \frac{\partial w_0}{\partial \boldsymbol{\nu}}=-\frac{\partial w_b}{\partial \boldsymbol{\nu}},&y\in \partial\Omega_{\e}.
     \end{array}
     \right.
  \end{align}
  Moreover, one has $w_0$ satisfies $w_0\lesssim \e^b.$  Letting $\bar w_1:=\e^b,$ we calculate to get for $y\in\Omega_{\e},$
  \begin{equation*}
   \begin{split}
      L^{o}[w] =&\Bigg(-\frac{C_1b(b+1)}{(\mu^2+\vert y-\xi'\vert)^{b+2}}+\frac{C_1b}{(\vert y-\xi'\vert)(\mu^2+\vert y-\xi'\vert)^{b+1}}\\
     & + \frac{4C_1 b\vert y - \xi'\vert^2}{(\mu^2+|y - \xi'|)^{b+2}(\mu^2+\vert y-\xi'\vert^2)}\Bigg)- \frac{C_1 b(y - \xi')}{(1+|y - \xi'|)^{b+2}}4\e\pi \nabla H^{\e}\\
     &+ \frac{C_1 \e^2 G(y, \xi')}{(1+|y - \xi'|)^{b}}+C_3 \e^2 G(y, \xi_j') {\bar w}_1\\
     \geq &\frac{C_1b(4-b)}{(\mu^2+\vert y-\xi'\vert)^{b+2}}- \frac{C_1 b(y - \xi')}{(1+|y - \xi'|)^{b+2}}4\e\pi  \nabla H^{\e}\\
     &+ \frac{C_1 \e^2 G(y, \xi_j')}{(1+|y - \xi'|)^{b}}+C_3 \e^2 G(y, \xi_j')  {\bar w}_1\geq \frac{C_4}{(1 + |y - \xi'|)^{b + 2}},
        \end{split}
   \end{equation*}
   where $C_4>0$ is a positive constant.  On the other hand, upon combining (\ref{sect4-outer-problem-barrier-function-1}) and (\ref{sect4-outer-problem-NBC-1}), we find $\frac{\partial w}{\partial \boldsymbol{\nu}}=0$ on $\partial\Omega.$  Therefore, we use the maximum principle to show that
   \begin{equation*}
       \vert \varphi\vert   \lesssim \e^{\delta_3}\|h\|_{b+2, o} w, 
    \end{equation*}
    which gives us the desired estimate (\ref{sect4-outer-problem-estimate-1}). 
\end{proof}
With the preparation of error estimates shown in Section \ref{sect3} and linear theory shown in Section \ref{sect4} arising from the transported Keller-Segel system, we next investigate the existence of solution to the ${\bf u}$-equation of \eqref{PKSNS-ss}.  In particular, we shall establish the linear theory of Stokes operator. 
\section{Error Estimate and Linear Theory: Stokes Operator}\label{sect-model-stokes}
In this section, given the ansatz of $(n,c)$, we compute the error of ${\bf u}$-equation at first.  Recall that ${\bf u}$-equation of (\ref{PKSNS-ss-equiv}) is  
\begin{align}\label{sect5-u-equation-1}
{\bf u}\cdot\nb {\bf u} + \nb P = \De {\bf u} - \e_0\nabla \cdot(\nabla c\otimes\nabla c)+\e_0 \nabla\Big(\frac{|\nabla c|^2}{2}\Big)+\e_0\nabla\Big(\frac{c^2}{2}\Big).
\end{align}
In light of (\ref{sect2-n-and-c}), we have the forcing term in (\ref{sect5-u-equation-1}) becomes
\begin{align}\label{sect5-forcing-term-1}
\nabla_x\cdot(\nabla c\otimes \nabla c)=&\nabla\cdot[\nabla (\Gamma+H^{\e}+\psi)\otimes\nabla (\Gamma+H^{\e}+\psi)]\nonumber\\
=&\nabla_x\cdot (\nabla_x\Gamma\otimes \nabla_x\Gamma)+2\nabla_x\cdot(\nabla_x\Gamma\otimes \nabla H)+2\nabla_x\cdot(\nabla_x\Gamma\otimes \nabla_x \psi)\nonumber\\
&+\nabla_x\cdot(\nabla H\otimes \nabla H)+2\nabla_x\cdot(\nabla H\otimes \nabla_x\psi)+\nabla_x\cdot(\nabla_x\psi\otimes \nabla_x \psi).
\end{align}
It follows from (\ref{sect5-forcing-term-1}) that (\ref{sect5-u-equation-1}) can be written as
\begin{align}\label{sect5-bfu-decompose-before}
&{\bf u}\cdot\nb {\bf u} + \nb \Big(P-\e_0\frac{|\nabla c|^2}{2}-\e_0\frac{c^2}{2}\Big)\nonumber\\
=&\De_x {\bf u} - \e_0\Big[\nabla_x\cdot (\nabla_x\Gamma\otimes \nabla_x\Gamma)+2\nabla_x\cdot(\nabla_x\Gamma\otimes \nabla H)+2\nabla_x\cdot(\nabla_x\Gamma\otimes \nabla_x \psi)\nonumber\\
&+\nabla_x\cdot(\nabla H\otimes \nabla H)+2\nabla_x\cdot(\nabla H\otimes \nabla_x\psi)+\nabla_x\cdot(\nabla_x\psi\otimes \nabla_x \psi)\Big].
\end{align}
Upon letting $\tilde P:=P-\e_0\frac{|\nabla c|^2}{2}-\e_0\frac{c^2}{2}$ and
  \begin{align*}
  \mathbb F:=& - \e_0\Big[ \nabla_x\Gamma\otimes \nabla_x\Gamma+2(\nabla_x\Gamma\otimes \nabla H)+2(\nabla_x\Gamma\otimes \nabla_x \psi)\nonumber\\
&+\nabla H\otimes \nabla H+2(\nabla H\otimes \nabla_x\psi)+\nabla_x\psi\otimes \nabla_x \psi\Big],
  \end{align*}
  we obtain (\ref{sect5-bfu-decompose-before}) becomes
\begin{align*}
&{\bf u}\cdot\nb_x {\bf u} + \nb_x \tilde P=\De_x {\bf u}+\nabla_x\cdot  \mathbb F.
\end{align*}
Considering the divergence-free and boundary conditions, we rewrite the ${\bf u}$-equation in \eqref{PKSNS-ss-equiv} as 
\begin{align}\label{sect5-bfu-decompose-before-formulate}
\left\{\begin{array}{ll}
{\bf u}\cdot\nb_x {\bf u} + \nb_x \tilde P=\De_x {\bf u}+\nabla_x\cdot  \mathbb F, &x\in\Omega,\\
\nabla_x\cdot {\bf u} =0,&x\in\Omega,\\
{\bf u}\cdot \boldsymbol{\nu}=0,&x\in\partial\Omega,\\
({\mathbb S} {\bf u}\cdot \boldsymbol{\nu})_{\boldsymbol{\tau}}=0,&x\in\partial\Omega.
\end{array}
\right.
\end{align}
It is straightforward to see that in $\mathbb R^2,$ $\mathbb F$ is a $2\times 2$ matrix, i.e. $\mathbb F=(F_{ij})_{i,j=1,2}$.  Of concern the existence of the solution to system (\ref{sect5-bfu-decompose-before-formulate}), we find it is necessary to study the corresponding homogeneous problem and impose the orthogonality conditions on (\ref{sect5-bfu-decompose-before-formulate}) if the non-trivial kernel exists.  Indeed, as shown in Section \ref{sect6-W2p-estimate-Stokes}, more precisely Theorem \ref{thm-4.1-AR}, if the following compatibility conditions hold:
\begin{align}\label{sect5-solvability-conditions-eq}
\int_{\Omega} ({\bf u}\cdot \nabla_x {\bf u}) \cdot\boldsymbol{\beta} \, dx=\int_{\partial\Omega}({\mathbb F}\cdot\boldsymbol{\nu})_{\boldsymbol{\tau}}\cdot\boldsymbol{\beta}\, dS-\int_{\Omega} \mathbb F:\nabla \boldsymbol{\beta}\, dx,
\end{align}
and
\begin{align}\label{sect5-solvability-condition-divergence}
\int_{\Omega} \nabla_x \cdot {\bf u}\, dx=\int_{\partial\Omega}{\bf u}\cdot \boldsymbol{\nu}\, dS,
\end{align}
where $\boldsymbol{\be} = c{\bf x}^\perp + {\bf b}$ and ${\bf x}^\perp=(-x_2,x_1)^T$ with $c\not=0$ and ${\bf b}$ being a constant vector, we have there exists a unique solution $\bf u$ to system \eqref{sect5-bfu-decompose-before-formulate}.  Thus, for the proof of the existence, it suffices to verify (\ref{sect5-solvability-conditions-eq}) and (\ref{sect5-solvability-condition-divergence}).  On one hand, since ${\bf u}\cdot \boldsymbol{\nu}=0$ on $\pd \Omega$, one obviously has (\ref{sect5-solvability-condition-divergence}) holds.  On the other hand, we compute
\begin{align}\label{sect5-original-ueq-int1}
\int_{\Omega}(\nabla_x\cdot {\mathbb F})\cdot\boldsymbol{\beta}\, dx=\int_{\partial\Omega}({\mathbb F}\cdot\boldsymbol{\nu})_{\boldsymbol{\tau}}\cdot\boldsymbol{\beta}\, dS-\int_{\Omega} \mathbb F:\nabla \boldsymbol{\beta}\, dx,
\end{align}
and use $\nabla\cdot {\bf u}=0$ in $\Omega$ to get
\begin{align}\label{sect5-original-ueq-int2}
\int_{\Omega}({\bf u}\cdot \nabla_x {\bf u})\cdot\boldsymbol{\beta}\, dx=&\int_{\Omega}[\nabla_x\cdot ({\bf u}\otimes {\bf u})]\cdot\boldsymbol{\beta}\, dx\nonumber\\
=&\int_{\partial\Omega}[({\bf u}\otimes {\bf u})\cdot\boldsymbol{\nu}]_{\boldsymbol{\tau}}\cdot\boldsymbol{\beta}\, dS-\int_{\Omega} ({\bf u}\otimes {\bf u}):\nabla \boldsymbol{\beta}\, dx.
\end{align}
Next, we calculate the integrals term by term.  Recall that $\mathbb F=\nabla c\otimes \nabla c,$ which is 
\begin{align*}
\mathbb F=\left( \begin{array}{cc}
    (\partial_{x_1} c)^2 & \partial_{x_1}c\partial_{x_2}c \\
 \partial_{x_1}c\partial_{x_2}c &  (\partial_{x_2} c)^2\\
    \end{array}
  \right).
\end{align*}
Then we have 
\begin{align*}
\mathbb F\cdot \boldsymbol{\nu}=\left( \begin{array}{c}
    \partial_{x_1} c(\partial_{x_1}c\nu_1+\partial_{x_2}c\nu_2)\\
   \partial_{x_2} c(\partial_{x_1}c\nu_1+\partial_{x_2}c\nu_2)\\
    \end{array}
  \right)=\left( \begin{array}{c}
    \partial_{x_1} c\frac{\partial c}{\partial \boldsymbol{\nu}}\\
 \partial_{x_2} c\frac{\partial c}{\partial \boldsymbol{\nu}}\\
    \end{array}
  \right),
\end{align*}
where $\boldsymbol{\nu}=(\nu_1,\nu_2)^T.$  Noting that $c$ satisfies the Neumann boundary condition, one further obtains
\begin{align}\label{sect5-original-ueq-int3}
\int_{\partial\Omega}({\mathbb F}\cdot\boldsymbol{\nu})_{\boldsymbol{\tau}}\cdot\boldsymbol{\beta}\, dS=0. 
\end{align}
Similarly, by using ${\bf u}\cdot \boldsymbol{\nu}=0$ on $\partial\Omega,$ we get
\begin{align}\label{sect5-original-ueq-int4}
\int_{\partial\Omega} [({\bf u}\otimes {\bf u})\cdot\boldsymbol{\nu}]_{\boldsymbol{\tau}}\cdot\boldsymbol{\beta}\, dS=0. 
\end{align}
In addition, since 
\begin{align*}
\nabla \boldsymbol{\beta}=\left(\begin{array}{cc}
   0 & -1 \\
1 &  0\\
    \end{array}
  \right),
\end{align*}
we find
\begin{align}\label{sect5-original-ueq-int5}
\int_{\Omega} \mathbb F:\nabla \boldsymbol{\beta}\, dx=\int_{\Omega} (\partial_{x_1}c\partial_{x_2}c- \partial_{x_1}c\partial_{x_2}c) \, dx=0.
\end{align}
Furthermore, the symmetry of $\bf u\otimes \bf u$ implies
\begin{align}\label{sect5-original-ueq-int6}
\int_{\Omega} ({\bf u}\otimes {\bf u}):\nabla \boldsymbol{\beta}\, dx=0.
\end{align}
 Upon collecting (\ref{sect5-original-ueq-int1}), (\ref{sect5-original-ueq-int2}), (\ref{sect5-original-ueq-int3}), (\ref{sect5-original-ueq-int4}), (\ref{sect5-original-ueq-int5}) and (\ref{sect5-original-ueq-int6}), one obtains \eqref{sect5-solvability-conditions-eq} is true.  Therefore, we have (\ref{sect5-bfu-decompose-before-formulate}) admits the unique solution $\bf u$ under the solvability conditions (\ref{sect5-solvability-conditions-eq}) and (\ref{sect5-solvability-condition-divergence}).
 
Next, we perform the error analysis.  Noting that the location $\xi$ of the single boundary spot is on $\partial\Omega,$ we have to flatten the boundary locally.  Indeed, one has for $|x-\xi|\lesssim \e,$
\begin{align*}
\Delta_x{\bf u}=&\Delta_{X}{\bf u}+(\rho'(X_1))^2\partial_{X_2X_2}{\bf u} -2\rho'(X_1)\partial_{X_2X_1}{\bf u}-\rho''(X_1)\partial_{X_2}{\bf u}\nonumber\\
=&\frac{1}{\e^2}\Delta_Y {\bf u}+\frac{1}{\e^2}(\rho'(\e Y_1))^2\partial_{Y_2Y_2}{\bf u}-\frac{2}{\e^2}\rho'(\e Y_1)\partial_{Y_2Y_1}{\bf u}-\frac{1}{\e}\rho''(\e Y_1)\partial_{Y_2}{\bf u},
\end{align*}
\begin{align*}
    \nabla_x\tilde P = &\Big(\partial_{X_1}\tilde P - \rho'(X_1)\partial_{X_2}\tilde P, \partial_{X_2}\tilde P \Big)\nonumber\\
    =&\nabla_{X}\tilde P- (\rho'(X_1)\partial_{X_2}\tilde P,0)\nonumber\\
    =&\frac{1}{\e}\nabla_{Y}\tilde P- \frac{1}{\e}(\rho'(\e Y_1)\partial_{Y_2}\tilde P,0).
    \end{align*}
and
\begin{align*}
  \nabla_x\cdot {\mathbb F} =& \Big(\frac{\partial}{\partial X_1}, \frac{\partial}{\partial X_2} \Big)\cdot \mathbb F- \Big( \frac{\partial}{\partial X_2}\cdot \rho'(X_1),0\Big)\cdot\mathbb F\nonumber\\
  =&\frac{1}{\e}\nabla_Y \cdot {\mathbb F}- \frac{1}{\e}\Big( \frac{\partial}{\partial Y_2}\cdot \rho'(\e Y_1),0\Big)\cdot\mathbb F.
  \end{align*}
Since $\rho'$ and $\rho''$ satisfy \eqref{sect3-rhoprime-Y1} and \eqref{sect3-rhodouble-Y1}, we find in the inner region,  
\begin{align}\label{sect5-inner-region-error}
\e^2\Delta_x {\bf u}=\Delta_Y {\bf u}+o(1),\quad \e \nabla_x \tilde P=\nabla_Y\tilde P+o(1),\quad \e \nabla_x\cdot \mathbb F=\nabla_Y\cdot \mathbb F+o(1).
\end{align}
In addition, considering the divergence-free condition $\nabla_x\cdot {\bf u}=0$ in $\Omega$ and Naiver boundary condition $(\mathbb S{\bf u}\cdot \boldsymbol{\nu})_{\boldsymbol{\tau}}=0$ on $\partial\Omega,$ we have in the inner region,
\begin{align}\label{sect5-inner-region-error-boundary}
\e\nabla_x \cdot {\bf u}=\nabla_Y\cdot {\bf u}+o(1),\quad \e (\mathbb S_x{\bf u}\cdot \boldsymbol{\nu})_{\boldsymbol{\tau}}=(\mathbb S_Y{\bf u}\cdot \boldsymbol{\nu})_{\boldsymbol{\tau}}+o(1).
\end{align}
Moreover, we decompose ${\bf u}$ as 
\begin{align}\label{sect5-bfu-decompose-1}
{\bf u}=\e{\bf u}^{\text{i}}\chi(y)+{\bf u}^{\text{o}}.
\end{align}
By using (\ref{sect5-inner-region-error}), (\ref{sect5-inner-region-error-boundary}) and (\ref{sect5-bfu-decompose-1}), we formulate the inner problem satisfied by ${\bf u}^{\text{i}}$ as
\begin{align}\label{sect5-inner-problem-velocity}
\left\{\begin{array}{ll}
\nabla P_1=\Delta {\bf u}^{\text{i}}+\nabla\cdot \mathbb F,&Y \in \mathbb R_{+}^2,\\
\nabla\cdot {\bf u}^{\text{i}}=0,&Y \in \mathbb R_{+}^2,\\
\partial_{Y_2} { u}^{\text{i}}_1= { u}^{\text{i}}_2=0, &Y \in\partial\mathbb R_{+}^2,
\end{array}
\right.
\end{align}
where ${{\bf u}^{\text{i}}}=(u^{\text{i}}_1,u^{\text{i}}_2)$. 
 Correspondingly, noting that
\begin{align*}
\Delta( {\bf u}^{\text{i}}\chi)=\Delta {\bf u}^{\text{i}}\chi+2\nabla {\bf u}^{\text{i}}\cdot\nabla \chi+{\bf u}^{\text{i}}\Delta\chi,
\end{align*}
one establishes the following outer problem of ${\bf u}^{\text{o}}$:
\begin{align}\label{sect5-outer-problem-velocity}
\left\{\begin{array}{ll}
\nabla_x (\tilde P-\chi P_1)=\Delta_x{\bf u}^{\text{o}}+(1-\chi)\nabla_x\cdot \mathbb F+2\e\nabla_x\chi\cdot\nabla_x {\bf u}^{\text{i}}+\e(\Delta_x \chi){\bf u}^{\text{i}}- P_1\nabla_x \chi,&x\in\Omega,\\
\nabla_x\cdot {\bf u}^{\text{o}}=-\e\nabla_x \chi\cdot {\bf u}^{\text{i}},&x\in\Omega,\\
{\bf u}^{\text{o}}\cdot \boldsymbol{\nu}=0,&x\in\partial\Omega,\\
(\mathbb S_x({\bf u}^{\text{o}})\cdot\boldsymbol{\nu})_{\boldsymbol{\tau}}=-(\mathbb S_x(\e{\bf u}^{\text{i}}\chi)\cdot\boldsymbol{\nu})_{\boldsymbol{\tau}},&x\in\partial\Omega.
\end{array}
\right.
\end{align}
After finishing the error estimate, we shall first focus on (\ref{sect5-inner-problem-velocity}) and establish the inner linear theory, then find the solution to (\ref{sect5-outer-problem-velocity}) by using the $W^{2,p}$ estimate of Stokes equation.   

To study the inner problem (\ref{sect5-inner-problem-velocity}), our idea is to perform the extension of velocity field ${\bf u}^{\text{i}}$ and solve the whole space problem, then derive the desired pointwise decay estimate by the representation formula.  To begin with, we state the following useful lemma.
\begin{lemma}[\text{\cite[Lemma 2.2]{KLLT-CMP2023}}]
\label{lem-KLLT-integral-est-1}
Let $a>0$, $b>0$, $k>0$, $m>0$ and $k+m>N$. Define
\[I:=\int_{\R^N}\frac{dz}{(|z|+a)^k(|z-x|+b)^m},\quad x\not=0.\]
Then, with $R=\max\{|x|,\,a,\,b\}\sim|x|+a+b$, we have
$$I\lesssim R^{N-k-m} + \de_{kN} R^{-m} \log \frac Ra
+ \de_{mN} R^{-k} \log \frac Rb
+ { \mathbbm 1}_{k>N} R^{-m}a^{N-k}
+ {\mathbbm 1}_{m>N} R^{-k}b^{N-m},
$$
where $\delta_{\alpha\beta}$ and $\mathbbm 1_{\alpha>\beta}$ are defined by
\begin{align*}
\delta_{\alpha\beta}=\left\{\begin{array}{ll}
1,&\alpha=\beta,\\
0,&\alpha\not=\beta;
\end{array}
\right.\quad 
\mathbbm 1_{\alpha>\beta}=\left\{\begin{array}{ll}
1,&\alpha>\beta,\\
0,&\text{otherwise.}
\end{array}
\right.
\end{align*}
\end{lemma}
By using Lemma \ref{lem-KLLT-integral-est-1}, we are able to derive the decay estimate satisfied by ${\bf u}^{\text{i}}$ in \eqref{sect5-inner-problem-velocity}. 
Before stating our result, we define the norm $\Vert \cdot \Vert_{S,\gamma-2,a+2}$ satisfied by the forcing term $\mathbb F$ as
 \begin{align}\label{sect5-stokes-norm-def}
 \Vert \mathbb F\Vert_{S,\gamma-2,a+2}:={\varepsilon^{-(\gamma-2)}}\sup_{Y\in\mathbb R_{+}^2}\Vert \mathbb F\Vert_{\infty}(1+|Y|^{a+2}),
 \end{align}
 where $\gamma,$ $a>0$ and $\Vert\mathbb F\Vert_{\infty}$ is given by
 \begin{align*}
\Vert \mathbb F\Vert_{\infty}(Y):=\max_{i,j=1,2}\vert F_{ij}(Y)\vert\text{~for~}Y\in\mathbb R_{+}^2.
\end{align*}
 Thanks to (\ref{sect5-stokes-norm-def}), we summarize our result as the following lemma:  
\begin{lemma}\label{sect5-lemmainner-interior}
Let $\gamma\in(0,1)$ and $a\in(0,1)$ be constants.  Assume that $\Vert{\mathbb F}\Vert_{S,\nu-2,a+2}<\infty,$ then the solution $({\bf u^{\text{i}}},P_1)$ of system \eqref{sect5-inner-problem-velocity} satisfies
\begin{align}\label{sect5-linear-theory-decay-estimate-u}
\max_{j=1,2}|{{\bf u}^{\text{i}}}_j(Y)|\lesssim \Vert \mathbb F\Vert_{S,\gamma-2,a+2}\frac{\varepsilon^{\gamma-2}}{1+\vert Y\vert},
\end{align}
and
\begin{align}\label{sect5-linear-theory-decay-estimate-P1}
    \vert P_1(Y)\vert \lesssim \Vert \mathbb F\Vert_{S,\gamma-2,a+1} \frac{\varepsilon^{\gamma-2}}{1+|Y|^2}.
\end{align}
\end{lemma}
\begin{proof}
First of all, we define ${\bf E}\mathbb F$ as an extension of $\mathbb F$ to $\R^2$, which is
\[
{\bf E}{\mathbb F}(Y_1,Y_2) = 
\begin{bmatrix}
F_{11}&-F_{12}\\
-F_{21}&F_{22}
\end{bmatrix}
(Y_1,-Y_2)\ \text{ for }Y_2<0.
\]
Then, we let $ {\bf {\td u^{\text{i}}}}$ be the solution of the following Stokes equation in the whole space $\mathbb R^2$:
\begin{align}\label{sect5-eq-stokes-divF-extension}
\left\{\begin{array}{ll}
\nb \td Q_1 = \De  {\bf {\td u}^{\text{i}}} + \nb\cdot({\bf E}{\mathbb F}), &Y\in\R^2, \\
\nb\cdot {\bf {\td u}^{\text{i}}}= 0,& Y\in\R^2.
\end{array}
\right.
\end{align}
We have from the representation formula that the solution to (\ref{sect5-eq-stokes-divF-extension}) is
\EQ{\label{sect5-eq-solution-formula-stokes-divF}
{\bf {\td u}^{\text{i}}}_i(Y) 
&=\sum_{j=1}^2 \int_{\R^2} U_{ij}(Y-z) (\nb\cdot({\bf E}\mathbb F))_j(z)\, dz
=\sum_{j,k=1}^2 \int_{\R^2} U_{ij}(Y-z) \pd_{z_k}({\bf E}{\mathbb F})_{jk}(z)\, dz,\\
{\td Q}_1(Y) &= \sum_{j=1}^2 \int_{\R^2} Q_j(Y-z) (\nb\cdot({\bf E}{\mathbb F}))_j(z)\, dz
=\sum_{j,k=1}^2 \int_{\R^2} Q_j(Y-z) \pd_{z_k}({\bf E}{\mathbb F})_{jk}(z)\, dz,
}
where $(U_{ij}, Q_j)$, $i,j=1,2$ is the Lorentz tensor given by
\EQN{
U_{ij}(x) &= - \frac1{4\pi} \de_{ij} \log|x| + \frac1{4\om_2}\, \frac{x_ix_j}{|x|^2},\ \
Q_j(x)  = \frac{x_j}{2\om_2|x|^2},
}
and $\omega_2$ denotes the measure of the unit ball in $\mathbb R^2.$
%\EQN{
%U_{ij}(x) &= \frac12 \de_{ij} E(x) + \frac1{2n\om_n}\, \frac{x_ix_j}{|x|^n}
%= - \frac1{4\pi} \de_{ij} \log|x| + \frac1{4\om_2}\, \frac{x_ix_j}{|x|^2},\\
%Q_j(x) &= \frac{x_j}{n\om_n|x|^n} = \frac{x_j}{2\om_2|x|^2}. 
%}

Now, we verify that the restriction ${\bf {\td u}^{\text{i}}}\vert_{\R_{+}^2}$ is the solution to \eqref{sect5-inner-problem-velocity}. It is straightforward to show that the restriction ${\bf {\td u}^{\text{i}}}\vert_{\R_{+}^2}$ satisfy Stokes equation \eqref{sect5-inner-problem-velocity}$_1$.  Next, we check the boundary condition \eqref{sect5-inner-problem-velocity}$_2$.
Indeed, we rewrite \eqref{sect5-eq-solution-formula-stokes-divF}$_1$ as the following form: 
\EQN{
{\bf {\td u}^{\text{i}}}_1(Y) &= \int_{\R^2} \bke{-\frac1{4\pi} \log|Y-z| + \frac1{4\om_2}\, \frac{(Y_1-z_1)^2}{|Y-z|^2}} (\nb\cdot({\bf E}{\mathbb F}))_1(z)\, dz\\
&\quad + \int_{\R^2} \frac1{4\om_2}\, \frac{(Y_1-z_1)(Y_2-z_2)}{|Y-z|^2}\, (\nb\cdot({\bf E}{\mathbb F}))_2(z)\, dz,\\
{\bf {\td u}^{\text{i}}}_2(Y) &= \int_{\R^2} \frac1{4\om_2}\, \frac{(Y_2-z_2)(Y_1-z_1)}{|Y-z|^2}\, (\nb\cdot({\bf E}{\mathbb F}))_1(z)\, dz\\
&\quad + \int_{\R^2} \bke{-\frac1{4\pi} \log|Y-z| + \frac1{4\om_2}\, \frac{(Y_2-z_2)^2}{|Y-z|^2}} (\nb\cdot({\bf E}{\mathbb F}))_2(z)\, dz.
}
Noting that 
\EQN{(\nb\cdot({\bf E}{\mathbb F}))_1(z) =& \pd_{z_1}({\bf E}{\mathbb F})_{11}(z) + \pd_{z_2}({\bf E}{\mathbb F})_{12}(z),\\
(\nb\cdot({\bf E}F))_2(z) = &\pd_{z_1}({\bf E}F)_{21}(z) + \pd_{z_2}({\bf E}F)_{22}(z),}
we find $(\nb\cdot({\bf E}{\mathbb F}))_1$ and $(\nb\cdot({\bf E}{\mathbb F}))_2$ are even and odd in $z_2$, respectively.  Then, we obtain
\EQN{
\pd_{Y_2} {\bf{\td  u}^{\text{i}}}_1(Y_1,0) &= \int_{\R^2} \bke{-\frac1{4\pi}\, \frac{-z_2}{(Y_1-z_1)^2 + z_2^2} - \frac2{4\om_2}\, \frac{(Y_1-z_1)^2(-z_2)}{\bke{ (Y_1-z_1)^2 + z_2^2 }^2 }} (\nb\cdot({\bf E}{\mathbb F}))_1(z)\, dz\\
&\quad + \int_{\R^2} \frac1{4\om_2} \bkt{\frac{(Y_1-z_1)\bke{(Y_1-z_1)^2 + z_2^2} - 2(Y_1-z_1)z_2^2}{\bke{(Y_1-z_1)^2 + z_2^2}^2}} (\nb\cdot({\bf E}{\mathbb F}))_2(z)\, dz
= 0,
}
and
\EQN{
{\bf{\td  u}^{\text{i}}}_2(Y_1,0)  &= \int_{\R^2} \frac1{4\om_2}\, \frac{-z_2(Y_1-z_1)}{(Y_1-z_1)^2 + z_2^2}\, (\nb\cdot({\bf E}{\mathbb F}))_1(z)\, dz\\
&\quad+ \int_{\R^2} \bke{-\frac1{4\pi} \log\sqrt{(Y_1-z_1)^2 + z_2^2} + \frac1{4\om_2}\, \frac{z_2^2}{(Y_1-z_1)^2 + z_2^2}} (\nb\cdot({\bf E}{\mathbb F}))_2(z)\, dz
= 0.
}
It follows that the boundary condition \eqref{sect5-inner-problem-velocity}$_2$ is satisfied, and hence ${\bf u^{\text{i}}} \equiv  {\bf {\td u^{\text{i}}}}|_{\R^2_+}$ solves \eqref{sect5-inner-problem-velocity}.

Our next aim is to show \eqref{sect5-linear-theory-decay-estimate-u} by using \eqref{sect5-eq-solution-formula-stokes-divF}.  In fact, we have from integration by parts formula that for ${\bf{\td u}^{\text{i}}}_i$,
\EQ{\label{sect5-eq-represent-formula-v}
{\bf{\td u}^{\text{i}}}_i(Y) 
= - \sum_{j,k=1}^2 \int_{\R^2} \pd_{z_k}U_{ij}(Y-z) ({\bf E}{\mathbb F})_{jk}(z)\, dz.
}
Noting that $\mathbb F$ satisfies $\Vert{\mathbb F}\Vert_{S,\nu-2,a+2}<\infty,$ we take 
\EQ{\label{sect5-eq-F-decay}
\Vert \mathbb F\Vert_{H,\infty}(z) \lec \frac{\e^{\gamma-2}}{1+|z|^{2+a}}, \ \ a\in(0,1).
}
Upon substituting (\ref{sect5-eq-F-decay}) into \eqref{sect5-eq-represent-formula-v}, one has
\begin{align}\label{sect5-integral-u-useful-1}
\max_{j=1,2}|{\bf{\td u}^{\text{i}}}_j(Y)| 
\lec &\int_{\R^2} \frac1{|Y-z|}\, \frac{\e^{\gamma-2}}{1+|z|^{2+a}}\, dz.
\end{align}
It is necessary to show that the RHS in (\ref{sect5-integral-u-useful-1}) is well-defined.  In fact, we find
\begin{align*}
\int_{\R^2} \frac1{|Y-z|}\, \frac{1}{1+|z|^{2+a}}\, dz=&\int_{\R^2} \frac1{|z|}\, \frac{1}{1+|z-Y|^{2+a}}\, dz\\
=&\int_{B_{0}(\delta)} \frac1{|z|}\, \frac{1}{1+|z-Y|^{2+a}}\, dz +\int_{\R^2\backslash B_{0}(\delta)} \frac1{|z|}\, \frac{1}{1+|z-Y|^{2+a}}\, dz\\
=&\frac{2\pi\delta}{1+|Y|^{2+a}}+\int_{\R^2\backslash B_{0}(\delta)} \frac1{|z|}\, \frac{1}{1+|z-Y|^{2+a}}\, dz.
\end{align*}
Hence,
\begin{align*}
\max_{j=1,2}|{\bf{\td u}^{\text{i}}}_j(Y)| \lec\e^{\gamma-2}\int_{\R^2} \frac1{|Y-z|}\, \frac{1}{1+|z|^{2+a}}\, dz\sim \e^{\gamma-2}\int_{\R^2} \frac{1}{\big(\frac{1}{2}+|z|\big)\big(|z-Y| +\frac{1}{2}\big)^{2+a}}\, dz.
\end{align*}
Then, we invoke Lemma \ref{lem-KLLT-integral-est-1} to find
\[
\max_{j=1,2}|{{\bf{\td u}}^{\text{i}}}_j(Y)| \lec \e^{\gamma-2} \bke{\frac1{(|Y|+1)^{1+a}} + \frac1{|Y|+1}}
\lec \frac{\e^{\gamma-2}}{1+\vert Y\vert}, 
\]
which gives the proof of (\ref{sect5-linear-theory-decay-estimate-u}).

Next, we derive the decay estimate for pressure $\tilde Q_1$.  {Recall that the representation formula for $\td Q_1$ in \eqref{sect5-eq-solution-formula-stokes-divF} is
\begin{align}\label{sect5-td-Q1-Y-formula}
\td Q_1(Y) = \sum_{j=1}^2 \int_{\R^2} Q_j(Y-z) (\nb_Y\cdot({\bf E}F))_j(z)\, dz,
\end{align}
where $|Q_j(x)| \lec |x|^{-1}$.  According to the assumption that $\Vert \mathbb F\Vert_{S,\nu-2,a+2}<\infty,$ we have
\begin{align}\label{sect5-grad-F-Q}
\max_{j=1,2}| (\nb_Y \cdot {\mathbb F})_j(z)| \lec \frac{\e^{\gamma-2}}{1+\vert z\vert^{3+a}}.
\end{align}
Upon substituting (\ref{sect5-grad-F-Q}) into (\ref{sect5-td-Q1-Y-formula}), one obtains 
\EQN{
|\td Q_1(Y)| \lec \int_{\R^2} \frac1{|Y-z|}\, \frac{\e^{\gamma-2}}{1+|z|^{3+a}}\, dz.}
Similarly, we have 
\EQN{\int_{\R^2} \frac1{|Y-z|}\, \frac{1}{1+|z|^{3+a}}\, dz=&\int_{\R^2} \frac1{|z|}\, \frac{1}{1+|z-Y|^{3+a}}\, dz\\
=&\int_{B_0(\delta)} \frac1{|z|}\, \frac{1}{1+|z-Y|^{3+a}}\, dz+\int_{\R^2\backslash B_0(\delta)} \frac1{|z|}\, \frac{1}{1+|z-Y|^{3+a}}\, dz\\
=&\frac{2\pi\delta}{1+|Y|^{3+a}}+\int_{\R^2\backslash B_0(\delta)} \frac1{|z|}\, \frac{1}{1+|z-Y|^{3+a}}\, dz,}
where $\delta>0$ is a small but fixed number.  Thus,
\EQN{
|\td Q_1(Y)| \lec \int_{\R^2} \frac1{|Y-z|}\, \frac{\e^{\gamma-2}}{1+|z|^{3+a}}\, dz
\sim \e^{\gamma-2}\int_{\R^2} \frac{1}{\big(\frac{1}{2}+|z|\big)\big(|z-Y| + \frac{1}{2}\big)^{3+a}}\, dz.
}
By using Lemma \ref{lem-KLLT-integral-est-1}, we further get
\[
|\td Q_1(Y)| \lec \e^{\gamma-2} \bke{\frac1{(|Y|+1)^{2+a}} + \frac1{|Y|+1}}
\lec \frac{\e^{\gamma-2}}{1+|Y|}. 
\]
}

Now, we derive the decay estimate of $\td Q_1$; however, by noting (\ref{sect5-linear-theory-decay-estimate-u}),  the algebraic decay rate satisfied by $\td Q_1$ is expected to be $2$.  To show this, we integrate the solution formula \eqref{sect5-td-Q1-Y-formula} of $\td Q_1$ by parts to get
\EQ{\label{sect5-eq-represent-formula-q}
\td Q_1(Y) 
=& \sum_{j,k=1}^2 \int_{\R^2} Q_j(Y-z) \pd_{z_k}({\bf E}{\mathbb F})_{jk}(z)\, dz\\
=&-\sum_{j,k=1}^2 \int_{\R^2} Q_j(z) \pd_{z_k}[({\bf E}{\mathbb F})_{jk}(Y-z)]\, dz\\
=&-\sum_{j,k=1}^2 \int_{B_0(\delta)} Q_j(z) \pd_{z_k}[({\bf E}{\mathbb F})_{jk}(Y-z)]\, dz-\sum_{j,k=1}^2 \int_{\R^2\backslash{B_0(\delta)}} Q_j(z) \pd_{z_k}[({\bf E}{\mathbb F})_{jk}(Y-z)]\, dz\\
=&-\sum_{j,k=1}^2 \int_{B_0(\delta)} Q_j(z) \pd_{z_k}[({\bf E}{\mathbb F})_{jk}(Y-z)]\, dz\\
&+\sum_{j,k=1}^2 \int_{\partial B_{0}(\delta)} Q_j(z) ({\bf E}{\mathbb F})_{jk}(Y-z)\cdot {\boldsymbol{\nu}}_{z}\, dS_z+\sum_{j,k=1}^2 \int_{\R^2\backslash B_0(\delta)} \pd_{z_k} Q_j(z) ({\bf E}{\mathbb F})_{jk}(Y-z)\, dz,
}
%\EQ{\label{sect5-eq-represent-formula-q-1}
%\td Q_1(Y) 
%=& \sum_{j,k=1}^2 \int_{\R^2} Q_j(Y-z) \pd_{z_k}({\bf E}{\mathbb F})_{jk}(z)\, dz\\
%=& \sum_{j,k=1}^2 \int_{B_Y(\delta)} Q_j(Y-z) \pd_{z_k}({\bf E}{\mathbb F})_{jk}(z)\, dz+\sum_{j,k=1}^2 \int_{\R^2\backslash{B_Y(\delta)}} Q_j(Y-z) \pd_{z_k}({\bf E}{\mathbb F})_{jk}(z)\, dz\\
%=& \sum_{j,k=1}^2 \int_{B_Y(\delta)} Q_j(Y-z) \pd_{z_k}({\bf E}{\mathbb F})_{jk}(z)\, dz\\
%&-\sum_{j,k=1}^2 \int_{\partial B_{Y}(\delta)} Q_j(Y-z) ({\bf E}{\mathbb F})_{jk}(z)\cdot {\boldsymbol{\nu}}_{z}\, dS_z- \sum_{j,k=1}^2 \int_{\R^2} \pd_{z_k} Q_j(Y-z) ({\bf E}{\mathbb F})_{jk}(z)\, dz,
%}
where ${\boldsymbol{\nu}}_z$ denotes the unit outer normal of $B_{0}(\delta).$  Noting that $| Q_j(x)|\lec|x|^{-1}$, we find
\begin{align}\label{sect5-term1-q-formula}
\Big\vert -\sum_{j,k=1}^2 \int_{B_0(\delta)} Q_j(z) \pd_{z_k}[({\bf E}{\mathbb F})_{jk}(Y-z)]\, dz\Big\vert \lec& \int_{B_{0}(\delta)} \frac1{|z|}\, \frac{\e^{\gamma-2}}{1+|z-Y|^{3+a}}\, dz\nonumber\\
=&\e^{\gamma-2}\frac{2\pi\delta}{1+|Y|^{3+a}},
\end{align}
and
\begin{align}\label{sect5-term2-q-formula}
\Big\vert\sum_{j,k=1}^2 \int_{\partial B_{0}(\delta)} Q_j(z) ({\bf E}{\mathbb F})_{jk}(Y-z)\cdot {\boldsymbol{\nu}}_{z}\, dS_z\Big\vert\lec &\int_{\partial B_0(\delta)}\frac{1}{|z|}\frac{\e^{\gamma-2}}{1+|z|^{2+a}}\, dz\nonumber\\
=&\frac{2\pi \e^{\gamma-2}}{1+\delta^{2+a}}.
\end{align}
Moreover, in light of $|\nb Q_j(x)|\lec|x|^{-2}$, combining (\ref{sect5-term1-q-formula}) and (\ref{sect5-term2-q-formula}), we apply \eqref{sect5-eq-F-decay} on \eqref{sect5-eq-represent-formula-q} and deduce that
\EQN{
|\td Q_1(Y)| \lec &\e^{\gamma-2}\frac{2\pi\delta}{1+|Y|^{3+a}}+\frac{2\pi \e^{\gamma-2}}{1+\delta^{2+a}}\\
&+ \int_{\R^2\backslash B_{0}(\delta)} \frac1{|z|^2}\, \frac{\e^{\gamma-2}}{1+|z-Y|^{2+a}}\, dz\\
\sim &\e^{\gamma-2}\int_{\R^2} \frac{1}{\big(\frac{1}{2}+|z|\big)^2\big(|z-Y|+\frac{1}{2}\big)^{2+a}}\, dz.
}
Thanks to Lemma \ref{lem-KLLT-integral-est-1}, one finds
\EQ{\label{sect5-tdQ1-est}
|\td Q_1(Y)| &\lec\e^{\gamma-2}\bkt{\frac1{(|Y|+1)^{2+a}}\bke{1+\log\bke{1+|Y|}} + \frac1{(|Y|+1)^2}}\\
&\lec \e^{\gamma-2} \bke{\frac{|Y|^{\delta}}{(|Y|+1)^{2+a}}\, + \frac1{(|Y|+1)^2}},\quad \de>0,\\
&\lec \frac{\e^{\gamma-2}}{\bke{1+|Y|}^{2+a-\de}} + \frac{\e^{\gamma-2}}{\bke{1+|Y|}^2}.
}
Then, we simply take $\de=a$ in \eqref{sect5-tdQ1-est} to obtain
\[
|\td Q_1(Y)| \lec \frac{\e^{\gamma-2}}{\bke{1+|Y|}^2},
\]
which completes the proof of \eqref{sect5-linear-theory-decay-estimate-P1}.
\end{proof}
% efore establishing the pointwise decay estimate of $ v$ in \eqref{eq-stokes-divF}, 
%\begin{align}\label{sect5-inner-problem-velocity-reflection}
%\left\{\begin{array}{ll}
%\nabla P_1=\Delta {\bf u}^{\text{i}}+\nabla\cdot \mathbb F,&Y \in \mathbb R^2,\\
%\nabla\cdot {\bf u}^{\text{i}}=0,&Y \in \mathbb R^2,\\
%\partial_{Y_2} {\bf u}^{\text{i}}_1= {\bf u}^{\text{i}}_2=0, &Y \in\partial\mathbb R^2.
%\end{array}
%\right.
%\end{align}
%In summary, we have the following lemma:
%\begin{lemma}\label{lemmainner-interior}
%For $\Vert F\Vert_{S,\nu-2,a+2}<+\infty,$ the solution $(v,q)$ of system \eqref{eq-stokes-divF} satisfies
%\begin{align*}
%|v(x)|\lesssim \Vert F\Vert_{S,\nu-2,a+2}\frac{\varepsilon^{\nu-1}}{1+\vert\frac{x-\xi}{\varepsilon}\vert},
%\end{align*}
%and
%\begin{align*}
%    \vert q(x)\vert \lesssim \Vert F\Vert_{S,\nu-2,a+1} \frac{\ep^{\nu-2}}{1+\vert\frac{x-\xi}{\varepsilon}\vert^2}.
%\end{align*}
Lemma \ref{sect5-lemmainner-interior} establishes the desired inner estimate of the velocity field $\bf u$.  Next, we consider the outer problem (\ref{sect5-outer-problem-velocity}) and hope to find the solution via the $W^{2,p}$ estimate of Stokes equation.  %The key step is to formulate the $W^{2,p}$ theory by utilizing the Green's representation formula and Poincar\'e-Morrey inequality.  Indeed, we abstract (\ref{sect5-outer-problem-velocity}) as the following problem:
%\begin{align}\label{sect5-eq-stokes-general}
%\left\{\begin{array}{ll}
%-\De{\bf u} + \nb\pi= {\bf f},&x\in \Om,\\
%\div{\bf u}=\eta,&x\in\Om,\\
%{\bf u}\cdot\boldsymbol{\nu} = g,&x\in\pd\Om,\\
%2[\mathbb S({\bf u})\cdot\boldsymbol{\nu}]_{\boldsymbol{\tau}}= %h\boldsymbol{\tau},&x\in\pd\Om,
%\end{array}
%\right.
%\end{align}
%then establish the $W^{2,p}$ estimate of the solution to (\ref{sect5-eq-stokes-general}), which is stated in Theorem \ref{thm-4.1-AR}.  In particular, the detailed discussion is postponed to Section \ref{sect6-W2p-estimate-Stokes}.
To this end, we must estimate $\nabla_x {\bf{u}}^{\text{i}}$ and readiy have the following lemma.
\begin{lemma}\label{sect5-lemma53-inner-estimate-for-outer}
Under the assumptions of Lemma \ref{sect5-lemmainner-interior}, the following estimates hold:
\begin{align}\label{sect5-linear-theory-decay-estimate-gradu}
|\nabla_x{{\bf u}^{\text{i}}}(Y)|\lesssim \Vert \mathbb F\Vert_{S,\gamma-2,a+2}\frac{\varepsilon^{\gamma-3}}{(1+\vert Y\vert)^2},
\end{align}
and
\begin{align}\label{sect5-linear-theory-decay-estimate-P1-W2p-estimate}
\Vert\nabla_x P_1\Vert_{L^p(B_{2\delta}(\xi)\backslash B_{\delta}(\xi))}\lesssim \Vert \nabla_x\cdot\mathbb F\Vert_{L^p(B_{2\delta}(\xi)\backslash B_{\delta}(\xi))}.
\end{align}
\end{lemma}
\begin{proof}
Noting the representation formula of ${\bf u}^{\text{i}}$ is \eqref{sect5-eq-solution-formula-stokes-divF}$_1$, we have
\begin{align}\label{sect5-pd-tdu-estimate-step1}
\partial_l{\bf{\td u}^{\text{i}}}_m(Y) 
=&\sum_{j,k=1}^2 \int_{\R^2} \pd_{z_l}U_{mj}(Y-z) \pd_{z_k}({\bf E}{\mathbb F})_{jk}(z)\, dz.
\end{align}
Since $\max\limits_{i,j=1,2}|\nabla^2 U_{ij}(x)|\lec  |x|^{-2}$ and $\mathbb F$ satisfies 
$$\Vert \mathbb F\Vert_{H,\infty}(z) \lec \frac{\e^{\gamma-2}}{1+|z|^{2+a}},$$
one applies the integration by parts on \eqref{sect5-pd-tdu-estimate-step1} to get
\begin{align*}
|\pd_{l} {\bf {\tilde u}^{\text{i}}}_m(Y)|\lec &\int_{B_0(\delta_1)}\frac{1}{|z|}\,  \frac{\e^{\gamma-2}}{1+|z-Y|^{3+a}}\, dz+\int_{\R^2\backslash B_0(\delta_1)}\frac{1}{|z|^2}\,  \frac{\e^{\gamma-2}}{1+|z-Y|^{2+a}}\, dz\\
\sim & \e^{\gamma-2}\int_{\mathbb R^2}\frac{1}{\big(|z|+\frac{1}{2}\big)^2\big(|z-Y|+\frac{1}{2}\big)^{2+a}}\, dz.
\end{align*}
By using Lemma \ref{lem-KLLT-integral-est-1}, we obtain 
$$\int_{\mathbb R^2}\frac{1}{\big(|z|+\frac{1}{2}\big)^2\big(|z-Y|+\frac{1}{2}\big)^{2+a}}\, dz\lec \frac{1}{(|Y|+1)^{2+a}}+\frac{1}{(|Y|+1)^2}.$$
Moreover, recall that $y=\frac{x-\xi}{\e}$, one completes the proof of \eqref{sect5-linear-theory-decay-estimate-gradu}.  

For pressure $\tilde P_1$, invoking the $W^{2,p}$ theory shown in Theorem \ref{thm-4.1-AR}, we readily derive the estimate \eqref{sect5-linear-theory-decay-estimate-P1-W2p-estimate}.  
\end{proof}
With the help of Lemma \ref{sect5-lemma53-inner-estimate-for-outer}, we now focus on the outer part and formulate the outer linear theory satisfied by ${\bf u}^\text{o}$.  To begin with, noting the existence of non-trivial kernel discussed in (\ref{sect5-solvability-conditions-eq}) and (\ref{sect5-solvability-condition-divergence}), we must impose the orthogonality conditions on system \eqref{sect5-outer-problem-velocity} and modify the problem as
\begin{align}\label{sect5-outer-problem-velocity-plus-beta-1}
\left\{\begin{array}{ll}
\nabla_x (\tilde P-\chi P_1)=\Delta_x{\bf u}^{\text{o}}+\boldsymbol{{\tilde {f}}}+d_1 \boldsymbol{\be},&x\in\Omega,\\
\nabla_x\cdot {\bf u}^{\text{o}}=-\e\nabla_x \chi\cdot {\bf u}^{\text{i}},&x\in\Omega,\\
{\bf u}^{\text{o}}\cdot \boldsymbol{\nu}=0,&x\in\partial\Omega,\\
(\mathbb S_x({\bf u}^{\text{o}})\cdot\boldsymbol{\nu})_{\boldsymbol{\tau}}=\tilde h\boldsymbol{\tau},&x\in\partial\Omega,\\
\int_{\Omega}{\bf u}^\text{o}\cdot{\boldsymbol{\beta}}\, dx=0,~~{\bf u}^\text{o}\in {\boldsymbol{H}}^2(\Omega),
\end{array}
\right.
\end{align}
where $\boldsymbol{\be} = c{\bf x}^\perp + {\bf b}$ with constant $c\neq0$ and constant vector ${\bf b}$, and
\EQ{\label{sect5-tildef-tildeh-def}
\boldsymbol{{\tilde {f}}}:=&(1-\chi)\nabla_x\cdot \mathbb F+2\e\nabla_x\chi\cdot\nabla_x {\bf u}^{\text{i}}+\e(\Delta_x \chi){\bf u}^{\text{i}}- P_1\nabla_x \chi,\\
\tilde h\boldsymbol{\tau}:=&-(\mathbb S_x(\e{\bf u}^{\text{i}}\chi)\cdot\boldsymbol{\nu})_{\boldsymbol{\tau}}.
}
Here $d_1$ is determined to satisfy
$$\int_{\Omega} \boldsymbol{\tilde f }\cdot\boldsymbol{\beta} \, dx+d_1\int_{\Omega} \boldsymbol{\beta}\cdot \boldsymbol{\beta}\, dx+2\int_{\partial\Omega} \tilde h \boldsymbol{\tau}\cdot\boldsymbol{ \beta} \,dS=0.
$$
  % Before applying the $W^{2,p}$ theory on the outer problem (\ref{sect5-outer-problem-velocity}), 
%We next discuss the outer linear theory of Stokes operator.  Focusing on the following outer problem:
%\begin{align}\label{stokesouter}
%\left\{\begin{array}{ll}
%-\Delta v+\nabla q=f_{\text{out}}+d_1\beta,&x\in\Omega,\\
%\nabla\cdot v=h,&x\in\Omega,\\
%({\mathbb S}v\cdot \boldsymbol{\nu})_{\tau}=g,~~v\cdot \boldsymbol{\nu}=0,&x\in\partial\Omega.
%\end{array}
%\right.
%\end{align}
%where $d_1$ is a constant, we have the following lemma:
Then, considering system (\ref{sect5-outer-problem-velocity-plus-beta-1}), we have the following results:
\begin{lemma}\label{sect5-lemma-outer-velocity-prop}
%We choose $d_1$ as
%\begin{align*}
%d_1\langle \beta,\beta\rangle_{\Omega}+\langle f_{\text{out}},\beta\rangle_{\Omega}+\langle %g,\beta\rangle_{\partial\Omega}=0,
%\end{align*}
%where $\beta=b_1(-x_2,x_1)+b_2$ with $b_1$ and $b_2$ being constants,
Assume that $\Vert\mathbb F\Vert_{S,\nu-2,a+1}<\infty,$ then we have system (\ref{sect5-outer-problem-velocity-plus-beta-1}) admits the solution $({\bf u}^{\text{o}},\tilde P)$ satisfying
\begin{align}\label{sect5-outer-sol-eq-conlusion-1}
\Vert {\bf u}^{\text{o}}\Vert_{W^{2,p}(\Omega)}+\Vert \td P-\chi P_1\Vert_{W^{1,p}(\Omega)}\lesssim \Vert \mathbb F\Vert_{S,\gamma-2,2+a},\ \ p>2.
\end{align}
Moreover, we have the following H{\"o}lder estimate holds:
\begin{align}\label{sect5-outer-sol-eq-conlusion-2}
\Vert {\bf u}^{\text{o}}\Vert_{C^{\alpha}(\Omega)}\lesssim \Vert \mathbb F\Vert_{S,\gamma-2,2+a},
\end{align}
where $\alpha\in(0,1).$
%\Vert \eta\Vert_{W^{2-\frac{1}{p},p}(\partial\Omega)}
\end{lemma}
\begin{proof}
We shall apply the $W^{2,p}$ theory on system (\ref{sect5-outer-problem-velocity-plus-beta-1}).  Before this, we must verify the compatibility conditions.  It is straightforward to see the first condition given by (\ref{eq-3.15-AR}) is satisfied since we choose $d_1$ such that
%\EQ{\label{sect5-tildef-tildeh-def}
%{\boldsymbol{{\tilde {f}}}:=(1-\chi)\nabla_x\cdot \mathbb F+2\e\nabla_x\chi\cdot\nabla_x {\bf u}^{\text{i}}+\e(\Delta_x \chi){\bf u}^{\text{i}}- P_1\nabla_x \chi,\\
%\tilde h\boldsymbol{\tau}:=-(\mathbb S_x(\e{\bf u}^{\text{i}}\chi)\cdot\boldsymbol{\nu})_{\boldsymbol{\tau}}.}}
%Similarly as the discussion for the existence of the solution to (\ref{sect5-bfu-decompose-before-formulate}), we must verify the compatibility conditions so as to apply the $W^{2,p}$ theory on \eqref{sect5-outer-problem-velocity}.  
\begin{align}\label{sect5-d1-chose-reason}
\int_{\Omega} \boldsymbol{\tilde f }\cdot\boldsymbol{\beta} \, dx+d_1\int_{\Omega} \boldsymbol{\beta}\cdot \boldsymbol{\beta}\, dx+2\int_{\partial\Omega} \tilde h \boldsymbol{\tau}\cdot\boldsymbol{ \beta} \,dS=0.
\end{align}
Next, we claim 
$$\int_{\Omega} \nabla_x\chi\cdot {\bf u}^{\text{i}}\, dx=0, $$
which establishes the second orthogonality condition (\ref{eq-2.14-AR}).  Indeed, we have the fact that ${\bf u}^{\text{i}}$ is divergence-free.  Then in light of ${\bf u}={\bf u}^{\text{i}}\chi+{\bf u}^{\text{o}}$ and the boundary conditions of ${\bf u}^{\text{i}}$, one completes the proof of our claim.

Now, we can use the $W^{2,p}$ estimate (\ref{eq-est-thm4.1}) to find
\EQ{\label{sect5-outer-sol-eqq-before}
\norm{{\bf u}^{\text{o}}}_{{\bf W}^{2,p}(\Om)} + \norm{\tilde P_1-\chi P}_{W^{1,p}(\Om)}
\lec&  \norm{(1-\chi)\nabla_x\cdot \mathbb F+2\e\nabla_x\chi\cdot\nabla_x {\bf u}^{\text{i}}+\e(\Delta_x \chi){\bf u}^{\text{i}}- P_1\nabla_x \chi}_{{\bf L}^p(\Om)} \\
&+  \norm{\e\nabla_x \chi\cdot {\bf u}^{\text{i}}}_{W^{1,p}(\Om)} + \norm{(\mathbb S_x(\e{\bf u}^{\text{i}}\chi)\cdot\boldsymbol{\nu})_{\boldsymbol{\tau}}}_{W^{1-\frac1p,p}(\pd\Om)} \\
&+\vert d_1\vert\Vert \boldsymbol{\beta}\Vert_{L^p(\Om)}.
}
For the formulation of the outer estimate, it is necessary to study $d_1$.  Thanks to the cut-off function $\chi$, we have from (\ref{sect5-d1-chose-reason}) that
\begin{align*}
d_1=&-\Big({\int_{\Omega}\boldsymbol{\tilde f}\cdot \boldsymbol{\beta}\, dx+2\int_{\partial\Omega}\tilde h\boldsymbol{\tau}\cdot\boldsymbol{\beta}\, dS}\Big)\Big/\int_{\Omega}\boldsymbol{\beta}\cdot \boldsymbol{\beta}\, dx\\
=&-\Big({\int_{\Omega\backslash B_{\delta}(\xi)}\boldsymbol{\tilde f}\cdot \boldsymbol{\beta} dx+2\int_{\partial\Omega\backslash B_{\delta}(\xi)}\tilde h\boldsymbol{\tau}\cdot\boldsymbol{\beta}dS}\Big)\Big/\int_{\Omega}\boldsymbol{\beta}\cdot \boldsymbol{\beta}\, dx.
\end{align*}
Then by H\"{o}lder's inequality and Trace theorem, we estimate $d_1$ to get
\begin{align}\label{sect5-d1estimate-before}
|d_1|\lec \Vert \boldsymbol{\tilde f}\Vert_{L^p(\Omega)}+\Vert \tilde h\Vert_{W^{1-\frac{1}{p},p}(\partial\Omega)},
\end{align}
where $\boldsymbol{{\tilde {f}}}$ and $\tilde h$ are given by \eqref{sect5-tildef-tildeh-def}. 
 With the help of (\ref{sect5-d1estimate-before}), we utilize the boundedness of $\Vert\boldsymbol{\beta}\Vert_{L^p(\Omega)}$ and check (\ref{sect5-outer-sol-eqq-before}) to obtain
\EQ{\label{sect5-outer-sol-eqq-then}
\norm{{\bf u}^{\text{o}}}_{{\bf W}^{2,p}(\Om)} + \norm{\tilde P_1-\chi P}_{W^{1,p}(\Om)}
\lec&  \norm{(1-\chi)\nabla_x\cdot \mathbb F+2\e\nabla_x\chi\cdot\nabla_x {\bf u}^{\text{i}}+\e(\Delta_x \chi){\bf u}^{\text{i}}- P_1\nabla_x \chi}_{{\bf L}^p(\Om)} \\
&+  \norm{\e\nabla_x \chi\cdot {\bf u}^{\text{i}}}_{W^{1,p}(\Om)} + \norm{(\mathbb S_x(\e{\bf u}^{\text{i}}\chi)\cdot\boldsymbol{\nu})_{\boldsymbol{\tau}}}_{W^{1-\frac1p,p}(\pd\Om)}.
}
Noting the definition of the cut-off function $\chi$, we further have
\begin{align}\label{sect5-outer-sol-eq-1}
|(1-\chi)\nabla_x\cdot \mathbb F|\lec \Vert \mathbb F\Vert_{S,\gamma-2,2+a}\e^{a+\gamma},
\end{align}
\begin{align}\label{sect5-outer-sol-eq-2}
|2\e\nabla_x\chi\cdot\nabla_x {\bf u}^{\text{i}}+\e(\Delta_x \chi){\bf u}^{\text{i}}- P_1\nabla_x \chi| \lec \Vert \mathbb F\Vert_{S,\gamma-2,2+a}\e^{\gamma},
\end{align}
\begin{align}\label{sect5-outer-sol-eq-3}
|\e\nabla_x \chi\cdot {\bf u}^{\text{i}}| \lec \Vert \mathbb F\Vert_{S,\gamma-2,2+a}\e^{\gamma}
\end{align}
and
\begin{align}\label{sect5-outer-sol-eq-4}
|\e\nabla_x \chi\cdot \nabla_x{\bf u}^{\text{i}}| \lec \Vert \mathbb F\Vert_{S,\gamma-2,2+a}\e^{\gamma}.
\end{align}
In particular, since $\partial_{Y_2} { u}^{\text{i}}_1= { u}^{\text{i}}_2=0$ on $\pd \R_{+}^2,$ we get
\begin{align}\label{sect5-outer-sol-eq-5}
\vert(\mathbb S_x(\e{\bf u}^{\text{i}}\chi)\cdot\boldsymbol{\nu})_{\boldsymbol{\tau}}\vert \lec |\e \nabla_x \chi\cdot{\bf u}^{\text{i}}|+|\e^2 {\bf u}^{\text{i}}|\lec \e^{\gamma} \Vert \mathbb F\Vert_{S,\gamma-2,2+a}
\end{align}
and
\begin{align}\label{sect5-outer-sol-eq-6}
\big\vert\nabla_x[ (\mathbb S_x(\e{\bf u}^{\text{i}}\chi)\cdot\boldsymbol{\nu})_{\boldsymbol{\tau}}]\big\vert \lec& |\e \nabla_x \chi\cdot\nabla_x{\bf u}^{\text{i}}|+|\e \Delta_x \chi{\bf u}^{\text{i}}|+|\e^2 \nabla_x{\bf u}^{\text{i}}|\nonumber\\
\lec &\e^{\gamma} \Vert \mathbb F\Vert_{S,\gamma-2,2+a}.
\end{align}

Upon collecting (\ref{sect5-outer-sol-eq-1}), (\ref{sect5-outer-sol-eq-2}), (\ref{sect5-outer-sol-eq-3}), (\ref{sect5-outer-sol-eq-4}), (\ref{sect5-outer-sol-eq-5}) and (\ref{sect5-outer-sol-eq-6}), one finds from (\ref{sect5-outer-sol-eqq-then}) that \eqref{sect5-outer-sol-eq-conlusion-1} holds. 
 Moreover, noting that $p$ is assumed to satisfy $p>2$, we use Sobolev embedding theorem to readily obtain \eqref{sect5-outer-sol-eq-conlusion-2}. 
\end{proof}
For the velocity field ${\bf u}$, by using Lemma \ref{sect5-lemmainner-interior} and Lemma \ref{sect5-lemma-outer-velocity-prop}, we have the following proposition:
\begin{proposition}\label{sect5-inner-outer-combine-prop}
Assume that $\Vert \mathbb F\Vert_{S,\gamma-1,2+a}<\infty$ with $0<\gamma, a<1.$  Then there exists a solution $({\bf u},\tilde P)$ to system \eqref{sect5-bfu-decompose-before-formulate} satisfying
\begin{itemize}
    \item for $ x\in B_{\delta}(\xi):=\{x||x-\xi|<\delta\}$ with $\delta>0$ being some sufficiently small number,
    \begin{align}\label{sect5-prop55-conclu1}
    |{\bf u}(x)|\lesssim\Vert \mathbb F\Vert_{S,\gamma-2,a+2}\frac{\e^{\gamma-1}}{1+\big\vert\frac{x-\xi}{\e}\big\vert}
    \end{align}
    and
    \begin{align*}
    |{\tilde P}(x)|\lec\Vert \mathbb F\Vert_{S,\gamma-2,a+2}\frac{\e^{\gamma-2}}{\bke{1+|Y|}^2};
    \end{align*}
    \item for $x\in\Omega \backslash B_{\delta}(\xi)$,
    \begin{align*}
\Vert {\bf u}\Vert_{W^{2,p}(\Omega\backslash B_{\delta}(\xi))}+\Vert \nabla\td P\Vert_{L^{p}(\Omega\backslash B_{\delta}(\xi))}\lesssim \Vert \mathbb F\Vert_{S,\gamma-2,2+a},\ \ p>2;
\end{align*} 
moreover,
\begin{align*}
\Vert {\bf u}\Vert_{C^{\alpha}(\Omega\backslash B_{\delta}(\xi))}\lesssim \Vert \mathbb F\Vert_{S,\gamma-2,2+a}.
\end{align*}
\end{itemize}
\end{proposition}
Focusing on the results stated in Proposition \ref{sect5-inner-outer-combine-prop}, we give some remarks.
\begin{remark}\label{sect5-remark-5.6}
~
\begin{itemize}
    \item For the outer problem (\ref{sect5-outer-problem-velocity}), we find the only difference between (\ref{sect5-outer-problem-velocity}) and (\ref{sect5-outer-problem-velocity-plus-beta-1}) is the Lagrange multiplier term $d_1\boldsymbol{\beta}.$ In fact, by solving the reduced problem in Section \ref{inn-out-gluing-sect}, one can readily see that $d_1=0$.
    \item \eqref{sect5-prop55-conclu1} in Proposition \ref{sect5-inner-outer-combine-prop} implies that 
    $$\Vert {\bf u}\Vert_{S,\gamma-1,1}\lec \Vert \mathbb F\Vert_{S,\gamma-2,2+a}.$$
    \item Noting that ${\bf u}$ is divergence-free, we can write ${\bf u}\cdot \nabla {\bf u}=\nabla\cdot ({\bf u}\otimes {\bf u})$, where $\otimes$ represents the tensor product
defined by $(v \otimes w)_{ij} = v_iw_j$.  We shall solve ${\bf u}$ in the class $\Vert {\bf u}\Vert_{S,\gamma-1,1}<\infty$.  Then the advection term ${\bf u}\cdot \nabla_x {\bf u}$ satisfies
$$|{\bf u}\cdot \nabla_x {\bf u}|\lec\frac{\e^{2\gamma-3}}{1+\big|\frac{x-\xi}{\e}\big|^3},$$
which is a perturbation compared to $\nabla\cdot \mathbb F.$  In this case, we are able to solve $\bf u$ by the fixed point argument and the detailed discussion will be shown in Section \ref{inn-out-gluing-sect}.
\end{itemize}
\end{remark}
With the help of Proposition \ref{sect5-inner-outer-combine-prop}, we are able to study the concentration phenomenon in Section \ref{inn-out-gluing-sect}.
%next establish the $W^{2,p}$ theory of Stokes equation in Section \ref{sect6-W2p-estimate-Stokes}.
%With the help of Proposition \ref{sect5-inner-outer-combine-prop}, we are able to solve the Keller-Segel-Navier-Stokes system in 2D, which will be shown in Section \ref{inn-out-gluing-sect}.  
%then we impose the following orthogonality conditions:
%\begin{align*}
%\int_{\Omega}\eta dx=\int_{\partial\Omega} hdS,
%\end{align*}
%and $d_1$ is chosen to satisfy
%\begin{align*}
%d_1\langle \beta,\beta\rangle_{\Omega}-\int_{\Omega}F:\nabla\beta dx+\langle g,\beta\rangle_{\partial\Omega}=0.
%\end{align*}
%Now, we have the following lemma:
%\begin{lemma}
%For $\Vert F\Vert_{S,\nu-2,a+1}<+\infty,$ we have the solution $(v_{\text{out}},p)$ of system satisfies
%\begin{align*}
%\Vert v_{\text{out}}\Vert_{W^{2,p}(\Omega)}+\Vert p\Vert_{W^{1,p}(\Omega)}\lesssim \Vert \eta\Vert_{W^{2-\frac{1}{p},p}(\partial\Omega)}+\Vert F\Vert_{W^{1,p}(\Omega)}+\Vert g\Vert_{W^{1-\frac{1}{p},p}(\partial\Omega)}.
%\end{align*}
%\end{lemma}

\section{Inner--outer gluing system: existence of solution}\label{inn-out-gluing-sect}
This section is devoted to the construction of the boundary spot in stationary problem (\ref{PKSNS-ss-equiv}) via the \textit{inner-outer} gluing method.  Before performing the inner-outer gluing procedure, we collect some notations and definitions.  Recall that the inner operator $L_{W}[\Phi]$ is given by
$$L_{W}[\Phi]=\Delta_y\Phi-\nabla_y\cdot(W\nabla_y \Psi)-\nabla_y\cdot(\Phi\nabla_y\Gamma),~~-\Delta_y^{-1} \Phi=\Psi;$$
the outer operator is defined as
\begin{align*}
 \bar L^o[\varphi]=\Delta_y \varphi-\nabla_y \varphi\cdot \nabla_y \bar V-\varepsilon^2\bar V \varphi, \ \  \bar V:=\Gamma+ H^{\varepsilon},
\end{align*}
where the inner and outer norms $\|\cdot\|_{\delta_1,\nu_1}$ and $\|\cdot\|_{\nu,o}$ are 
 \begin{equation*}
      \|h\|_{\delta_1,\nu_1} := \e^{-\delta_1}\sup_{y\in \R^2}|h(y)|(1 + |y|)^{\nu_1}~\text{and}~\|h\|_{\nu, o}: = \sup_{y \in \Omega_{\e}} \e^{-\delta_3}{\vert h(y)\vert}{(1 + |y - \xi'|)^{\nu_3}},
 \end{equation*}
 with $y=\frac{x}{\e}$ and $\xi'=\frac{\xi}{\e}$.  In addition, $\phi$ is decomposed as
\begin{align*}
\phi(x)=\frac{1}{\e^2}\Phi^{\text{i}}(y)\chi(|y-\xi'|)+\varphi^{\text{o}},
\end{align*}
where $\chi$ is
\begin{align*}
\chi(r):=\left\{\begin{array}{ll}
1,&r\leq \frac{\delta}{\e},\\
0,&r\geq \frac{2\delta}{\e},
\end{array}
\right.
\end{align*}
and $\delta>0$ is fixed but small constant.  However, the center of the boundary spot is located at the boundary of $\Omega.$  To tackle this, as shown in Section \ref{sect3}, we define the graph $\rho(x_1)$ as $\{(x_1,x_2)=(x_1,\rho(x_1)\}$ with $\rho(0)=\rho'(0)=0$ and transform $(x_1,x_2)$ and $(y_1,y_2)$ as
\EQN{
X_1=&x_1-\xi_{1},~~X_{2}=x_2-\xi_{2}-\rho(x-\xi_1),\\
Y_1=&y_1-\xi'_{1},~~Y_{2}=y_2-\xi'_{2}-\frac{1}{\e}\rho(\e (Y_1-\xi'_{1})),
}
where $y_1=x_1/\e$ and $y_2=x_2/\e.$  Moreover, for any function $w,$
\EQ{\label{sect6-new-neumann-boundary-X}
\Delta_xw=\Delta_{X}w+(\rho'(X_1))^2\partial_{X_2X_2}w-2\rho'(X_1)\partial_{X_2X_1}w-\rho''(X_1)\partial_{X_2}w,\\
 (1+(\rho'(X_1))^2)\frac{\partial w}{\partial\boldsymbol{\nu}}=(\rho'(X_1))\partial_{X_1}w-\partial_{X_2}w-(\rho'(X_1))^2\partial_{X_2}w,
}
and
\EQ{\label{sect6-new-neumann-boundary-Y}
 \Delta_y w =& \Delta_{Y} w + (\rho'(\e Y_1))^2 \partial_{Y_2 Y_2} w - 2\rho'(\e Y_1) \partial_{Y_1Y_2} w -\e \rho''(\e Y_1) \partial_{Y_2}w,\\
 \sqrt{1+(\rho'(\e Y_1))^2}\frac{\partial w}{\partial\boldsymbol{\nu}_{\e}}=&\rho'(\e Y_1)\partial_{Y_1}w-[1+(\rho'(\e Y_1))^2]\partial_{Y_2}w.
}
  For the sake of convenience, we denote the flatten operator $\bar P_{\rho, \xi'}$ such that for any function $w,$
  \begin{equation*}
      {\bar P}_{\rho,\xi}w(x_1,x_2)=w(X_1,X_2),\ \ \bar P_{\rho, \xi'}w(y_1, y_2) = w(Y_1, Y_2).
  \end{equation*}
In addition, we are able to compute $\nabla_{x}({\bar P}_{\rho,\xi}w)$ and $\Delta_{y}(\bar P_{\rho, \xi'}w)$ in the $X$ and $Y$ variable by using (\ref{sect6-new-neumann-boundary-X}) and (\ref{sect6-new-neumann-boundary-Y}), respectively.  With the flatten operator ${\bar P}_{\rho,\xi}$, we further define the inner norm in the half space $\R^2_{+}$ as      
\begin{equation*}
     \| h\|_{\delta_2,\nu_2,H} = \varepsilon^{-\delta_2}\sup_{y \in \R^2_{+}} |h|(1 + |y|)^{\nu_2},~~\delta_2,\nu_2>0.
  \end{equation*}

We can get the desired solution $(\phi,\psi,{\bf u})$ of (\ref{PKSNS-ss-equiv}) if $(\Phi^{\text{i}}_H,\varphi^{\text{o}},{\bf u},\xi)$ solves the following inner-outer gluing system:
\begin{align}\label{sect7-stokes-problem-solve}
\left\{\begin{array}{ll}
{\bf u}\cdot\nb {\bf u} + \nb P = \De {\bf u} - \e_0\nabla\cdot \mathcal F\big( \Phi^{\text{i}}_H,\varphi^{\text{o}},{\bf u}, \xi\big)\ \ &\text{in} \ \ \Omega,\\
\nabla\cdot {\bf u}=0\ \ &\text{in} \ \ \Omega,\\
{\bf u}\cdot \boldsymbol{\nu}=0,\quad ({\mathbb S}{\bf u}\cdot \boldsymbol{\nu})_{\boldsymbol{\tau}}=0\ \ &\text{on} \ \ \partial\Omega,
\end{array}
\right.
\end{align}
and
\begin{align}\label{sect6-inn-outer-gluing-firsttwo}
\left\{\begin{array}{ll}
 L_{W}\big[ \Phi_H^{\text{i}}\big] =\mathcal H\big( \Phi_H^{\text{i}},\varphi^{\text{o}},{\bf u}, \xi_{H}\big)  \ \ &\text{in} \ \ \R^2_{+}, \\
 \frac{\partial \Phi^{\text{i}}_H}{\partial Y_2}-W\frac{\partial \bar\Psi^{\text{i}}_H}{\partial Y_2}=\bar\beta\big(\Phi_H^{\text{i}},\varphi^{\text{o}},{\bf u}, \xi_{H}\big)  \ \ &\text{on} \ \ \pd \R^2_{+}, \\
 \e^2{\bar L}^{\text{o}}[\varphi^o]= \mathcal G\big( \Phi_H^{\text{i}},\varphi^{\text{o}},{\bf u}, {\xi}\big) \ \ &\text{in} \ \  \Omega_{\e},
\end{array}
\right.
\end{align}
where $\Phi_H^{\text{i}}=\bar P_{\rho, \xi'}\Phi^{\text{i}}$, ${\bar\Psi}_H^{\text{i}}=-\Delta^{-1}_{Y}\Phi_H^{\text{i}}$, $\xi_{H}=(\xi_{H,1},0)^T$,
$$ \mathcal F\big( \bar P_{\rho, \xi'}\Phi^{\text{i}},\varphi^{\text{o}},{\bf u}, \xi\big)=\nabla c\otimes \nabla c-\bigg( \frac{\vert \nabla c\vert^2}{2}\bigg)-\Big(\frac{c^2}{2}\Big),$$
\begin{align*}
\mathcal H\big( \bar P_{\rho, \xi'}\Phi^{\text{i}},\varphi^{\text{o}},{\bf u}, \xi_H\big):=&\e\nabla_y\cdot\big(W\nabla_x H^{\varepsilon}\big)\chi+\e {\bf u}\cdot\nabla_y W\chi+{\varepsilon^2}\nabla_x\cdot(W\nabla_x\psi^{\text{o}})\chi\nonumber\\
&+{\varepsilon^2}\nabla_y\cdot\Big(\Big(\frac{1}{\varepsilon^2}\Phi^{\text{i}}\chi+\varphi^{\text{o}}\Big)\nabla_y\psi\Big)\chi+{\varepsilon}\nabla_y\cdot\big(\Phi^{\text{i}}\nabla_x H^{\varepsilon}\big)\chi\nonumber\\
&+{\varepsilon}{\bf u}\cdot\nabla_y\Phi^{\text{i}}\chi-\nabla_y\cdot\big(W\nabla_y\bar\Psi^{\text{i}}\big)\chi+\nabla_y\cdot(W\nabla_y\Psi^{\text{i}})\chi\nonumber\\
&+[\nabla_y\cdot(W\nabla_y\hat\Psi^{\text{i}})-\nabla_y\cdot(W\nabla_y\Psi^{\text{i}})\chi]\chi,
\end{align*}
\begin{align*}
\bar\beta\big(\Phi_H^{\text{i}},\varphi^{\text{o}},{\bf u}, \xi_{H}\big):=&\Bigg[\rho''(0)\varepsilon Y_1\bigg(\frac{\partial \Phi^{\text{i}}_H}{\partial Y_1}-W\frac{\partial \bar\Psi^{\text{i}}_H}{\partial Y_1}\bigg)-[\rho''(0)\varepsilon Y_1]^2\bigg(\frac{\partial \Phi^{\text{i}}_H}{\partial Y_2}-W\frac{\partial \bar\Psi^{\text{i}}_H}{\partial Y_2}\bigg)\nonumber\\
&-\Phi^{\text{i}}_H\Big(\frac{\partial \Gamma}{\partial Y_1}\rho''(0)\varepsilon Y_1-\frac{\partial \Gamma}{\partial Y_2}-\frac{\partial\Gamma}{\partial Y_2}(\rho''(0)\varepsilon Y_1)^2\Big)\nonumber\\
&-\Phi^{\text{i}}_H\Big(\frac{\partial \bar\Psi^{\text{i}}_H}{\partial Y_1}\rho''(0)\varepsilon Y_1-\frac{\partial \bar\Psi^{\text{i}}_H}{\partial Y_2}-\frac{\partial\bar\Psi^{\text{i}}_H}{\partial Y_2}(\rho''(0)\varepsilon Y_1)^2\Big)\Bigg]\nonumber\\
&+\varepsilon(W+\Phi^{\text{i}}_H)\Big(\frac{\partial H}{\partial X_2}-\rho'(\varepsilon Y_1)\frac{\partial H}{\partial X_1}+[\rho'(\varepsilon Y_1)]^2\frac{\partial H}{\partial X_2}\Big)+\text{H. O. T. ,}
\end{align*}
and
\begin{align*}
\mathcal G\big(\bar P_{\rho, \xi'}\Phi^{\text{i}},\varphi^{\text{o}},{\bf u}, \xi\big):=&\e^2\nabla_x\cdot\big(W\nabla_x H^{\varepsilon}\big)(1-\chi)+\e^2 {\bf u}\cdot\nabla_x W(1-\chi)+{\varepsilon^2}\nabla_x\cdot(W\nabla_x\psi^{\text{o}})(1-\chi)\nonumber\\
&-2\nabla_y\Phi^{\text{i}}\cdot\nabla_y\chi-\Phi^{\text{i}}\Delta_y\chi+\Phi^{\text{i}}\nabla_y\Gamma\cdot\nabla_y\chi\nonumber\\
&+{\varepsilon^2}\nabla_y\cdot(\phi\nabla_y\psi)(1-\chi)+{\varepsilon}\Phi^{\text{i}}(\nabla_x H^{\varepsilon}\cdot\nabla_y\chi )\nonumber\\
&+\e ({\bf u}\cdot\nabla_y\chi) \Phi^{\text{i}}+\e^4 {\bf u}\cdot \nabla_x\varphi^{\text{o}}\nonumber\\
&+[\nabla_y\cdot(W\nabla_y\hat\Psi^{\text{i}})-\nabla_y\cdot(W\nabla_y\Psi^{\text{i}})\chi](1-\chi).
\end{align*}
In particular, as shown in Section \ref{sect3}, all terms involving the curvature are readily perturbations, which implies  
\begin{align*}
\mathcal H\big( \bar P_{\rho, \xi'}\Phi^{\text{i}},\varphi^{\text{o}},{\bf u}, \xi\big)=&\e\nabla_Y\cdot\big(W\nabla_X H^{\varepsilon}\big)\chi_{H}+\e {\bf u}\cdot\nabla_Y W\chi_{H}+{\varepsilon^2}\nabla_X\cdot(W\nabla_X\psi^{\text{o}})\chi_{H}\nonumber\\
&+{\varepsilon^2}\nabla_Y\cdot\Big(\Big(\frac{1}{\varepsilon^2}\Phi^{\text{i}}\chi_{H}+\varphi^{\text{o}}\Big)\nabla_Y\psi\Big)\chi_{H}+{\varepsilon}\nabla_Y\cdot\big(\Phi^{\text{i}}\nabla_X H^{\varepsilon}\big)\chi\nonumber\\
&+{\varepsilon}{\bf u}\cdot\nabla_Y\Phi^{\text{i}}\chi_{H}-\nabla_Y\cdot\big(W\nabla_Y\bar\Psi^{\text{i}}\big)\chi_H+\nabla_Y\cdot(W\nabla_Y\Psi^{\text{i}})\chi_{H}\nonumber\\
&+[\nabla_Y\cdot(W\nabla_Y\hat\Psi^{\text{i}})-\nabla_Y\cdot(W\nabla_Y\Psi^{\text{i}})\chi_{H}]\chi_H+\text{H.O.T.},
\end{align*}
 where the cut-off function $\chi_{H}$ is defined by
      \begin{equation*}
            \chi_{H}(y) = 1 \ \  \text{for} \ \  y \in \bar\R^2_{+}\cap \bar B_{\delta/\e}(0) \ \ \ \  \text{and} \ \ \ \    \chi_{H}(y) = 0 \ \   \text{for} \ \  y \in \R^2_{+}\cap B^c_{2\delta/\e}(0),
            \end{equation*}
            with $\delta>0$ is a small constant.  For operator $L_{W}\big[ \bar P_{\rho, \xi'}\Phi^{\text{i}}\big]$, we rewrite it in the $Y$-variable to get 
            \begin{align*}
            L_{W}[\bar P_{\rho, \xi'}\Phi^{\text{i}}]=L_{W,Y}[\Phi^{\text{i}}_H]+\frac{1}{\e^2}N_{\rho}[\Phi^{\text{i}}_H],
            \end{align*}
            where 
            $$L_{W,Y}[\Phi]=\Delta_Y\Phi-\nabla_Y\cdot(W\nabla_Y \Psi)-\nabla_Y\cdot(\Phi\nabla_Y\Gamma),~~-\Delta_Y^{-1} \Phi=\Psi,$$
            and $N_{\rho}[\Phi^{\text{i}}_{H}]$ is given by
    \begin{equation}\label{sect7-inner-new-error}
   \begin{split}
       N_{\rho}[\Phi^{\text{i}}_{H}]=&  (\rho'(\e Y_1))^2 \Big[\frac{\pd^2( \Phi_{H}\chi_{H,j})}{\pd Y_2^2} - \Big(\frac{\pd W}{\pd Y_2}\frac{\pd \bar{\Psi}_{H}}{\pd Y_2}+ \frac{\pd^2 \bar{\Psi}_{H}}{\pd Y_2^2}W  \\
       & + \frac{\pd (\phi_{H}\chi_{H})}{\pd Y_2} \frac{\pd \Gamma}{\pd Y_2}
       + \frac{\pd^2 \Gamma}{\pd Y_2^2}(\Phi_{H}\chi_{H})\Big)  \Big] \\
       &   - \rho'(\e Y_1)\Big[\frac{\pd^2(\phi_{H}\chi_{H})}{\pd Y_1 \pd Y_2} - \Big(\frac{\pd W}{\pd Y_1}\frac{\pd \bar{\Psi}_{H}}{\pd Y_2} + \frac{\pd W}{\pd Y_2}\frac{\pd \bar{\Psi}_{H}}{\pd Y_1}\Big) - \frac{\pd^2 \bar{\Psi}_{H}}{\pd Y_1 \pd Y_2}W  \\
         &  - \Big(\frac{\pd (\Phi_{H}\chi_{H})}{\pd Y_1} \frac{\pd \Gamma}{\pd Y_2} + \frac{\pd \Gamma}{\pd Y_2}\frac{\pd (\phi_{H}\chi_{H})}{\pd Y_1}\Big) - \frac{\pd^2 \Gamma}{\pd Y_1 \pd Y_2} (\Phi_{H}\chi_{H}) \Big]\\
         &   -\e\rho''(\e Y_1)\Big[\frac{ \pd (\Phi_{H}\chi_{H})}{\pd Y_2} - W\frac{\pd \bar{\Psi}_{H}}{\pd Y_2} - (\Phi_{H}\chi_{H}) \frac{\pd \Gamma}{\pd Y_2} \Big],
       \end{split}
   \end{equation}    
with $\bar\Psi_{H}:=-(\Delta+\e^2)^{-1}(\Phi_{H}\chi_{H}).$

To use Lemma \ref{sect4-linear-theory-bdry}, we have to impose the orthogonality condition on \eqref{sect6-inn-outer-gluing-firsttwo}.  To begin with, we define compactly supported radial functions $W_{0}(r)$ such that
$$\int_{\R^2_{+}} W_{0}(|Y|)\, dY=1,$$
and compactly supported radial functions $W_{1,1}$ such that
$$\int_{\R^2_{+}}W_{1,1}\big(|Y|\big)Y_1\, dY=1.$$
Next, we modify (\ref{sect6-inn-outer-gluing-firsttwo}) to formulate the following problem:
\begin{align}\label{sect6-inn-outer-gluing-firsttwo-modify}
\left\{\begin{array}{ll}
 L_{W,Y}\big[ \Phi^{\text{i}}_H\big] ={\tilde{\mathcal H}}\big(\Phi_H^{\text{i}},\varphi^{\text{o}},{\bf u}, \xi_{H}\big) -m_{0}[\tilde{\mathcal H}] W_{0}-m_{1}[\tilde{\mathcal  H}] W_{1,1}  \ \ &\text{in} \ \ \R^2_{+}, \\
  \frac{\partial \Phi^{\text{i}}_H}{\partial Y_2}-W\frac{\partial \bar\Psi^{\text{i}}_H}{\partial Y_2}=\bar\beta\big(\Phi_H^{\text{i}},\varphi^{\text{o}},{\bf u}, \xi_{H}\big)  \ \ &\text{on} \ \ \pd \R^2_{+}, \\
 \e^2{\bar L}^{\text{o}}[\varphi^o]= \mathcal G\big( \Phi^{\text{i}}_H,\varphi^{\text{o}},{\bf u}, {\xi}\big) \ \ &\text{in} \ \  \Omega_{\e},
\end{array}
\right.
\end{align}
where 
$$\tilde {\mathcal H}(\Phi_H^{\text{i}},\varphi^{\text{o}},{\bf u}, \xi_{H}):=\mathcal H(\Phi_H^{\text{i}},\varphi^{\text{o}},{\bf u}, \xi_{H})-N_{\rho}[\Phi^{\text{i}}_H];$$
$m_{0}[h]$ and $m_{1}[h]$ are given by
   \begin{equation}\label{m0eq0}
       m_{0}[h] = \int_{\R^2_+}h\, dY-\int_{\pd \R^{2}_+} \bar\beta \, dY  \ \ \text{and} \ \  m_{1}[h] = \int_{\R^2_+}h Y_1\,dY-\int_{\pd\R^2_+}{\bar \beta}Y_1\,dY.
   \end{equation}  
Given the velocity field $\bf u$, we are able to find the solution to (\ref{sect6-inn-outer-gluing-firsttwo-modify}) by invoking Lemma \ref{sect4-linear-theory-bdry} and Lemma \ref{sect4-outer-problem-linear-theory}.  Indeed, based the linear theories developed in Section \ref{sect4}, we shall solve the inner and outer problems arising from the transported Keller-Segel model in the norms below.
\begin{itemize}
    \item We use the norm $\Vert\cdot \Vert_{\delta_2,4+\sigma,H}$ to measure the right hand side $\mathcal {\tilde H}$ in (\ref{sect6-inn-outer-gluing-firsttwo-modify}), where
    \begin{align*}
    \big\Vert {\tilde h}\big\Vert_{\delta_2,4+\sigma,H}:= \varepsilon^{-\delta_2}\sup_{y \in \R^2_{+}} \big|{\tilde h}\big|(1 + |y|)^{4+\sigma}
    \end{align*}
    with $\sigma>0$ and $\delta_2\in(0,1).$
    \item We use the norm $\Vert\cdot \Vert_{\delta_2,2+\sigma,H}$ to measure the inner solution $\Phi^{\text{i}}_H$ in (\ref{sect6-inn-outer-gluing-firsttwo-modify}), where
        \begin{align}\label{sect7-inner-norm-def}
    \big\Vert \Phi^{\text{i}}_H\big\Vert_{\delta_2,2+\sigma,H}:= \varepsilon^{-\delta_2}\sup_{y \in \R^2_{+}} \big|\Phi^{\text{i}}_H\big|(1 + |y|)^{2+\sigma}.
    \end{align}
    \item We use the norm $\Vert \cdot \Vert_{\delta_3,2+b,o}$ to measure the right hand side $\mathcal G$ in (\ref{sect6-inn-outer-gluing-firsttwo-modify}), where
    \begin{align*}
     \Vert g \Vert_{\delta_3,2+b,o}:=\varepsilon^{-\delta_3}\sup_{y \in \Omega_{\varepsilon}}{ |g|}{ (1 + |y - \xi'|)^{2+b}}
     \end{align*}  
     with $\delta_3\in(0,1)$ and $b\in(2,3).$
     \item  We use the norm $\Vert \cdot \Vert_{\delta_3,b,o}$ to measure the solution ${\varphi}^{\text{o}}$ in (\ref{sect6-inn-outer-gluing-firsttwo-modify}), where
 \begin{equation*}
      \|\varphi^{\text{o}}\|_{\delta_3,b, o}: = \varepsilon^{2-\delta_3}\sup_{y \in \Omega_{\varepsilon}}{ |\varphi^{\text{o}}|}{ (1 + |y - \xi'|)^{b}}.
    \end{equation*}
\end{itemize}

Based on linear theory developed in Section \ref{sect-model-stokes}, we shall solve the incompressible Navier-Stokes equation with the free-slip boundary condition in the following norms.
\begin{itemize}
    \item We use the norm $\Vert \cdot \Vert_{S,\gamma-2,2+a}$ to measure the forcing term $\mathcal F$, where
     \begin{align*}
 \Vert  F\Vert_{S,\gamma-2,a+2}:={\varepsilon^{-(\gamma-2)}}\sup_{x\in \Omega}| F(x)|\Big(1+\Big|\frac{x-\xi}{\e}\Big|^{a+2}\Big)
 \end{align*}
 with $\gamma\in(0,1)$ and $a\in(0,1).$
    \item We use the norm $\Vert \cdot \Vert_{S,\gamma-1,1}$ to measure the solution $\bf u$, where
       \begin{align}\label{sect7-inner-norm-velocity}
 \Vert  {\bf u}\Vert_{S,\gamma-1,1}:={\varepsilon^{-(\gamma-1)}}\sup_{x\in \Omega}|{\bf u}(x)|\Big(1+\Big|\frac{x-\xi}{\e}\Big|\Big).
 \end{align}
\end{itemize}

We then define the spaces for the inner problem, outer problem and parameters as
  \begin{equation}\label{sect7-space-collect1}
  \begin{split}
E_{\text{i}}= \Big\{\Phi^{\text{i}}_H\in &L^{\infty}(\R^2_{+}):\nabla_y\Phi^{\text{i}}_H\in L^{\infty}(\R^2_{+}); \, \big\Vert\Phi^{\text{i}}_H \big\Vert_{\delta_2,2+\sigma,H}<\infty\Big\},
        \end{split}
  \end{equation}
   \begin{equation}\label{sect7-space-collect2}
 E_o = \Big\{\varphi^{\text{o}} \in L^{\infty}(\Omega_\e):\,  \nabla_y \varphi^{\text{o}} \in L^{\infty}(\Omega_\e); \, \, \|\varphi^{\text{o}}\|_{\delta_3,b, o}< \infty, \ \  \frac{\pd \varphi^{\text{o}}}{\pd {\boldsymbol{\nu}}_{\e}} =0\text{~on~}\partial\Omega_{\e}  \Big\},
 \end{equation}
 and
   \begin{equation}\label{sect7-space-collect3}
     E_p= \{\xi \in\R^2:    \|\xi\|_p = |\xi| < \infty \}.
    \end{equation}
Moreover, we define $E_{\phi}$ and solution ${\boldsymbol{\phi}}$ as 
\begin{align*}
E_{\phi}=E_{\text{i}}\times E_{\text{o}}, \ \ {\boldsymbol{\phi}}=(\Phi^{\text{i}}_H,\varphi^{\text{o}})^T
\end{align*}
with the norm $\Vert \cdot \Vert_{E_{\phi}}$ is given by
\begin{align*}
\Vert \boldsymbol{\phi} \Vert_{E_{\phi}}=  \big\|\Phi^{\text{i}}_H\big\|_{\delta_2,2 + \sigma,H} + \|\varphi^{\text{o}}\|_{\delta_3,b, o}.
\end{align*}
    For the incompressible Navier-Stokes equation (\ref{sect7-stokes-problem-solve}), we shall solve $\bf u$ in the following space:
    \begin{align}\label{sect7-space-collect4}
    E_{u}=\{ u\in L^2(\Omega):\nabla \cdot {\bf u}=0\text{~in~}\Omega ,\ \   \Vert {\bf u}\Vert_{S,\gamma-1,1}< M\e_0\},
    \end{align}
where $\bf u$ satisfies the boundary conditions, $\e_0>0$ is the sufficiently small but fixed number and $M>0$ is some fixed constant. 

We collect (\ref{sect7-space-collect1}), (\ref{sect7-space-collect2}), (\ref{sect7-space-collect3}) and (\ref{sect7-space-collect4}) to define space $\mathcal X$ as
\begin{align}\label{sect-top-X}
\mathcal X:=E_{\text{i}}\times E_{\text{o}}\times E_{p}\times E_{u}
\end{align}
In conclusion, we will solve (\ref{sect6-inn-outer-gluing-firsttwo-modify}) and (\ref{sect7-stokes-problem-solve}) in the space $\mathcal X$ by the fixed point theorem.

\medskip

\subsection{Estimate of remainder term $\phi$}\label{subsect-esimate-phi}
~

The coupled system (\ref{PKSNS-ss-equiv})$_1$--(\ref{PKSNS-ss-equiv})$_2$ is close in spirit to the classical minimal Keller-Segel model.  To find the desired steady state, it suffices to show ${\bf u}\cdot \nabla n$ is a perturbation term in the topology given above.  In this case, we are able to show the existence of $\phi=\frac{1}{\e^2}\Phi^{\text{i}}_H\chi_H+\varphi^{\text{o}}$ by performing the argument shown in \cite{KWX2022} with the slight modification.

\medskip

\textbf{Effect of the transport term ${\bf u}\cdot \nabla n$ in the inner problem.}

As discussed in Section \ref{sect0-intro}, the scaling invariance leads to the fully coupled property of system (\ref{PKSNS-ss-equiv}).  Concerning the linear theory established in Section \ref{sect4}, we find the term ${\bf u}\cdot \nabla n$ gets coupled in each mode.  More precisely, the mode $k$ of velocity field ${\bf u}$ solved from (\ref{sect7-stokes-problem-solve}) with the forcing term $-\e_0\nabla\cdot (\nabla \Gamma\otimes \nabla \psi_k)$ enters the inner problem of transported Keller-Segel model.  To study the role of advection $({\bf u}\cdot \nabla n)_k$, we note
\begin{align*}
{\bf u}\cdot \nabla_x n={\bf u}\cdot \nabla_x \Big(\frac{1}{\e^2}W+\frac{1}{\e^2}\Phi^{\text{i}}_H\chi_H+\varphi^{\text{o}}\Big),
\end{align*}
then use the topology defined in (\ref{sect-top-X}) and the norms given by  (\ref{sect7-inner-norm-def}) and (\ref{sect7-inner-norm-velocity}) to obtain in $Y$-variable and the inner region,
$$|\e^4 ({\bf u} \cdot \nabla n)_k| \leq \e^{\gamma}\frac{M\e_0}{1+|Y|^{4+\sigma}}.$$
It follows that 
$$\Vert \e^4 ({\bf u} \cdot \nabla n)_k\Vert_{\delta_2,4+\sigma,H}\leq \e^{\gamma-\delta_2}M\e_0.$$
If we choose $\delta_2=\gamma\in(0,1)$ and let $\e_0$ be fixed but sufficiently small number, we have ${\bf u}\cdot \nabla n$ is readily a perturbation compared to the other terms in right hand side $\mathcal H.$ 

\medskip

\textbf{Effect of the transport term ${\bf u}\cdot \nabla n$ in the outer problem.}
~

In the outer region, we can see the leading term in advection ${\bf u}\cdot \nabla_x n$ is ${\bf u}\cdot \nabla_x \varphi^{\text{o}}$.  Then, thanks to the cut-off function, we similarly substitute $n=\frac{1}{\e^2}W+\frac{1}{\e^2}\Phi^{\text{i}}_H\chi_H+\varphi^{\text{o}}$ into ${\bf u}\cdot \nabla n$ to get
\begin{align*}
|\e^4(1-\chi_H){\bf u}\cdot \nabla n|\lec &\e{(1-\chi_H)}|{\bf u}\cdot\nabla_y W|+{\e}{(1-\chi_H)}|W{\bf u}\cdot\nabla_y \chi_H|\nonumber\\
&+\e{(1-\chi_H)}|{\bf u}\cdot \nabla\Phi^{\text{i}}_H|+\e{(1-\chi_H)}|\Phi^{\text{i}}_H{\bf u}\cdot\nabla_y \chi_H|\nonumber\\
&+\e^3{(1-\chi_H)}|{\bf u}\cdot \nabla_y\varphi^{\text{o}}|\lec \frac{\e^{\gamma+\delta_2}(1-\chi_H)}{(1+|y-\xi'|)^{b+2}}\leq M\e_0 \e^{\gamma+b+2},
\end{align*}
which implies
$$\Vert \e^4(1-\chi_H) ({\bf u} \cdot \nabla n)\Vert_{\delta_3,b+2,o}\leq M\e_0\e^{\gamma+\delta_2-\delta_3}.$$
If we choose $\gamma>\delta_3-\delta_2\in(0,1)$, we find the drift term ${\bf u}\cdot \nabla n$ is a perturbation in the outer problem due to the smallness of $\e_0.$

We have shown that the transport term ${\bf u}\cdot \nabla n$ can be regarded as the perturbation and hence does not influence the fixed point argument.  Next, we focus on the incompressible Navier-Stokes equation (\ref{sect7-stokes-problem-solve}) and estimate the velocity field ${\bf u}.$

   \medskip
   
   \subsection{Estimate of the velocity field ${\bf u}$}\label{subsect71-stokes}
   ~

   For the analysis of solution ${\bf u}$ to (\ref{sect7-stokes-problem-solve}), the key is to estimate the coupled forcing term $-\e_0\nabla\cdot \mathcal F,$ which is
   \begin{align}\label{sect7-forcing-glue}
-\e_0\nabla\cdot\mathcal F=-\e_0\nabla\cdot (\nabla c\otimes \nabla c)+\e_0\nabla\bigg( \frac{\vert \nabla c\vert^2}{2}\bigg)+\e_0\nabla\Big(\frac{c^2}{2}\Big).
\end{align}
First of all, by using the $c$-equation of (\ref{PKSNS-ss-equiv}), we obtain the right hand side of (\ref{sect7-forcing-glue}) has the following equivalent form:
\begin{align}\label{sect7-forcing-glue-1}
-\e_0\nabla\cdot (\nabla c\otimes \nabla c)+\e_0\nabla\bigg( \frac{\vert \nabla c\vert^2}{2}\bigg)+\e_0\nabla\Big(\frac{c^2}{2}\Big)=\e_0(c-\Delta c)\nabla c.
\end{align}
Then one observes that in (\ref{sect7-forcing-glue-1}), the main contribution terms come from $\Gamma+H^{\e}+\big({\bar\Psi}^{\text{i}}_{H,0}+{\bar\Psi}^{\text{i}}_{H,1}+{\bar\Psi}^{\text{i}}_{H,\perp}\big)\chi$, where ${\bar\Psi}^{\text{i}}_{H,0}$, ${\bar\Psi}^{\text{i}}_{H,1}$ and ${\bar\Psi}^{\text{i}}_{H,\perp}$ are mode $0$, mode $1$ and higher modes in the remainder term, respectively.  Noting that $\Gamma$, $H^{\e}$ and ${\bar \Psi}^{\text{i}}_{H,0}$ dominate other small terms thanks to Lemma \ref{sect4-linear-theory-bdry}, we compute from $-\Delta_x H^{\e}+H^{\e}=-\Gamma$ in $\Omega$ that
\begin{align*}
(c-\Delta c)\nabla c\approx& (\Gamma+H^{\e}+{\bar \Psi}^{\text{i}}_{H,0}-\Delta \Gamma-\Delta H^{\e}-\Delta {\bar \Psi}^{\text{i}}_{H,0})(\nabla\Gamma+\nabla H^{\e}+\nabla {\bar \Psi}^{\text{i}}_{H,0})\nonumber\\
=&-\Delta_x\Gamma\nabla_x \Gamma-\Delta_x\Gamma\nabla_x H^{\e}\nonumber\\
&-\nabla\cdot \big(\nabla  {\bar \Psi}^{\text{i}}_{H,0}\otimes \nabla  {\bar \Psi}^{\text{i}}_{H,0}\big)+\nabla\Bigg( \frac{\big\vert \nabla  {\bar \Psi}^{\text{i}}_{H,0}\big\vert^2}{2}\Bigg)+\nabla\Bigg(\frac{( {\bar \Psi}^{\text{i}}_{H,0})^2}{2}\Bigg).
\end{align*}
Hence, by using $-\Delta_x\Gamma=\frac{1}{\e^2}W$, we next only need to evaluate
\begin{align}\label{sect7-7.30}
\nabla_x\cdot (\nabla_x\Gamma\otimes \nabla_x\Gamma)\text{~and~}\nabla_x\cdot \big(\nabla_x\Gamma\otimes \nabla_x{\bar \Psi}^{\text{i}}_{H,0}\big),
\end{align}
and
\begin{align}\label{sect7-7.30-continue}
\Delta_x\Gamma\nabla_x H^{\e}.
\end{align}
%t is left to formulate the inner--outer gluing scheme and show the existence via the fixed point theorem.  The key step is to study the forcing term $n\nabla c$.  We observe that it has a good structure.  Recall that
%$$\Delta c-c+n=0,$$
%then we have
%\begin{align}
%n\nabla c=-\Delta c\nabla c+c\nabla c=-\nabla\cdot(\nabla c\otimes\nabla c)+\nabla\bigg( \frac{\vert \nabla c\vert^2}{2}\bigg)+\nabla\Big(\frac{c^2}{2}\Big)
%\end{align}
First of all, noting that ${\bar\Psi}^{\text{i}}_{H,0}$ and $\Gamma$ are both radial, we claim $\nabla_x\cdot(\nabla_x \Gamma\otimes \nabla_x\Gamma)$ and $\nabla_x\cdot \big(\nabla_x\Gamma\otimes \nabla_x{\bar \Psi}^{\text{i}}_{H,0}\big)$ can be written as potentials.  To show this, we shall prove the following lemma:
\begin{lemma}\label{sect7-aux-lemma-potential}
Assume $R_1(|y|)$ and $R_2(|y|)$, $y\in \mathbb R^2$ are radial functions.  Then we have there exists a scalar function $\omega(y)$ such that 
\begin{align}\label{sect7-lemma-Stokes}
\nabla_y\cdot (\nabla_y R_1\otimes \nabla_yR_1)=\nabla_y\omega.
\end{align}
\end{lemma}
\begin{proof}
We rewrite the left hand side of (\ref{sect7-lemma-Stokes}) in the polar coordinate $(\rho,\theta)$ and study by component to get
\begin{align}\label{sect7-lemma-combine-1}
[\nabla_y\cdot (\nabla_y R_1\otimes \nabla_yR_1)]_1=&\partial_{y_1}(\partial_{y_1}R_1\partial_{y_1}R_2)+\partial_{y_2}(\partial_{y_2}R_1\partial_{y_1}R_2)\nonumber\\
=&\cos\theta \partial_{\rho}(\partial_{\rho}R_1\partial_{\rho}R_2\cos^2\theta)-\frac{\sin\theta}{\rho}\partial_{\theta}(\cos^2\theta)\partial_{\rho}R_1\partial_{\rho}R_2\nonumber\\
&+\sin\theta\partial_{\rho}(\partial_{\rho}R_1\partial_{\rho}R_2\sin\theta\cos\theta)+\frac{\cos\theta}{\rho}\partial_{\theta}(\partial_{\rho}R_1\partial_{\rho}R_2\sin\theta\cos\theta)\nonumber\\
=&\cos\theta\Big[\partial_{\rho}(\partial_{\rho}R_1\partial_{\rho}R_2)+\frac{1}{\rho}\partial_{\rho}R_1\partial_{\rho}R_2\Big].
\end{align}
Similarly, we obtain the second component satisfies
\begin{align}\label{sect7-lemma-combine-2}
[\nabla_y\cdot (\nabla_y R_1\otimes \nabla_yR_1)]_2=&\partial_{y_1}(\partial_{y_1}R_1\partial_{y_2}R_2)+\partial_{y_2}(\partial_{y_2}R_1\partial_{y_2}R_2)\nonumber\\
=&\sin\theta\Big[\partial_{\rho}(\partial_{\rho}R_1\partial_{\rho}R_2)+\frac{1}{\rho}\partial_{\rho}R_1\partial_{\rho}R_2\Big].
\end{align}
Combining (\ref{sect7-lemma-combine-1}) and (\ref{sect7-lemma-combine-2}), one lets
$$\omega:=\Big(\partial_{\rho}R_1\partial_{\rho}R_2+\int\frac{1}{\rho}\partial_{\rho}R_1\partial_{\rho}R_2\, d\rho\Big),$$
then has
$$\nabla_y\cdot (\nabla_y R_1\otimes \nabla_yR_1)=\nabla_y\omega,$$
which finishes the proof.
\end{proof}
To show our claim, we compute (\ref{sect7-7.30}) to get
\EQ{\label{sect7-7.34}
\nabla_x\cdot (\nabla_x\Gamma\otimes \nabla_x\Gamma)=&\frac{1}{\e^3}\nabla_y\cdot (\nabla_y\Gamma\otimes \nabla_y\Gamma) \\
\nabla_x\cdot \big(\nabla_x\Gamma\otimes \nabla_x{\bar \Psi}^{\text{i}}_0\big)=&\frac{1}{\e^3}\nabla_y\cdot \big(\nabla_y\Gamma\otimes \nabla_y{\bar \Psi}^{\text{i}}_0\big). 
}
Noting the facts that $\Gamma$ and ${\bar\Psi}^{\text{i}}_0$ are radial, we apply Lemma \ref{sect7-aux-lemma-potential} on (\ref{sect7-7.34}) to readily obtain $\nabla_x\cdot (\nabla_x\Gamma\otimes \nabla_x\Gamma)$ and $\nabla_x\cdot \Big(\nabla_x\Gamma\otimes \nabla_x{\bar \Psi}^{\text{i}}_0\Big)$ are potentials, which can be absorbed in the pressure $\tilde P$ of problem \eqref{sect7-stokes-problem-solve}.  It is obvious that other terms such as $\nabla\Big( \frac{\vert \nabla  {\bar \Psi}^{\text{i}}_{H,0}\vert^2}{2}\Big)$ and $\nabla\Big(\frac{( {\bar \Psi}^{\text{i}}_{H,0})^2}{2}\Big)$ are potentials.  On the other hand, we have to estimate (\ref{sect7-7.30-continue}), which is
\begin{align*}
\vert \Delta_x\Gamma\nabla_x H^{\e}\vert \lec \frac{1}{\e^2}\frac{1}{(1+|Y|)^4},
\end{align*}  
so that
\begin{align*}
\Vert \Delta_x\Gamma\nabla_x H^{\e}\Vert_{S,\gamma-3,3+a}\ll 1
\end{align*}
with $\gamma\in(0,1).$
%Focusing on \eqref{sect7-7.30-continue}, we note that $\nabla_x H^{\e}(\xi+\e y)=\nabla_x H^{\e}(\xi)+\e\nabla^2_xH^{\e}(\xi)y$ estimate 
%\begin{align}
%\nabla_x\cdot (\nabla_x\Gamma\otimes \nabla_xH^{\e})=&\frac{1}{\e^2}\nabla_y\cdot(\nabla_y\Gamma\otimes %\nabla_xH^{\e})\\
%=&\frac{1}{\e^2}\nabla_y\cdot(\nabla_y\Gamma\otimes \nabla_xH^{\e}(\xi))+\frac{1}{\e}\nabla_y\cdot (\nabla_y \Gamma\otimes y).
%\end{align}

In conclusion, the leading order term in $\mathcal F$ defined by (\ref{sect7-forcing-glue}) is 
\begin{align*}
    \big|\nabla_x \Gamma\otimes \nabla_x{\bar\Psi}^{\text{i}}_{1}\big|\leq \frac{\e^{\delta_2-2}}{(1+|Y|)^{2+\sigma}}\big\Vert {\bar\Psi}^{\text{i}}_1\big\Vert_{\delta_2,\sigma,H},
\end{align*}
from which we conclude that
\begin{align*}
\big\Vert \e_0\nabla_x \Gamma\otimes \nabla_x{\bar\Psi}^{\text{i}}_{1}\big\Vert_{S,\gamma-2,2+a}\leq \e_0,
\end{align*}
where we have chosen $\delta_2=\gamma\in(0,1)$ and $a=\sigma\in(0,1).$  For the estimate of ${\bf u}$, we give the following remarks:
\begin{remark}
~
\begin{itemize}
    \item As shown in Lemma \ref{sect4-linear-theory-bdry}, we have constructed the mode $0$ solution ${\bar\Psi}_0^{\text{i}}$ with the logarithmic growth owing to the existence of non-trivial kernel $Z_0$, which may cause the slow decay difficulty in the fixed point argument.
    \item The forcing involving radial modes can be absorbed into the pressure $P$ thanks to Lemma \ref{sect7-aux-lemma-potential}.
    \item The equation $-\Delta_x H^{\e}+H^{\e}=-\Gamma$ helps us rule out the slow decay problem caused by $H^{\e}$.
    \item The smallness of $\e_0$ and the decay property of mode $1$ solution guarantee the required estimate of the coupled forcing $\mathcal F$.
\end{itemize}
\end{remark}

It is necessary to discuss the advection ${\bf u}\cdot \nabla {\bf u}$. 
 As mentioned in Remark \ref{sect5-remark-5.6}, considering the solution ${\bf u}\in E_u$ with $E_u$ defined by (\ref{sect7-space-collect4}), the nonlinear term ${\bf u}\cdot \nabla {\bf u}$ in (\ref{sect7-stokes-problem-solve}) can be regarded as a perturbation compared to the forcing $\e_0\nabla\cdot \mathcal F$.  Indeed, since ${\bf u}\in E_{u}$, we get
\begin{align*}
|{\bf u}\cdot \nabla {\bf u}|\lec \frac{\e^{2\gamma-3}}{(1+|Y|)^{3}},
\end{align*}
which implies 
\begin{align*}
\Vert {\bf u}\cdot \nabla {\bf u}\Vert_{S,\gamma-3,a+3} \lec \e^{\gamma-a}\ll 1
\end{align*}
if we choose $0<a<\gamma<1.$

In summary, we obtain the incompressible Navier-Stokes equation (\ref{sect7-stokes-problem-solve}) serves as the following perturbed Stokes system:
\begin{align*}
\nabla\tilde P=\Delta {\bf u}-\e_0\nabla \cdot \mathcal F_1(\Phi^{\text{i}}_H,\varphi^{\text{o}},{\bf u}, \xi_H),
\end{align*}
with
\begin{align*}
\mathcal F_1(\Phi^{\text{i}}_H,\varphi^{\text{o}},{\bf u}, \xi_H)=\mathcal F(\Phi^{\text{i}}_H,\varphi^{\text{o}},{\bf u}, \xi_H)+{\bf u}\otimes {\bf u},
\end{align*}
where we have rewritten ${\bf u}\cdot\nabla{\bf u}=\nabla\cdot ({\bf u}\otimes {\bf u})$ by using the fact that ${\bf u}$ is divergence-free.

\medskip

\subsection{Fixed point argument: proof of Theorem \ref{thm11}}
~

We are ready to perform the gluing procedure and prove Theorem \ref{thm11}.  Similarly as in \cite{KWX2022}, we use Lemma \ref{sect4-linear-theory-bdry}, Lemma \ref{sect4-outer-problem-linear-theory} and Proposition \ref{sect5-inner-outer-combine-prop} to rewrite the solution $\boldsymbol{\mathcal E}=({\bf u},\Phi^{\text{i}}_H,\varphi^{\text{o}}, \xi_H)^T$ as 
\EQN{
{\bf u}= \mathcal{A}_{u}({\bf u},\Phi^{\text{i}}_H, \varphi^{\text{o}}, \xi_{H}), \ \ \Phi^{\text{i}}_H = \mathcal{A}_{\text{i}}({\bf u},\Phi^{\text{i}}_H, \varphi^{\text{o}}, \xi_{H}),   \ \ \varphi^{\text{o}} = \mathcal{A}_o({\bf u},\Phi^{\text{i}}, \varphi^{\text{o}}, \xi_H),}
and 
\EQN{\xi_H= \mathcal{A}_{p}({\bf u},\Phi^{\text{i}}, \varphi^{\text{o}}, \xi_{H}).}
Recall that space $\mathcal X$ is given by \eqref{sect-top-X}, then we define the norm $\Vert\cdot\Vert_{ X}$ as
\begin{align*}
\Vert{\mathcal E} \Vert_{ X}= \big\|\Phi^{\text{i}}\big\|_{\delta_2,2 + \sigma,H} + \|\varphi^{\text{o}}\|_{\delta_3,b, o}+\Vert {\bf u}\Vert_{S,\gamma-1,1}+\Vert \xi_H\Vert_{p}.
\end{align*}
%Combining (\ref{sect7-EQ-gluing-inn-1}) and (\ref{sect7-EQ-gluing-inn-2}), we let $\tilde {\boldsymbol{\phi}}=(\Phi^{\text{i}},\varphi^{\text{o}},\xi_H)^T$ and
%\begin{align}
%{\tilde E}_{\phi}=E_{\text{i}}\times E_{\text{o}}\times E_p, \ \ \Vert \tilde{\boldsymbol{\phi}} \Vert_{{\tilde {E}}_{\phi}}=  \|\phi\|_{\delta_2,2 + \sigma,H} + \|\varphi^{\text{o}}\|_{\delta_3,b, o}+\Vert \xi\Vert_{p}.
%\end{align}

Now, we formulate the fixed point problem as
\begin{align*}
\boldsymbol{\mathcal E}=\mathcal A(\boldsymbol{\mathcal E}),
\end{align*}
where $\mathcal A(\boldsymbol{\mathcal E})$ is given by
\begin{align*}
\mathcal A(\boldsymbol{\mathcal E})=\big(\mathcal A_u(\boldsymbol{\mathcal E}), \mathcal A_{\text{i}}(\boldsymbol{\mathcal E}),\mathcal A_{\text{o}}(\boldsymbol{\mathcal E}),\mathcal A_{p}(\boldsymbol{\mathcal E})\big), \ \ \mathcal A: {\bar {\mathcal B}}_1\subset \mathcal X\rightarrow   \mathcal X
\end{align*}
with 
$$\mathcal B_1:=\big\{\boldsymbol{\mathcal E}\in \mathcal X:\Vert\boldsymbol{\mathcal E}\Vert_{X}< 1\big\}.$$

We claim that $\mathcal A$ is a contraction mapping from ${\bar {\mathcal B}}_1$ onto itself.  The proof is based on Section 4 and Section 5 of \cite{KWX2022}, we only need to perform the slight modification.  For the sake of completeness, we give the sketch of arguments.  First of all,  we shall show for $\|\boldsymbol{\mathcal E}\|_X \leq 1$,
    \begin{equation*}
     \|\mathcal{A}(\boldsymbol{\mathcal E})\|_{X} \leq 1.
          \end{equation*}
For the inner operator $\mathcal A_{\text{i}}$, since we have shown the drift term ${\bf u}\cdot \nabla n$ is a small perturbation, one finds for $Y \in B_{2\delta/\e}(0)$, the leading term in $\mathcal H$ satisfies
 $$\e\vert\nabla_Y\Phi^{\text{i}}\cdot\nabla_X  H^{\e}(\xi)\vert\leq \frac{C\delta\e^{\delta_2}}{( 1+\vert Y\vert)^{4+\sigma}},$$
 where $C>0$ is some constant.  On the other hand, since we flatten the boundary locally near the location $\xi$, it is necessary to estimate the new error $N_{\rho}$ given by \eqref{sect7-inner-new-error}.  However, according to \eqref{sect3-rhoY1-expand}, \eqref{sect3-rhoprime-Y1} and \eqref{sect3-rhodouble-Y1}, one has $N_{\rho}$ is a small perturbation compared to error $\mathcal H$ shown in \eqref{sect6-inn-outer-gluing-firsttwo-modify}. 
 
 For the outer operator $\mathcal A_{\text{o}}$, due to the smallness of ${\bf u}\cdot \nabla_x n$ in the outer region, we note that the error term involving the inner solution $\Phi^{\text{i}}$ are the leading one, then obtain for $\nabla_Y \Phi^{\text{i}}\cdot \nabla_Y \chi$,
  \begin{align*}
      \e^{\delta_2-\delta_3} \vert  \nabla_Y \Phi^{\text{i}} \cdot \nabla_Y \chi_H\vert \leq C_1 \e^{\delta_2-\delta_3}\frac{\e^{\delta_2}}{(1+\vert Y\vert)^{4+\sigma}}\leq C_1\frac{\e^{\delta_3-\delta_2}}{\delta^{2\delta_2-2\delta_3}} \frac{\e^{\delta_2}}{(1+\vert Y\vert)^{4+\sigma+2\delta_2-2\delta_3}},
    \end{align*}
    where $C_1>0$ is some constant.  For the Stokes operator $\mathcal A_{u}$, as discussed in Subsection \ref{subsect71-stokes}, we have from the smallness of $\e_0$ that $\Vert \mathcal A_u\Vert_{X}\leq 1.$  For the parameter operator $\mathcal A_p$, since ${\bf u}\cdot \nabla_x n$ is proved to be the perturbation in the inner problem, we only need to adjust $\xi_H$ to eliminate Lagrange multipliers $m_0$ and $m_1$ given by \eqref{m0eq0}, which will be discussed later on.
    
    As shown in \cite{KWX2022}, by choosing suitable $\delta$, $\delta_2,$ $\delta_3,$ $\sigma$, $b$ $\gamma$ and $a$, we prove that $\mathcal A$ maps from ${\bar {\mathcal B}}_1$ into itself.  Here we choose $\delta$, $\delta_2,$ $\delta_3,$ $\sigma$ and $b$ such that $\delta \sim \sqrt{\e}$, $\sigma=a\in\big(0,\frac{4}{5}\big)$, $\gamma=\delta_2\in(\sigma,1)$, $0<\delta_3=\delta_2+\frac{\sigma}{4}<1$ and $b=2+\frac{\sigma}{2}$.  Moreover, under the same restriction of parameters, we can similarly show that there exist constants $\tau_1,\tau_2, \tau_3\in(0,1)$ such that 
\begin{align}\label{sect8-operator-contraction-combine1}
    \left\{\begin{array}{ll}
\Vert \mathcal A_{\text{i}}[\boldsymbol{\mathcal E}_1]-\mathcal A_{\text{i}}[\boldsymbol{\mathcal E}_2]\Vert_{\delta_2,2+\sigma,H} \leq \tau_1\Vert \boldsymbol{\mathcal E}_1-\boldsymbol{\mathcal E}_2\Vert_{X},\\
\Vert \mathcal A_{o}[[\boldsymbol{\mathcal E}_1]-\mathcal A_o[[\boldsymbol{\mathcal E}_2]\Vert_{\delta_3,b,o} \leq \tau_2\Vert\boldsymbol{\mathcal E}_1-\boldsymbol{\mathcal E}_2\Vert_{X},\\
\Vert \mathcal A_{u}[\boldsymbol{\mathcal E}_1]-\mathcal A_u[\boldsymbol{\mathcal E}_2]\Vert_{S,\gamma-1,1} \leq \tau_3\Vert\boldsymbol{\mathcal E}_1-\boldsymbol{\mathcal E}_2\Vert_{X}
    \end{array}
    \right.
\end{align}
for any $\boldsymbol{\mathcal E}_1$, $\boldsymbol{\mathcal E}_2 \in\mathcal X$ with $\|\boldsymbol{\mathcal E}_1\|_X$, $\|\boldsymbol{\mathcal E}_2\|_X  \le 1$.  Focusing on the operator $\mathcal A_p$, we claim there exists $\tau_4\in(0,1)$ such that 
\begin{align}\label{sect8-operator-contraction-combine2}
\Vert \mathcal A_{p}[\boldsymbol{\mathcal E}_1]-\mathcal A_p[\boldsymbol{\mathcal E}_2]\Vert_{p} \leq \tau_4\Vert\boldsymbol{\mathcal E}_1-\boldsymbol{\mathcal E}_2\Vert_{X}.
\end{align}
Combining (\ref{sect8-operator-contraction-combine1}) with (\ref{sect8-operator-contraction-combine2}), one finds
       \begin{equation*}
            \mathcal{A}\big(\bar{\mathcal{B}}_1\big) \subset \bar{\mathcal{B}}_1 \ \  \text{and} \ \  \|\mathcal{A}(\mathcal E_1) - \mathcal{A}(\mathcal E_2)\|_X \le \hat\tau\|\mathcal E_1 - \mathcal E_2\|_X  \ \ \text{,} \ \ \forall \mathcal E_1, \, \mathcal E_2 \in \bar{\mathcal{B}}_1,
       \end{equation*}
       where constant $\hat\tau\in(0,1)$.  It immediately follows that there exist the solution $\mathcal E$ satisfying $\mathcal E = \mathcal{A}(\mathcal E)$ by the fixed point theorem.

Now, we prove the claim (\ref{sect8-operator-contraction-combine2}), i.e. the contraction property of the operator $\mathcal A_p$.  It suffices to adjust the location $\xi_H$ such that $m_0$ and $m_1$ defined by (\ref{m0eq0}) satisfy $m_0=m_1=0$.  We shall show $\xi_H$ only takes care the first moment orthogonality condition and the mass orthogonality condition will be satisfied automatically.  To begin with, by using the Neumann boundary conditions of $(n,c)$, we integrate the $n$-equation in \eqref{PKSNS-ss} with the Lagrange multiplier by parts to get
\begin{align*}
0=\int_{\partial\Omega}\Big(\frac{\partial n}{\partial \boldsymbol{\nu}}-n\frac{\partial c}{\partial \boldsymbol{\nu}}\Big)\, dS+\int_{\Omega}(\nabla\cdot {\bf u})n \, dx-\int_{\partial\Omega}({\bf u}\cdot \boldsymbol{\nu}) n\, dS+m_0,
\end{align*}
which implies $m_0=0$ since ${\bf u}$ is divergence-free and ${\bf u}$ satisfies the no-slip boundary condition.

We next consider the first moment orthogonality condition stated in (\ref{331}).  We have shown that ${\bf u}\cdot \nabla n$ can be regarded as a small perturbation in Subsection \ref{subsect-esimate-phi}.  Thus, similarly as in \cite{KWX2022}, we have the leading term in the $Y$-variable is $\nabla \cdot(W \nabla_X H^{\e})$.  Then by using the integration by parts, one gets
\begin{equation}\label{firstbd}
\begin{split}
    &\e\int_{\R^2_+}\nabla_Y\cdot(W \nabla H^{\e})Y_1 \chi_{H}\, dY\\
     =&\e\int_{\partial \R_+^2}W\nabla H^{\e}\cdot\boldsymbol{\nu}Y_1\chi_{H}\, dS  + \e\int_{\R^2_+}W\nabla_x H^{\e} \cdot e_1 \chi_{H} \, d Y\\
     &+ \e\int_{\R^2_+}W \nabla_x H^{\e} \cdot  Y_1 \nabla \chi_{H}\, dY,
    \end{split}
\end{equation}
where $e_1=(1,0)^T.$  Then we expand $\nabla_XH^{\e}$ as 
\begin{align}\label{slowH-expansion}
\nabla_X H^{\e}(x)=\nabla_X H^{\e}(\xi)+\e\nabla^2_X H^{\e}(\xi)Y+O(\e^2).
\end{align}
Upon substituting (\ref{slowH-expansion}) into (\ref{firstbd}), we find the boundary integral in (\ref{firstbd}) exactly matches the corresponding error which comes from $\bar \beta$ given in (\ref{sect6-inn-outer-gluing-firsttwo-modify}).  Proceeding term  $\nabla_x\cdot(\Phi^{\text{i}}\nabla_x H)$ with the same argument, we finally obtain from the first moment condition that 
$$\xi_{H,1}=O(\e^{\gamma_p}),$$
where $0<\gamma_p<1$ but $\gamma_p\approx 1.$  In addition, we have $\xi_{H,2}\equiv 0$ since the centre of boundary spot is located at $\partial \Omega.$  As a consequence, the leading term of $\xi_{H}$ is precisely given by the critical point of (\ref{jm}).  

Now, we have established the existence of desired single boundary spot shown in Theorem \ref{thm11} via the fixed point theorem.  For the fixed point argument shown in Section \ref{inn-out-gluing-sect}, we give the following remarks:
\begin{remark}
~
\begin{itemize}
    \item For the mass orthogonality condition arising from the study of $n$-equation in (\ref{PKSNS-ss-equiv}), we show it must be satisfied without adjusting any parameter by integrating the $n$-equation by parts. 
    \item While solving the incompressible Navier-Stokes equation (\ref{sect7-stokes-problem-solve}), we also impose the compatibility condition due to the existence of non-trivial kernel $\beta.$  However, as discussed in Section \ref{sect-model-stokes}, we find \eqref{sect5-solvability-conditions-eq} holds without adjusting any parameter.
    \item the smallness of $\e_0$ guarantee that the operator $\mathcal A_p$ has the contraction property. 
\end{itemize}
\end{remark}
For the multi-spot case, we only need to modify the ansatz as
\begin{align*}
 n_{\e}(x)=\frac{1}{\e^2}\sum_{j =1}^m W\bigg(\frac{x - \xi_j}{\e}\bigg) +\bigg(\frac{1}{\e^2}\sum_{j=1}^m \Phi^{\text{i}}_{j}(y-\xi'_j) \chi_{j}(y-\xi'_j) + \varphi^{\text{o}}(x)\bigg)
\end{align*}
 \begin{align*}
  c_{\e}(x) = \sum_{j =1}^{m}\Big[\Gamma\Big(\frac{x-\xi_j}{\e}\Big) +  \hat c_j H^{\e}(x, \xi_j)\Big]  +\bigg(\sum_{j=1}^m \Psi^{\text{i}}_{j}(y-\xi'_j) \chi_{j}(y-\xi'_j) + \psi^{\text{o}}(x)\bigg),
 \end{align*}
  where $\xi_j'=\xi_j/\e$, $(\Phi^{\text{i}}_j,\Psi^{\text{i}}_{j}, \varphi^{\text{o}},\psi^{\text{o}})$ are remainder terms, $\xi_j\in \Omega$ and $\hat c_j=8\pi$ for $j\leq k$; $\xi_j\in \partial\Omega$ and $\hat c_j=4\pi$ for $k<j\leq m$.  Next, we can solve the transported Keller-Segel models and the Navier-Stoke equation together, then perform the argument to construct the desired multi-spots in the same manner.  We omit the details.

\appendix \section{Appendix: Local-in-time Existence}\label{appen-local-in-time}  
 In this appendix, we shall adapt the argument shown in \cite{winkler2012global-NS-PKS} and take the slight modification to prove the local well-posedness of system (\ref{PKSNS-time-dependent}), which is
 
\medskip

 \textit{Proof of Lemma \ref{sect3-lemma-local-in-time-1}:}
 
\textbf{Existence:}  Our strategy is to employ the contraction mapping theorem.  To this end, we fix $T\in(0,1)$ and $R>0$.  Concerning $\iota>0,$ we define the Banach space $ X$ as 
\begin{align*}
 X:=L^{\infty }((0,T); C^0(\bar\Omega)\times W^{1,\infty}(\Omega)\times {\bf D_1}(A_q)),~q>2.
\end{align*}
where ${\bf D}_1$ is defined by
\begin{align}\label{appendix-bfD1-operator-space-final}
{\bf D}_1(A_q) = \bket{{\bf u}\in{\bf W}^{1,q}(\Om)\cap{\bf L}^q_\si(\Om): \bkt{\mathbb S({\bf u})\boldsymbol{\nu}}_{\boldsymbol{\tau}}=0,~{\bf u}\cdot \boldsymbol{\nu}=0\ \text{ on }\pd\Om}.
\end{align}
Let the closed set $S$ be 
 \begin{align*}
    S:=\{(n,c,{\bf u})\in X|\Vert n(\cdot, t)\Vert_{L^\infty}+\Vert c\Vert_{W^{1,\infty}}+\Vert {\bf  u}(\cdot, t)\Vert_{{\bf W}^{1,q}}\leq R\text{ for a.e. }t\in(0,T)\}.
 \end{align*}
Then we introduce a mapping $\hat {\boldsymbol{\Phi}}=(\hat{{\Phi}}_1,\hat {{\Phi}}_2,\hat {{\Phi}}_3)$ on $S$ by defining 
\begin{align}\label{appendix-A1-hatphi-1}
\hat \Phi_1(n,c,{\bf u})(\cdot,t):=e^{t\Delta }n_0-\int_0^{t} e^{(t-s)\Delta}\{\nabla \cdot(n\nabla c)+{\bf u}\cdot \nabla n\}(\cdot,s)\, ds,
\end{align}
\begin{align*}
\hat \Phi_2(n,c,{\bf u})(\cdot,t):=e^{t(\iota\Delta-1)}c_0-\int_0^{t} e^{(t-s)(\iota\Delta-1)}n(\cdot,s)\, ds,
\end{align*}
and 
\begin{align}\label{appendix-A1-hatphi-3}
\hat \Phi_3(n,c,{\bf u})(\cdot,t):=e^{-tA_q}{\bf u}_0+\int_0^{t} e^{-(t-s)A_q}\mathbb P[({\bf u}\cdot \nabla) {\bf u}+n\nabla c ](\cdot,s)\, ds
\end{align}
for $(n,c,\bf u)\in S$ and $t\in(0,T).$  Here and below, $(e^{t\Delta})_{t\geq 0}$, $(e^{-tA_q})_{t\geq 0}$ and $\mathbb P$ are the Neumann heat semigroup, the Stokes semigroup with the Navier boundary condition and the Helmholtz projection in ${\bf L}^2(\Omega),$ respectively.

For $q>2$, we define $B$ as the sectorial operator $-\Delta+1$ in $L^q(\Omega)$ with the homogeneous Neumann boundary condition.  Then, one has the fact that $D(B^{\beta})\hookrightarrow C^0(\bar\Omega)$ continuously, where we pick $\beta\in(0,1)$ such that $2q<\beta<1.$  Noting that $\nabla\cdot(n{\bf u})={\bf u}\cdot\nabla n$, we similarly obtain as in \cite{winkler2012global-NS-PKS} that there exist constants $c_1,$ $c_2,$ $c_3(R)>0$ such that 
\begin{align*}
\Vert \hat \Phi_1(n,c,{\bf u})(\cdot,t)\Vert_{L^\infty}\leq &\Vert e^{t\Delta} n_0\Vert_{L^\infty}+c_1\int_0^t \Vert B^{\beta}e^{-(t-s)(B-1)}n(\cdot,s)\Vert_{L^q}\, ds\nonumber\\
\leq & \Vert n_0\Vert_{L^\infty}+c_2\int_0^t (t-s)^{-\beta-\frac{1}{2}}\Vert (n\nabla c+n {\bf u})(\cdot,s)\Vert_{L^q} \, ds\nonumber\\
\leq &\Vert n_0\Vert_{L^\infty}+c_3(R) T^{\frac{1}{2}-\beta},
\end{align*}
for all $t\in(0,T).$  Here we have used Theorem \ref{thm-ghosh-3.5.4} to get ${\bf u}\in {\bf L}^q$ since $\alpha\in\big(\frac{1}{2},1\big).$

For $\hat \Phi_2(n,c,{\bf u})(\cdot,t),$ we perform the similar argument to find there exist constants $c_4$, $c_5$, $c_6(R)$ such that
\begin{align*}
\Vert \hat \Phi_2(n,c,{\bf u})(\cdot,t)\Vert_{W^{1,\infty}}\leq &\Vert e^{t(\iota\Delta-1)} c_0\Vert_{W^{1,\infty}}+c_4\int_0^t \Vert \hat B^{\beta_1}e^{-(t-s)(\hat B-1)}\nabla\cdot(n\nabla c+n{\bf u})(\cdot,s)\Vert_{L^q}\, ds\nonumber\\
\leq & \Vert c_0\Vert_{W^{1,\infty}}+c_5\int_0^t (t-s)^{-\beta_1}\Vert n(\cdot,s)\Vert_{L^q} \, ds\nonumber\\
\leq &\Vert c_0\Vert_{W^{1,\infty}}+c_6(R) T^{{1}-\beta_1},
\end{align*}
where $\hat B=-\iota\Delta+1$, $\beta_1$ is chosen such that $\beta_1\in\big(\frac{1}{2},1\big)$ and $t\in(0,T)$.

Finally, we proceed $ \hat \Phi_3(n,c,{\bf u})(\cdot,t) $ and use \eqref{eq-ghosh-3.66} in Theorem \ref{sect3-semigroup-estimate-pre} to obtain
\begin{align*}
\Vert \hat \Phi_3(n,c,{\bf u})(\cdot,t)\Vert_{W^{1,q}}\leq &\Vert e^{-tA_q} {\bf u}_0\Vert_{W^{1,q}}+c_7\int_0^t \Vert e^{-(t-s)A_q}[({\bf u}\cdot 
\nabla) {\bf u}+n\nabla c](\cdot,s)\Vert_{L^q}\, ds\nonumber\\
\leq & \Vert {\bf u}_0\Vert_{W^{1,q}}+c_8\int_0^t (t-s)^{-\frac{1}{2}}(\Vert ({\bf u}\cdot 
\nabla){\bf u}\Vert_{{\bf L}^q}+\Vert n\nabla c\Vert_{L^q})(\cdot,s) \, ds\nonumber\\
\leq &\Vert  {\bf u}_0\Vert_{W^{1,q}}+c_9(R) T^{\frac{1}{2}},
\end{align*}
where constants $c_7$, $c_8$ and $c_9(R)$ are positive.  Here we have applied the H\"{o}lder's inequality to find
\begin{align*}
\Vert ({\bf u}\cdot 
\nabla){\bf u}\Vert_{{\bf L}^q}\leq \Vert {\bf u}\Vert_{{\bf L}^\infty}\Vert \nabla {\bf u}\Vert_{{\bf L}^q}\leq c_{10}\Vert {\bf u}\Vert^2_{W^{1,q}},
\end{align*}
and
\begin{align*}
\Vert n 
\nabla c\Vert_{ L^q}\leq c_{11}\Vert n\Vert_{L^\infty}\Vert c\Vert_{W^{1,\infty}},
\end{align*}
with $c_{10}$ and $c_{11}$ are positive constants. 

Now, we choose $R>0$ large enough at first, then take $T>0$ be sufficiently small to prove that $\hat {\boldsymbol{\Phi}}$ maps $S$ into itself and is further a contraction mapping.  Then by the Banach fixed point theorem that there exists $(n, c, {\bf u})\in S$ such that $\hat  {\boldsymbol{\Phi}}(n, c, {\bf u}) = (n, c, {\bf u})$.  Thanks to the standard bootstrap arguments, parabolic regularity theories and de Rham theory, we have $(n, c,{\bf u},P)$ readily solves (\ref{PKSNS-time-dependent}) classically in $\Omega\times(0, T)$.  Noting that $T$ only depends on $\Vert n_0\Vert_{L^\infty}$, $\Vert c_0\Vert_{W^{1,\infty}}$ and $\Vert {\bf u}_0\Vert_{W^{1,q}}$ with $q>2$, one further obtains \eqref{eq-in-lemma3.1} holds. 

Consider the case of $\iota=0$, then we similarly define
\begin{align*}
 \hat X:=L^{\infty }((0,T); C^0(\bar\Omega)\times {\bf D_1}(A_q)),~q>2.
\end{align*}
where ${\bf D}_1$ is given by (\ref{appendix-bfD1-operator-space-final}).  Let the closed set $\hat S$ be 
 \begin{align*}
    \hat S:=\{(n,{\bf u})\in \hat X|\Vert n(\cdot, t)\Vert_{L^\infty}+\Vert {\bf  u}(\cdot, t)\Vert_{{\bf W}^{1,q}}\leq R\text{ for a.e. }t\in(0,T)\},
 \end{align*}
 and introduce a mapping $\tilde  {\boldsymbol{\Phi}}=(\tilde {{\Phi}}_1,\tilde {{\Phi}}_3)$ on $\hat S$ with $\tilde {{\Phi}}_1$ and $\tilde {{\Phi}}_3$ defined by (\ref{appendix-A1-hatphi-1}) and (\ref{appendix-A1-hatphi-3}), respectively.  To show the mapping $\tilde {\boldsymbol{\Phi}}$ has the contraction property, we first notice that by applying the standard elliptic estimate on the $c$-equation, one has for $t\in(0,T),$ there exists $c\in W^{1,\infty}(\Omega)$ since $n$ is assumed to satisfy $n\in C^0(\bar\Omega)$.  Then, proceeding in a similar way as the case of $\iota>0$, one can show $\tilde  {\boldsymbol{\Phi}}$ is the contraction mapping if $R$ is chosen large enough then $T$ is chosen sufficiently small, and further there exists $(n,c,{\bf u})$ satisfying (\ref{eq-in-lemma3.1}) solves (\ref{PKSNS-time-dependent}) classically.

In summary, we obtain the local-in-time existence of the classical solution $(n,c,\bf{u})$ to system (\ref{PKSNS-time-dependent}) with $\iota\geq 0$ satisfying (\ref{eq-in-lemma3.1}).

The positivity of $(n,c)$ is the direct consequence of the parabolic weak and strong maximum principles.

\medskip

\textbf{Uniqueness:} we shall argue by contradiction.  Assume there are two solutions $(n_1,c_1,{\bf u}_1,P_1)$ and $(n_2,c_2,{\bf u}_2,P_2)$ satisfying (\ref{PKSNS-time-dependent}) in $\Omega\times (0,T)$ for some $T>0.$  We subtract the $n_1$-equation and $n_2$-equation then multiply it by $n_1-n_2$ to get for $t\in(0,T_0)$ with $T_0<T,$
\begin{align}
&\frac{1}{2}\frac{d}{dt}\int_{\Omega}(n_1-n_2)^2\, dx+\int_{\Omega}|\nabla(n_1-n_2)|^2\, dx\nonumber\\
=&\int_{\Omega} (n_1-n_2)\nabla c\cdot \nabla(n_1-n_2)\, dx+\int_{\Omega}n_2 \nabla(c_1-c_2)\cdot\nabla(n_1-n_2)\, dx\nonumber\\
&-\int_{\Omega}(u_1-u_2)\nabla n(n_1-n_2)\, dx-\int_{\Omega}u_2\nabla(n_1-n_2)(n_1-n_2)\, dx\nonumber.
\end{align}
Since $T_0<T,$ one has for some $q>2,$
\begin{align*}
\Vert n_1(\cdot,t)\Vert_{L^\infty}+\Vert n_2(\cdot,t)\Vert_{L^\infty}+\Vert\nabla c_1(\cdot,t)\Vert_{L^\infty}+\Vert\nabla c_2(\cdot,t)\Vert_{L^\infty}+\Vert {\bf u}_1\Vert_{{\bf W}^{1,q}}+\Vert {\bf u}_2\Vert_{{\bf W}^{1,q}}\lesssim
1.\end{align*}

Similarly as discussed in \cite{winkler2012global-NS-PKS}, we apply $\nabla\cdot {\bf u}_1=\nabla\cdot {\bf u}_2=0$ and $W^{1,q}\hookrightarrow C^{0}(\bar\Omega)$ for $q>2$ to finally arrive at 
\begin{align}\label{collect1-appendixA}
&\frac{1}{2}\frac{d}{dt}\int_{\Omega}(n_1-n_2)^2\, dx+\frac{1}{2}\int_{\Omega}|\nabla(n_1-n_2)|^2\, dx\nonumber\\
\leq &\frac{1}{2}\int_{\Omega}|\nabla(c_1-c_2)|^2\, dx+{\hat C}_1\int_{\Omega}(n_1-n_2)^2\, dx+{\hat C}_1\int (c_1-c_2)^2\, dx+{\hat C}_1\int_{\Omega} |{\bf u}_1-{\bf u}_2|^2\, dx,
\end{align}
for $t\in(0,T_0)$ with positive constant ${\hat C}_1$.  Proceeding the $c$-equation with the similar argument, one finds
\begin{align}\label{collect2-appendixA}
\frac{\iota}{2}\frac{d}{dt}\int_{\Omega} (c_1-c_2)^2\, dx+\frac{1}{2}\int_{\Omega}(c_1-c_2)^2\, dx+\int_{\Omega}|\nabla(c_1-c_2)|^2\, dx\leq\frac{1}{2}\int_{\Omega}(n_1-n_2)^2\, dx.
\end{align}

For the ${\bf u}$-equation, we test the subtraction of ${\bf u}_1$ and ${\bf u}_2$ equation against ${\bf u}_1-{\bf u}_2$ and integrate it by parts to have
\begin{align}
\frac{1}{2}\int_{\Omega}({\bf u}_1-{\bf u}_2)^2\, dx+\int_{\Omega}|\nabla ({\bf u}_1-{\bf u}_2)|^2\, dx=&-\int_{\Omega} [({\bf u}_1-{\bf u}_2)\cdot \nabla]({\bf u}_1-{\bf u}_2)\cdot {\bf u}_1\, dx\nonumber\\
&+\int_{\Omega}({\bf u}_1\cdot\nabla)({\bf u}_1-{\bf u}_2)({\bf u}_1-{\bf u}_2)\, dx\nonumber\\
&+\int_{\Omega}(n_1-n_2) \nabla c_1({\bf u}_1-{\bf u}_2)\, dx\nonumber\\
&+\int_{\Omega} n_2\nabla(c_1-c_2)\cdot ({\bf u}_1-{\bf u}_2)\, dx\nonumber.
\end{align}
By using the H\"{o}lder's inequality, one further obtains for $t\in(0,T_0),$
\begin{align}\label{collect3-appendixA}
&\frac{1}{2}\frac{d}{dt}\int_{\Omega}|{\bf u}_1-{\bf u}_2|\, dx+\frac{1}{2}\int_{\Omega}|\nabla ({\bf u}_1-{\bf u}_2)|^2\, dx\nonumber\\
\leq & \hat C_2\int_{\Omega}(n_1-n_2)^2\, dx+\frac{1}{2}\int_{\Omega}|\nabla (c_1-c_2)|^2\, dx+\hat C_2\int_{\Omega}|{\bf u}_1-{\bf u}_2|^2\, dx,
\end{align}
where $\hat C_2>0$ is a constant.

Define $y=\int_{\Omega}(n_1-n_2)^2\,dx+\iota \int_{\Omega}(c_1-c_2)^2\,dx+\int_{\Omega}|{\bf u}_1-{\bf u}_2|^2\, dx$, then we collect (\ref{collect1-appendixA}), (\ref{collect2-appendixA}) and (\ref{collect3-appendixA}) to find $y'\leq {\hat C}_3 y$ with constant $\hat C_3>0$, which implies $y\equiv 0$ for $t\in(0,T_0)$.  Then if $\iota\not=0,$ we immediately obtain a contradiction.  If $\iota=0,$ $y\equiv 0$ implies
$$\int_{\Omega} |n_1-n_2|^2\,dx\equiv 0,\ \ \forall t\in(0,T_0).$$
Then it follows from (\ref{collect2-appendixA}) that 
\begin{align*}
\int_{\Omega} |\nabla(c_1-c_2)|^2\,dx+\int_{\Omega} (c_1-c_2)^2\, dx\equiv 0,
\end{align*}
which also reaches a contradiction.  As a consequence, we have (\ref{PKSNS-time-dependent}) admits the unique solution locally under the condition of Theorem \ref{thm-global-existence-PKS-NS}.
\qed

\section*{Acknowledgements}
We thank Prof. T. Tsai for pointing out the useful references and critical discussions.  The research of J. Wei is partially supported by NSERC of Canada.  The research of C. Lai is supported in part by the Simons Foundation Math + X Investigator Award \#376319 (Michael I. Weinstein).
  %The research of C. Lai is partially supported by FYF (\#6456) of Graduate and Postdoctoral Studies, UBC. The research of J. Wei is partially supported by the NSERC grant of Canada.

\bibliographystyle{plain}
\bibliography{ref}

\end{document}